\documentclass[11pt,a4paper]{amsart}

\usepackage{amssymb,amsmath,amsxtra}
\usepackage{hyperref}
\usepackage{xcolor}
\hypersetup{
    colorlinks,
    linkcolor={red!50!black},
    citecolor={blue!50!black},
    urlcolor={blue!80!black}
}
\usepackage[margin=2.4cm]{geometry} 
\usepackage{cite}
\usepackage{graphicx}
\usepackage[curve]{xypic}

\newtheorem{theorem}{Theorem}[section]
\newtheorem{lemma}[theorem]{Lemma}
\newtheorem{corollary}[theorem]{Corollary}
\newtheorem{proposition}[theorem]{Proposition}
\theoremstyle{definition}
\newtheorem{definition}[theorem]{Definition}
\newtheorem{example}[theorem]{Example}
\newtheorem{notation}[theorem]{Notation}
\newtheorem{assumption}[theorem]{Assumption}
\newtheorem{remark}[theorem]{Remark}

\numberwithin{equation}{section}

\newcommand{\cA}{\mathcal{A}}

\DeclareMathOperator{\Aut}{Aut}
\DeclareMathOperator{\ch}{ch}
\newcommand{\orbich}{\widetilde{\ch}}

\newcommand{\Cstar}{\CC^\times}
\newcommand{\cC}{\mathcal{C}}
\newcommand{\CC}{\mathbb{C}}

\newcommand{\ev}{\mathrm{ev}}

\renewcommand{\cH}{\mathcal{H}}
\DeclareMathOperator{\Hom}{Hom}
\newcommand{\tti}{\mathtt{i}}

\DeclareMathOperator{\id}{id}
\DeclareMathOperator{\im}{im}
\newcommand{\KK}{\mathbb{K}}
\newcommand{\KKeff}{\mathbb{K}^{\text{\rm eff}}}
\renewcommand{\cL}{\mathcal{L}}
\newcommand{\LL}{\mathbb{L}}
\DeclareMathOperator{\Lie}{Lie}
\newcommand{\cM}{\mathcal{M}}

\newcommand{\cO}{\mathcal{O}}
\newcommand{\PP}{\mathbb{P}}
\newcommand{\QQ}{\mathbb{Q}}
\newcommand{\RR}{\mathbb{R}}
\newcommand{\FM}{\mathbb{FM}}

\DeclareMathOperator{\Res}{Res}
\DeclareMathOperator{\Sym}{Sym}
\newcommand{\cU}{\mathcal{U}}

\newcommand{\ZZ}{\mathbb{Z}}
\newcommand{\UU}{\mathbb{U}}

\newcommand{\CR}{{\operatorname{CR}}}
\newcommand{\amb}{{\operatorname{amb}}}
\DeclareMathOperator{\cbox}{Box}
\newcommand{\correlator}[1]{\left \langle #1 \right \rangle}
\newcommand{\fun}{\mathbf{1}}
\newcommand{\<}{\left\langle}
\renewcommand{\>}{\right\rangle}
\def\hsymb#1{\mbox{\strut\rlap{\smash{\Huge$#1$}}\quad}}

\newcommand{\II}{\mathbb{I}} 
\newcommand{\bbS}{\mathbb{S}}

\newcommand{\hGamma}{\widehat{\Gamma}} 
\newcommand{\hS}{\widehat{S}}
 
\newcommand{\hcM}{\widehat{\cM}}
\newcommand{\hU}{\widehat{U}} 
\newcommand{\hvarrho}{\hat{\varrho}}
\newcommand{\hR}{\widehat{R}}

\newcommand{\tLL}{\widetilde{\LL}}

\newcommand{\tcM}{\widetilde{\cM}} 
\newcommand{\ttau}{\tilde{\tau}} 
\newcommand{\tUpsilon}{\widetilde{\Upsilon}}
\newcommand{\tcH}{\widetilde{\cH}} 
\newcommand{\tPsi}{\widetilde{\Psi}} 
\newcommand{\tD}{\widetilde{D}} 
\newcommand{\tomega}{{\tilde{\omega}}} 
\newcommand{\tX}{\widetilde{X}} 
\newcommand{\tdelta}{{\tilde{\delta}}} 
\newcommand{\tR}{\widetilde{R}} 
\newcommand{\tS}{\widetilde{S}} 
\newcommand{\tK}{\widetilde{K}} 

\newcommand{\N}{\mathbb{N}}

\newcommand{\Ker}{\operatorname{Ker}} 
\newcommand{\Frac}{\operatorname{Frac}}
\newcommand{\End}{\operatorname{End}} 
\newcommand{\NE}{\operatorname{NE}} 
\newcommand{\inv}{\operatorname{inv}} 
\newcommand{\Spec}{\operatorname{Spec}} 
 
\renewcommand{\ev}{\operatorname{ev}} 
\newcommand{\Gr}{\operatorname{Gr}}
\newcommand{\tch}{\operatorname{\widetilde{\ch}}}
\newcommand{\rank}{\operatorname{rank}} 
\newcommand{\Td}{\operatorname{Td}} 
\newcommand{\tTd}{\widetilde{\Td}}
\newcommand{\Ext}{\operatorname{Ext}} 
\newcommand{\age}{\operatorname{age}}

\newcommand{\ad}{\operatorname{ad}}

\def\parfrac#1#2{{\frac{\partial #1}{\partial #2}}}

\newcommand{\cF}{\mathcal{F}} 
\newcommand{\cE}{\mathcal{E}} 
\newcommand{\calD}{\mathcal{D}} 
\newcommand{\cT}{\mathcal{T}} 

\newcommand{\bt}{\mathbf{t}}
\newcommand{\bN}{\mathbf{N}}
\newcommand{\bL}{\mathbf{L}}
\newcommand{\bE}{\mathbf{E}} 
\newcommand{\bA}{\mathbf{A}} 
\newcommand{\bmu}{\boldsymbol{\mu}}
\newcommand{\bF}{\mathbf{F}} 
\newcommand{\bnabla}{\boldsymbol{\nabla}} 
\newcommand{\bGr}{\operatorname{\mathbf{Gr}}} 

\newcommand{\fbar}{\overline{f}}

\newcommand{\sfB}{\mathsf{B}}

\newcommand{\cQ}{\mathcal{Q}} 

\newcommand{\frs}{\mathfrak{s}} 
\newcommand{\frc}{\mathfrak{c}} 

\newcommand{\ty}{\tilde{y}} 
\newcommand{\tx}{\tilde{x}} 
\newcommand{\sfy}{\mathsf{y}}
\newcommand{\tsfy}{\tilde{\sfy}}  
\newcommand{\sx}{\mathsf{x}} 

\newcommand{\sfp}{\mathsf{p}}

\begin{document}

\title[The Crepant Transformation Conjecture]{The Crepant Transformation Conjecture\\ for Toric Complete Intersections}

\author[Coates]{Tom Coates}
\address{Department of Mathematics\\
Imperial College London\\
180 Queen's Gate\\
London SW7 2AZ 
\\UK}
\email{t.coates@imperial.ac.uk}

\author[Iritani]{Hiroshi Iritani}
\address{Department of Mathematics\\
Graduate School of Science\\
Kyoto University\\
Oiwake-cho\\
Kitashirakawa\\
Sakyo-ku\\
Kyoto, 606-8502\\
Japan}
\email{iritani@math.kyoto-u.ac.jp}

\author[Jiang]{Yunfeng Jiang}
\address{Department of Mathematics\\
University of Kansas\\
1460 Jayhawk Boulevard\\
Lawrence, Kansas 66045-7594\\
USA}
\email{y.jiang@ku.edu}

\subjclass[2010]{14N35 (Primary); 14A20, 14E16, 14F05, 53D45 (Secondary)}

\keywords{Gromov--Witten theory, Crepant Resolution Conjecture, toric Deligne--Mumford stacks, orbifolds, quantum cohomology, mirror symmetry, Fourier--Mukai transformation, flop, $K$-equivalence, Gamma class, integral structure, variation of GIT quotient, Givental's symplectic formalism, GKZ system, Mellin--Barnes method}

\date{}

\begin{abstract}
Let $X$ and $Y$ be $K$-equivalent toric Deligne--Mumford stacks related by a single toric wall-crossing.  We prove the Crepant Transformation Conjecture in this case, fully-equivariantly and in genus zero.  That is, we show that the equivariant quantum connections for $X$ and $Y$ become gauge-equivalent after analytic continuation in quantum parameters.  Furthermore we identify the gauge transformation involved, which can be thought of as a linear symplectomorphism between the Givental spaces for $X$ and $Y$, with a Fourier--Mukai transformation between the $K$-groups of $X$ and $Y$, via an equivariant version of the Gamma-integral structure on quantum cohomology.  We prove similar results for toric complete intersections.  We impose only very weak geometric hypotheses on $X$ and $Y$: they can be non-compact, for example, and need not be weak Fano or have Gorenstein coarse moduli space.  Our main tools are the Mirror Theorems for toric Deligne--Mumford stacks and toric complete intersections, and the Mellin--Barnes method for analytic continuation of hypergeometric functions.
\end{abstract}

\maketitle

\section{Introduction}

A birational map $\varphi \colon X_+ \dashrightarrow X_-$ between smooth varieties, orbifolds, or Deligne--Mumford stacks is called a \emph{$K$-equivalence}
if there exists a smooth variety, orbifold, or Deligne--Mumford stack $\tX$ and projective birational morphisms 
$f_\pm \colon \tX \to X_\pm$ such that $f_- = \varphi \circ f_+$ 
and $f_+^\star K_{X_+} = f_-^\star K_{X_-}$: 
\begin{equation} 
  \label{eq:common_blowup} 
  \begin{aligned}
    \xymatrix{
      & \tX \ar[dr]^{f_-} \ar[dl]_{f_+} & \\ 
      X_+ \ar@{-->}[rr]^{\varphi} & & X_-
    }
  \end{aligned}
\end{equation}
In this case, the celebrated Crepant Transformation Conjecture of Y.~Ruan predicts that the quantum (orbifold) cohomology algebras of $X_+$ and $X_-$ should be related by analytic continuation in the quantum parameters.  This conjecture has stimulated a great deal of interest in the connections between quantum cohomology (or Gromov--Witten theory) and birational geometry: see, for example, \cite{Ruan:crepant,Perroni,Boissiere--Mann--Perroni:1,Boissiere--Mann--Perroni:2,Bryan--Graber--Pandharipande,Bryan--Gholampour:1,Bryan--Gholampour:2,Gillam,Wise,Coates,Lee--Lin--Wang:1,Lee--Lin--Wang:2,Lee--Lin--Wang:ICCM,Iwao--Lee--Lin--Wang,Lee--Lin--Qu--Wang,Chen--Li--Li--Zhao,Chen--Li--Zhang--Zhao,Zhou,Li--Ruan,Gonzalez--Woodward}.  Ruan's original conjecture was subsequently refined, revised, and extended to higher genus Gromov--Witten invariants, first by Bryan--Graber~\cite{Bryan--Graber} under some additional hypotheses, and then by Coates--Iritani--Tseng, Iritani, and Ruan in general~\cite{CIT,Iritani,Coates--Ruan}.  Recall that a toric Deligne--Mumford stack $X$ can be constructed as a GIT quotient $\big[\CC^m /\!\!/_\omega K\big]$ of $\CC^m$ by an action of a complex torus $K$, where $\omega$ is an appropriate \emph{stability condition}, and that wall-crossing in the space of stability conditions induces birational transformations between GIT quotients~\cite{Dolgachev--Hu,Thaddeus}.  Our main result implies the CIT/Ruan version of the Crepant Transformation Conjecture in genus zero, in the case where $X_+$ and $X_-$ are complete intersections in toric Deligne--Mumford stacks and $\varphi \colon X_+ \dashrightarrow X_-$ arises from a toric wall-crossing.  We concentrate initially on the case where $X_+$ and $X_-$ are toric, deferring the discussion of toric complete intersections to~\S\ref{sec:introduction_ci}.

\subsection{The Toric Case}
\label{sec:toric_case}
We consider toric Deligne--Mumford stacks $X_\pm$ of the form $\big[\CC^m /\!\!/_\omega K \big]$, where $K$ is a complex torus, and consider a $K$-equivalence $\varphi \colon X_+ \dashrightarrow X_-$ determined by a wall-crossing in the space of stability conditions~$\omega$.  The action of $T=(\Cstar)^m$ on $\CC^m$ descends to give (ineffective) actions of $T$ on $X_\pm$, and we consider the $T$-equivariant Chen--Ruan cohomology groups $H_{\CR,T}^\bullet(X_\pm)$~\cite{Chen--Ruan:orbifold_GW}.  There is a $T$-equivariant big quantum product $\star_\tau$ on $H_{\CR,T}^\bullet(X_\pm)$, parametrized by $\tau \in H_{\CR,T}^\bullet(X_\pm)$ and defined in terms of $T$-equivariant Gromov--Witten invariants of $X_\pm$.  The \emph{$T$-equivariant quantum connection} is a pencil of flat connections:
\begin{equation} 
\label{eq:qconn} 
\nabla = d + z^{-1} \sum_{i=0}^N (\phi_i \star_\tau) d\tau^i
\end{equation}
on the trivial $H_{\CR,T}^\bullet(X_\pm)$-bundle over an open set in $H_{\CR,T}^\bullet(X_\pm)$; here $z \in \Cstar$ is the pencil variable, $\tau \in H_{\CR,T}^\bullet(X_\pm)$ is the co-ordinate on the base of the bundle, $\phi_0,\dots,\phi_N$ are a basis for $H_{\CR,T}^\bullet(X_\pm)$, and $\tau^0,\dots, \tau^N$ are the corresponding co-ordinates of $\tau \in H_{\CR,T}^\bullet(X_\pm)$, so that $\tau = \sum_{i=0}^N \tau^i \phi_i$.

\begin{theorem}
\label{thm:main_theorem} 

Let $X_+$ and $X_-$ be toric Deligne--Mumford stacks, and let $\varphi \colon X_+ \dashrightarrow X_-$ be a $K$-equivalence that arises from a wall-crossing of GIT stability conditions.  Then:
\begin{enumerate}
\item the equivariant quantum connections of $X_\pm$ become gauge-equivalent after analytic continuation in $\tau$, via a gauge transformation $\Theta(\tau,z) \colon H^\bullet_{\CR,T}(X_-) \to H^\bullet_{\CR,T}(X_+)$ which is homogeneous of degree zero, regular at $z=0$, and preserves the equivariant orbifold Poincar\'{e} pairing.
\item there exists a common toric blowup $\tX$ of $X_\pm$ as in \eqref{eq:common_blowup} such that gauge transformation $\Theta$ coincides with the Fourier--Mukai transformation
\begin{align*}
  \FM \colon K^0_T(X_-) \to K^0_T(X_+) && E\mapsto (f_+)_\star (f_-)^\star (E)  
\end{align*}
via the equivariant Gamma-integral structure introduced in \S\ref{sec:integral_structure} below.
\end{enumerate}
\end{theorem} 
\noindent Here:
\begin{itemize}
\item The Gamma-integral structure on equivariant quantum cohomology is an assignment, to each class $E \in K^0_T(X_\pm)$ of $T$-equivariant vector bundles on $X_\pm$, of a flat section $\frs(E)$ for the equivariant quantum connection on $X_\pm$.  This gives a 
 lattice in the space of flat sections which is isomorphic to the integral equivariant $K$-group $K^0_T(X_\pm)$.  The flat section $\frs(E)$ is, roughly speaking, given by the Chern character of $E$ multiplied by a characteristic class of $X_\pm$, called the $\hGamma$-class, that is defined in terms of the $\Gamma$-function.  Part~(2) of Theorem~\ref{thm:main_theorem} asserts that the flat section $\frs(E)$ analytically continues to $\frs(\FM(E))$.
\item The gauge transformation $\Theta(\tau,z)$ will in general be non-constant: it depends on the parameter $\tau$ for the equivariant quantum product, and also on the parameter $z$ appearing in the equivariant quantum connection.  When written in terms of the integral structure, however, it becomes a constant, integral linear transformation. 
\end{itemize}

\begin{remark}
  Throughout this paper, when we consider $K$-equivalence \eqref{eq:common_blowup} of Deligne--Mumford stacks $X_\pm$, $K_{X_\pm}$ 
  means the canonical class as a stack; in general this is different from the ($\QQ$-Cartier) 
  canonical divisor $K_{|X_\pm|}$ of the coarse moduli space $|X_\pm|$.  In particular, we do not require the coarse moduli spaces $|X_\pm|$ to be Gorenstein.
\end{remark}

\begin{remark} 
Gonzalez and Woodward \cite{Gonzalez--Woodward} have proved 
a very general wall-crossing formula for Gromov--Witten invariants under  variation of GIT quotient, using gauged Gromov--Witten theory.
Their result, which is a quantum version of Kalkman's wall-crossing formula, gives a complete description of how non-equivariant 
genus-zero Gromov--Witten invariants change under wall-crossing.  Thus their theorem must imply the non-equivariant version of the first part of Theorem~\ref{thm:main_theorem}, and the first part of Theorem~\ref{thm:main_theorem_ci}.
Our methods are significantly less general --- they apply only to toric stacks and toric complete intersections --- but give a much more explicit relationship between the genus-zero 
Gromov--Witten theories. 
\end{remark} 

\noindent Theorem~\ref{thm:main_theorem} is slightly imprecisely stated:  
we give precise statements, once the necessary notation and definitions are in place, 
as Theorems~\ref{thm:global_qconn},~\ref{thm:U}, 
and~\ref{thm:CTC_qconn} below.  
We now explain how Theorem~\ref{thm:main_theorem} implies 
the CIT/Ruan version of the Crepant Transformation Conjecture.

\bigskip

The CIT/Ruan version of the Crepant Transformation Conjecture is stated in terms of Givental's symplectic formalism for Gromov--Witten theory~\cite{Givental:symplectic}.  In our context, this associates to $X_\pm$ the vector spaces $\cH(X_\pm) := H^\bullet_{\CR,T}(X_\pm)(\!(z^{-1})\!)$ equipped with a certain symplectic form, and encodes $T$-equivariant genus-zero Gromov--Witten invariants via a Lagrangian cone $\cL_\pm \subset \cH(X_\pm)$.  The Givental cone $\cL_\pm$ for $X_\pm$ determines the big quantum product $\star_\tau$ on $H^\bullet_{\CR,T}(X_\pm)$, and \emph{vice versa}.  The CIT/Ruan Crepant Transformation Conjecture, made in the context of non-equivariant Gromov--Witten theory, asserts that there exists a $\CC(\!(z^{-1})\!)$-linear grading-preserving symplectic isomorphism $\UU \colon \cH(X_-) \to \cH(X_+)$, such that after analytic continuation of $\cL_\pm$ we have $\UU(\cL_-) = \cL_+$.  See~\cite{CIT,Coates--Ruan} for more details.

There are various subtle points in the notion of analytic continuation of the (infinite-dimensional) cones $\cL_\pm$, especially under the weak convergence hypotheses that we impose, and some necessary foundational material is missing.  Thus we choose to state Theorem~\ref{thm:main_theorem} in terms of the equivariant quantum connections for $X_\pm$ rather than in terms of the Givental cones $\cL_\pm$.  The two formulations are very closely related, however, as we now explain.  Let $L_\pm(\tau,z)$ denote a fundamental solution for the equivariant quantum connection $\nabla$, that is, a matrix with columns that give a basis of flat sections for $\nabla$.  The assignment
\begin{align*}
  \tau \mapsto L_\pm(\tau,z)^{-1} \cH_+ && \tau \in H^\bullet_{\CR,T}(X_\pm) && \text{where $\cH_+:= H^\bullet_{\CR,T}(X_\pm) \otimes \CC[z]$}
\end{align*}
gives the family of tangent spaces to the Givental cone $\cL_\pm$.  As emphasized in~\cite{CIT}, this defines a variation of semi-infinite Hodge structure in the sense of Barannikov~\cite{Barannikov:qperiods}.  The Givental cone $\cL_\pm$ can be reconstructed from the semi-infinite variation as:
\[
\cL_\pm = \bigcup_{\tau} z L_\pm(\tau,z)^{-1}  \cH_+
\]
Thus part~(1) of Theorem~\ref{thm:main_theorem} implies the CIT/Ruan-style Crepant Transformation Conjecture whenever it makes sense, with the symplectic transformation $\UU$ defined in terms of the gauge transformation $\Theta$ by $\UU = L_+^{-1} \Theta L_-$.  The fact that $\UU$ is independent of $\tau$ follows from the fact that $\Theta$ is a gauge equivalence.  The fact that $\UU$ is symplectic (or equivalently, the fact that $\Theta$ is pairing-preserving) follows from the identification, in part~(2) of Theorem~\ref{thm:main_theorem}, of $\Theta$ with the Fourier--Mukai transformation~$\FM$.  The Fourier--Mukai transformation is a derived equivalence and thus preserves the Mukai pairings on $K^0_T(X_\pm)$; this implies, via the equivariant Hirzebruch--Riemann--Roch theorem, that $\Theta$ is pairing-preserving.  The identification of $\Theta$ with $\FM$ also makes clear that the symplectic transformation $\UU$ has a well-defined non-equivariant limit, since the Fourier--Mukai transformation itself can be defined non-equivariantly.

In terms of the symplectic transformation $\UU$, part (2) of 
Theorem \ref{thm:main_theorem} can be rephrased 
as the commutativity of the diagram
\[
\xymatrix{
  K_T^0(X_-) \ar[r]^{\FM} \ar[d]_{\tPsi_-} &  K_T^0(X_+) \ar[d]^{\tPsi_+} \\
  \tcH(X_-) \ar[r]^{\UU} & \tcH(X_+) 
}
\]
where $\tcH(X_\pm)$ is a variant of Givental's symplectic space and $\tPsi_\pm$ are certain `framing maps' built from the Gamma-integral structure: see Theorem~\ref{thm:U}.  This identification of $\UU$ with a Fourier--Mukai transformation was proposed in \cite{Iritani}.  Our results also imply Ruan's original conjecture that the quantum cohomology rings of $X_\pm$ are (abstractly) isomorphic, and that the associated $F$-manifold structures are isomorphic. We refer the reader to \cite{CCIT:computing, CIT, Coates--Ruan, Coates, Iritani:Ruan} for discussions on the consequence of these conjectures and several concrete examples.

\subsection{The Mellin--Barnes Method and the Work of Borisov--Horja} 

The main ingredients in the proof of Theorem~\ref{thm:main_theorem} are the Mirror Theorem for toric stacks~\cite{CCIT,Cheong--Ciocan-Fontanine--Kim}, which determines the equivariant quantum connection $\nabla$ (or, equivalently, the Givental cone $\cL_\pm)$ in terms of a certain cohomology-valued hypergeometric function called the \emph{$I$-function}, and the Mellin--Barnes method~\cite{Barnes,CdOGP}, which allows us to analytically continue the $I$-functions for $X_\pm$.  From this point of view, the symplectic transformation $\UU$ arises as the matrix which intertwines the two $I$-functions (see Theorem~\ref{thm:U}):
\[ 
\UU I_-  = I_+. 
\]
On the other hand, components of the $I$-function 
give hypergeometric solutions to the Gelfand--Kapranov--Zelevinsky (GKZ) system of 
differential equations. The analytic continuation of 
solutions to the GKZ system has been studied by Borisov--Horja 
\cite{Borisov--Horja:FM}. 
They showed that, under an appropriate identification of the 
spaces of GKZ solutions with the $K$-groups of the corresponding 
toric Deligne--Mumford stacks, the analytic continuation of 
solutions to a GKZ system is induced by a Fourier--Mukai 
transformation between the $K$-groups. 
Our computation may be viewed as a straightforward 
generalization of theirs. The differences from their situation 
are:
\begin{itemize} 
\item[(a)] we work with a fully equivariant version, that is, 
the parameters $\beta_j$ in the GKZ system are arbitrary 
and we use the equivariant $K$-groups (here 
$\beta_j$ corresponds to the equivariant parameter); 

\item[(b)] we compute analytic continuation of the $I$-function 
corresponding to the \emph{big} quantum cohomology; 
in terms of the GKZ system, we do not assume that 
lattice vectors in the set\footnote
{Recall that Gelfand--Kapranov--Zelevinsky defined the GKZ system 
in terms of a finite set $A \subset \ZZ^d$. 
They called it the $A$-hypergeometric system.} 
$A$ lie on a hyperplane of height one. 
\end{itemize} 

\noindent Since we work equivariantly, we can use the fixed point basis in localized equivariant cohomology to calculate the analytic continuation of the $I$-functions.  It turns out that analytic continuation via the Mellin--Barnes method becomes much easier to handle in the fully equivariant setting, because we only need to evaluate residues at \emph{simple} poles\footnote{For an example of the complexities caused by non-simple poles, see the orbifold flop calculation in~\cite[\S7]{Coates}.}.  It is also straightforward to compute the Fourier--Mukai transformation in terms of the fixed point basis in the localized equivariant $K$-group, and hence to see that analytic continuation coincides with Fourier--Mukai.

Regarding part (b) above, we choose $A$ to be the set $\{b_1,\dots,b_m\} \subset \bN$ of ray vectors of an extended stacky fan~\cite{Borisov--Chen--Smith, Jiang}.  Since we do not restrict ourselves to the weak Fano case, and since we work with Jiang's extended stacky fans, the generic rank of the GKZ system can be bigger than the rank of $H^\bullet_{\CR,T}(X_\pm)$. To remedy this, we treat one special variable analytically and work formally in the other variables.  In fact, the big $I$-functions are not necessarily convergent in all of the variables, and we analytically continue the $I$-function with respect to one specific variable $\sfy_r$.  This amounts to considering an adic completion of the Borisov--Horja  better-behaved GKZ system~\cite{Borisov--Horja:GKZ} with respect to the other variables.  The analytic continuation in Theorem~\ref{thm:main_theorem} occurs across a ``global K\"{a}hler moduli space'' $\tcM^\circ$ which is treated as an analytic space in one direction and as a formal scheme in the other directions.

\subsection{The Toric Complete Intersection Case}
\label{sec:introduction_ci}

Let $\varphi \colon X_+ \dashrightarrow X_-$ be a $K$-equivalence between toric Deligne--Mumford stacks that arises from a toric wall-crossing, as in \S\ref{sec:toric_case}.  Let $\widetilde{X}$ be the common toric blow-up of $X_\pm$ and let $\overline{X}_0$ denote the common blow-down; $\overline{X}_0$ here is a (singular) toric variety, not a stack.
\[
\xymatrix{
  & \tX \ar[dr]^{f_-} \ar[dl]_{f_+} & \\ 
  X_+ \ar@{-->}[rr]^{\varphi} \ar[dr]_{g_+} & & X_- \ar[dl]^{g_-} \\
  & \overline{X}_0
}
\]
Consider a direct sum of semiample line bundles $E_0 \to \overline{X}_0$, and pull this back to give vector bundles $E_+ \to X_+$, $\widetilde{E} \to \widetilde{X}$, and $E_- \to X_-$.  Let $s_+$,~$\tilde{s}$, and~$s_-$ be sections of, respectively, $E_+$,~$\widetilde{E}$, and~$E_-$ that are compatible via $f_+$ and $f_-$ (so $f_+^\star s_+ = \widetilde{s} = f_-^\star s_-$) such that the zero loci of $s_\pm$ intersect the flopping locus of $\varphi$ transversely.  Let $Y_+$,~$\widetilde{Y}$, and $Y_-$ denote the substacks defined by the zero loci of, respectively, $s_+$,~$\tilde{s}$, and $s_-$.  In this situation there is a commutative diagram:
\begin{equation}
  \label{eq:Y_and_X}
  \begin{aligned}
    \xymatrix{
      & \widetilde{Y} \ar[ld]_-{F_-} \ar[rd]^-{F_+} \ar[d]_-{\tilde{\iota}} \\
      Y_-\ar[d]_{\iota_-}  &\widetilde{X} \ar[ld]^-{f_-} \ar[rd]_-{f_+} & Y_+ \ar[d]^{\iota_+}\\
      X_-  && X_+
    }
  \end{aligned}
\end{equation}
where the vertical maps are inclusions, the bottom triangle is \eqref{eq:common_blowup}, and the squares are Cartesian.  The $K$-equivalence $\varphi \colon X_+ \dashrightarrow X_-$ induces a $K$-equivalence $\varphi \colon Y_+ \dashrightarrow Y_-$.  We now consider the Crepant Transformation Conjecture for $\varphi \colon Y_+ \dashrightarrow Y_-$.

Since the complete intersections $Y_\pm$ will not in general be $T$-invariant we consider non-equivariant Gromov--Witten invariants and the non-equivariant quantum product.  (Our assumptions on $X_\pm$ ensure that the non-equivariant theory makes sense.)  Denote by $H^\bullet_\amb(Y_\pm)$ the image $\im \iota^\star_\pm \subset H_{\CR}^\bullet(Y_\pm)$, where $\iota_\pm \colon Y_\pm \to X_\pm$ is the inclusion map.  If $\tau \in H_{\amb}^\bullet(Y_\pm)$ then the big quantum product $\star_\tau$ preserves the ambient part $H_{\amb}^\bullet(Y_\pm) \subset H_{\CR}^\bullet(Y_\pm)$.  We can therefore define a quantum connection on the ambient part:
\[
\nabla = d + z^{-1} \sum_{i=0}^N (\phi_i \star_\tau) d\tau^i
\]
This is a pencil of flat connections on the trivial $H_{\amb}^\bullet(Y_\pm)$-bundle over an open set in $H_{\amb}^\bullet(Y_\pm)$ where, as in \eqref{eq:qconn}, $z \in \Cstar$ is the pencil variable, $\tau \in H_{\amb}^\bullet(Y_\pm)$ is the co-ordinate on the base of the bundle, $\phi_0,\dots,\phi_N$ are a basis for $H_{\amb}^\bullet(Y_\pm)$, and $\tau^0,\dots, \tau^N$ are the corresponding co-ordinates of~$\tau$.

In \S\ref{sec:ambient_part} below we construct an ambient version of the Gamma-integral structure, which is an assignment to each class $E$ in the ambient part of $K$-theory
\[
K^0_\amb(Y_\pm) = \im \iota_\pm^\star \subset K^0(Y_\pm)
\]
of a flat section $\frs(E)$ for the quantum connection on the ambient part $H_{\amb}^\bullet(Y_\pm)$.  This gives a lattice in the space of flat sections which is isomorphic to the ambient part of (integral) $K$-theory $K^0_\amb(Y_\pm)$.  
\begin{theorem}
\label{thm:main_theorem_ci} 

Let $\varphi \colon Y_+ \dashrightarrow Y_-$ be a $K$-equivalence between toric complete intersections as above.  Then:
\begin{enumerate}
\item the quantum connections on the ambient parts $H_{\amb}^\bullet(Y_\pm) \subset H_{\CR}^\bullet(Y_\pm)$ become gauge-equivalent after analytic continuation in $\tau$, via a gauge transformation
\[
\Theta_Y(\tau,z) \colon H^\bullet_{\amb}(Y_-) \to H^\bullet_{\amb}(Y_+)
\]
which is homogeneous of degree zero and regular at $z=0$.  If $Y$ is compact then $\Theta_Y$ preserves the orbifold Poincar\'{e} pairing.
\item when expressed in terms of the ambient integral structure, the gauge transformation $\Theta_Y$ coincides with the Fourier--Mukai transformation
\begin{align*}
  \FM \colon K^0_\amb(Y_-) \to K^0_\amb(Y_+) && E\mapsto (F_+)_\star (F_-)^\star (E)  
\end{align*}
given by the top triangle in \eqref{eq:Y_and_X}.
\end{enumerate}
\end{theorem} 

\noindent As before, Theorem~\ref{thm:main_theorem_ci} is slightly imprecisely stated:  precise statements can be found as Theorems~\ref{thm:global_qconn_ci},~\ref{thm:U_ci}, and~\ref{thm:CTC_qconn_ci} below.  Arguing as in~\S\ref{sec:toric_case} shows that Theorem~\ref{thm:main_theorem_ci} implies the CIT/Ruan version of the Crepant Transformation Conjecture for $\varphi \colon Y_+ \dashrightarrow Y_-$ whenever it makes sense, with the corresponding map 
\[
\UU_Y \colon \cH_\amb(Y_-) \to \cH_\amb(Y_+)
\]
between the ambient parts of the Givental spaces for $Y_\pm$ being given by:
\[
\UU_Y = (L_+^\amb)^{-1} \Theta_Y L_-^\amb
\]
where $L_\pm^\amb$ are the fundamental solutions for the quantum connections on the ambient parts $H_\amb^\bullet(Y_\pm)$.   

The proof of Theorem~\ref{thm:main_theorem_ci} relies on the Mirror Theorem for toric complete intersections~\cite{CCIT:applications}, and on \emph{non-linear Serre duality}~\cite{Givental:equivariant,Givental:elliptic,Coates--Givental,Tseng}, which relates the quantum cohomology of $Y_\pm$ to the quantum cohomology of the total space of the dual bundles $E_\pm^\vee$.  Since $E_\pm^\vee$ is toric, it can be analyzed using Theorem~\ref{thm:main_theorem}.  

\begin{remark}
  The idea of using non-linear Serre duality to analyze wall-crossing has been developed independently by Lee--Priddis--Shoemaker~\cite{Lee--Priddis--Shoemaker}, in the context of the Landau--Ginzburg/ Calabi--Yau correspondence.
\end{remark}

\begin{example}
  A mirror $Y$ to the quintic $3$-fold arises~\cite{Greene--Plesser,CdOGP,Batyrev} as a crepant resolution of an anticanonical hypersurface in $X = \big[ \PP^4 / (\ZZ/5\ZZ)^3 \big]$.  A mirror theorem for $Y$ has been proved by Lee--Shoemaker~\cite{Lee--Shoemaker}.  The variety $Y$ is a Calabi--Yau $3$-fold with $h^{1,1}(Y) = 101$.  There are many birational models of $Y$ as toric hypersurfaces, corresponding to the many different lattice triangulations of the boundary of the fan polytope for $X$. Theorem~\ref{thm:main_theorem_ci} implies that the quantum connections (and quantum cohomology algebras) of all of these birational models become isomorphic after analytic continuation over the K\"ahler moduli space (which is $101$-dimensional), and that the isomorphisms involved arise from Fourier--Mukai transformations.
\end{example}

\subsection{A Note on Hypotheses}

Since we work with $T$-equivariant Gromov--Witten invariants of the toric Deligne--Mumford stacks $X_\pm$, we  
do not need to assume that the coarse moduli spaces $|X_\pm|$ 
of $X_\pm$ are projective.   We insist instead that $|X_\pm|$ is semi-projective, 
i.e.~that $|X_\pm|$ is projective over the affinization 
$\Spec(H^0(|X_\pm|,\cO))$, and also that $X_\pm$ contains 
at least one torus fixed point.   These conditions are equivalent to demanding that $X_\pm$ is 
obtained as the GIT quotient $\big[\CC^m /\!\!/_\omega K\big]$ of a vector space by the linear action of a complex torus $K$; they ensure that the equivariant quantum cohomology of $X_\pm$
admits a non-equivariant limit.  In particular, therefore,  the non-equivariant version of the Crepant Transformation Conjecture
follows automatically from Theorem~\ref{thm:main_theorem}. 

We do not assume, either, that the stacks $X_\pm$ or $Y_\pm$ satisfy any sort of positivity or weak Fano condition; put differently, we do not impose any additional convergence hypotheses on the $I$-functions for $X_\pm$ and $Y_\pm$.  This extra generality is possible because of our hybrid formal/analytic approach, where we single out one variable $\sfy_r$ and analytically continue in that variable alone. The same technique allows us to describe the analytic continuation of \emph{big} quantum cohomology (or its ambient part), as opposed to small quantum cohomology.  In general, obtaining convergence results for big quantum cohomology is hard.

\subsection{The Hemisphere Partition Function} 

Recently there was some progress in physics in the exact computation of hemisphere partition functions for gauged linear sigma models.  Hori--Romo~\cite{Hori--Romo} explained why the Mellin--Barnes analytic continuation of hemisphere partition functions should be compatible with brane transportation~\cite{Herbst--Hori--Page} in the B-brane category.  In the language of this paper, the hemisphere partition function corresponds to a component of the $K$-theoretic flat section $\frs(E)$, and brane transportation corresponds to the Fourier--Mukai transformation.  Theorem~\ref{thm:main_theorem} thus confirms the result of Hori--Romo.  Note that the relevant equivalence between B-brane categories should depend on a choice of a path of analytic continuation, and that the Fourier--Mukai transformation in Theorem~\ref{thm:main_theorem} corresponds to a specific choice of path (see Figure~\ref{fig:path_ancont}).

\subsection{Plan of the Paper} 

We fix notation for equivariant Gromov--Witten invariants and equivariant quantum cohomology in~\S\ref{sec:equivariant_qc}, and introduce the equivariant Gamma-integral structure in~\S\ref{sec:integral_structure}.  We establish notation for toric Deligne--Mumford stacks in~\S\ref{sec:notation}.  In \S\ref{sec:wall-crossing} we study $K$-equivalences $\varphi \colon X_+ \dashrightarrow X_-$ of toric Deligne--Mumford stacks arising from wall-crossing, constructing global versions of the equivariant quantum connections for $X_\pm$.  We prove the Crepant Transformation Conjecture for toric Deligne--Mumford stacks (Theorem~\ref{thm:main_theorem}) in~\S\ref{sec:CRC}, and the Crepant Transformation Conjecture for toric complete intersections (Theorem~\ref{thm:main_theorem_ci}) in~\S\ref{sec:CRC_ci}. 

\subsection{Notation} \label{sec:standing_notation} 
We use the following notation throughout the paper. 

\begin{itemize}
\item $X$ denotes a general smooth Deligne--Mumford stack in~\S\ref{sec:equivariant_qc} and~\S\ref{sec:integral_structure}; it denotes a smooth toric Deligne--Mumford stack in~\S\ref{sec:notation} and later.  
\item $T = (\Cstar)^m$. 
\item $R_T = H^\bullet_T({\rm pt},\CC)$.
\item $\lambda_j \in H^2_T({\rm pt},\CC) = \Lie(T)^\star$ is the character of $T = (\Cstar)^m$ given by projection to the $j$th factor, so that $R_T = \CC[\lambda_1,\ldots,\lambda_m]$.
\item $S_T$ is the localization of $R_T$ with respect to the set of non-zero homogeneous elements.
\item $\ZZ[T] = K^\bullet_T({\rm pt})$, so that $\ZZ[T] = \ZZ[e^{\pm\lambda_1},\ldots,e^{\pm\lambda_m}]$.
\item $\bmu_l = \{ z \in \Cstar : z^l = 1\}$ is a cyclic group of order $l$.
\end{itemize} 

\section{Equivariant Quantum Cohomology} 
\label{sec:equivariant_qc} 

In this section we establish notation for various objects in equivariant Gromov--Witten theory.  We introduce equivariant Chen--Ruan cohomology in~\S\ref{sec:CR}, equivariant Gromov--Witten invariants in~\S\ref{sec:GW},  equivariant quantum cohomology in~\S\ref{sec:QC}, Givental's symplectic formalism in~\S\ref{sec:Givental_cone},  and the equivariant quantum connection in~\S\ref{sec:quantum_connection}. 

\subsection{Smooth Deligne--Mumford stacks with Torus Action}
\label{sec:conditions}

Let $X$ be a smooth Deligne--Mumford stack of finite type over $\CC$ 
equipped with an action of an algebraic torus $T \cong (\CC^\times)^m$. 
Let $|X|$ denote the coarse moduli space of $X$ and 
let $IX$ denote the inertia stack $X \times_{|X|} X$ of $X$: a 
point on $IX$ is given by a pair $(x,g)$ with $x\in X$ 
and $g\in \Aut(x)$. We write 
\[
IX = \bigsqcup_{v\in \sfB} X_v 
\]
for the decomposition of $IX$ into connected components.  We assume the following conditions: 
\begin{enumerate} 
\item the coarse moduli space $|X|$ is semi-projective, i.e.~is projective over the affinization $\Spec H^0(|X|,\cO) = \Spec H^0(X,\cO)$; \label{condition:1}
\item all the $T$-weights appearing in the $T$-representation $H^0(X,\cO)$ are contained in a strictly convex cone in $\Lie(T)^*$, and the $T$-invariant subspace $H^0(X,\cO)^T$ is $\CC$; \label{condition:2}
\item the inertia stack $IX$ is equivariantly formal, that is, the $T$-equivariant cohomology $H_T^\bullet(IX;\CC)$ is a free module over $R_T := H^\bullet_T({\rm pt};\CC)$ and one has a (non-canonical) isomorphism of $R_T$-modules
$H_{T}^\bullet(IX;\CC) \cong  H^\bullet(IX;\CC) \otimes_\CC R_T$. \label{condition:3}
\end{enumerate}
These conditions allow us to define Gromov--Witten invariants of $X$ and also the equivariant (Dolbeault) index of coherent sheaves on $X$.  The first and second conditions together imply that the fixed set $X^T$ is compact.  The third condition seems to be closely related to the first two, but it implies for example the localization of equivariant cohomology: the restriction $H^\bullet_T(IX;\CC) \to H^\bullet_T(IX^T;\CC)$ to the $T$-fixed locus is injective and becomes an isomorphism after localization (see \cite{GKM}).  Later we shall restrict to the case where $X$ is a toric Deligne--Mumford stack, where conditions (\ref{condition:1}--\ref{condition:3}) automatically hold, but the definitions in this section make sense for general $X$ satisfying these conditions.

\subsection{Equivariant Chen--Ruan Cohomology} 
\label{sec:CR}
Let $H_{\CR,T}^\bullet(X)$ denote the even part of the $T$-equivariant orbifold cohomology group of Chen and Ruan.   It is defined as the even degree part of the $T$-equivariant cohomology 
\[
H_{\CR,T}^k(X) = \bigoplus_{v \in \sfB : k- 2 \iota_v \in  2\ZZ} 
H_T^{k-2\iota_v}(X_v;\CC) 
\]
of the inertia stack $IX$.  The grading of $H_{\CR,T}^\bullet(X)$ is shifted from that of $H_{T}^\bullet(IX)$ by the so-called \emph{age} or \emph{degree shifting number} $\iota_v\in \QQ$ \cite{Chen--Ruan:orbifold}; note that we consider only the even degree classes in $H_T^\bullet(IX)$. (For toric stacks, all cohomology classes on $IX$ are of even degree.) Equivariant formality of $IX$ gives that $H^\bullet_{\CR,T}(X)$ is a free module over $R_T$.  We write 
\begin{align*}
  (\alpha,\beta)  = \int_{IX} \alpha \cup \inv^* \beta, && \alpha,\beta \in H_{\CR,T}^\bullet(X) 
\end{align*}
for the equivariant orbifold Poincar\'{e} pairing: here $\inv \colon IX \to IX$ denotes the involution on the inertia stack $IX$ that sends a point $(x,g)$ with $x\in X$, $g\in \Aut(x)$ to $(x,g^{-1})$.  Since $X$ is not necessarily proper, the equivariant integral on the right-hand side here is defined via the Atiyah--Bott localization formula~\cite{Atiyah--Bott} and takes values in the localization $S_T$ of $R_T$ with respect to the multiplicative set of non-zero homogeneous elements\footnote{Note that $R_T \subsetneq S_T \subsetneq \Frac(R_T)$; we use $S_T$ instead of $\Frac(R_T)$ since we need a grading on $S_T$ later.} in $R_T$.  

\subsection{Equivariant Gromov--Witten Invariants} 
\label{sec:GW}

Let $X_{g,n,d}$ denote the moduli space of degree-$d$ stable maps to $X$ from genus $g$ orbifold curves with $n$ marked points \cite{Abramovich--Graber--Vistoli:1,Abramovich--Graber--Vistoli:2}; here $d\in H_2(|X|;\ZZ)$.  The moduli space carries a $T$-action and a virtual fundamental cycle  $[X_{g,n,d}]^{\rm vir} \in A_{\bullet,T}(X_{g,n,d};\QQ)$.  There are $T$-equivariant evaluation maps $\ev_i \colon X_{g,n,d} \to \overline{IX}$, $1 \leq i \leq n$, to the rigidified inertia stack $\overline{IX}$ (see \cite{Abramovich--Graber--Vistoli:2}).  Let $\psi_i\in H^2_T(X_{g,n,d})$ denote the psi-class at the $i$th marked point, i.e.~the equivariant first Chern class of the $i$th universal cotangent line bundle $L_i \to X_{g,n,d}$. For $\alpha_1,\dots,\alpha_n \in H_{\CR, T}^\bullet(X)$ and non-negative integers $k_1,\dots,k_n$, 
the \emph{$T$-equivariant Gromov--Witten invariant}  is defined to be: 
\begin{equation}
  \label{eq:Gromov--Witten}
  \correlator{\alpha_1 \psi^{k_1},\dots,\alpha_n \psi^{k_n}}_{g,n,d}^X 
  = \int_{[X_{g,n,d}]^{\rm vir}} 
  \prod_{i=1}^n (\ev_i^* \alpha_i) \psi_i^{k_i}  
\end{equation}
where we regard $\alpha_i$ as a class in $H_T^\bullet(\overline{IX})$ via the canonical isomorphism $H_T^\bullet(\overline{IX})  \cong H_T^\bullet(IX)$.  The moduli space here is not necessarily proper: the right-hand side is again defined via the Atiyah--Bott localization formula and so belongs to $S_T$. Conditions~\eqref{condition:1} and~\eqref{condition:2} in~\S\ref{sec:conditions} ensure that the $T$-fixed locus $X_{g,n,d}^T$ in the moduli space is compact, and thus that the right-hand side of \eqref{eq:Gromov--Witten} is well-defined.

\subsection{Equivariant Quantum Cohomology} 
\label{sec:QC}
Consider the cone $\NE(X)\subset H_2(|X|,\RR)$ generated by classes of effective curves and set $\NE(X)_\ZZ := \{ d \in H_2(|X|,\ZZ) : d \in \NE(X)\}$.  For a ring $R$, define $R[\![Q]\!]$ to be the ring of formal power series with coefficients in $R$: 
\[
R[\![Q]\!] = \left\{ \sum_{d \in \NE(X)_\ZZ} a_d Q^d : a_d \in R \right\}
\]
so that $Q$ is a so-called \emph{Novikov variable}~\cite[III~5.2.1]{Manin}.  Let $\phi_0,\phi_1,\dots,\phi_N$ be a homogeneous basis for $H_{\CR,T}^\bullet(X)$ over $R_T$ and let $\tau^0,\tau^1,\dots,\tau^N$ be the corresponding linear co-ordinates.  We assume that $\phi_0 = 1$ and $\phi_1,\dots,\phi_r \in H^2_T(X)$ are degree-two untwisted classes that induce a $\CC$-basis of $H^2(X;\CC) \cong H^2_T(X)/H^2_T({\rm pt})$.  We write $\tau = \sum_{i=0}^N \tau^i \phi_i$ for a general element of $H_{\CR,T}^\bullet(X)$.  The \emph{equivariant quantum product} $\star_\tau$ at $\tau \in H_{\CR,T}^\bullet(X)$ is defined by the formula 
\[
(\phi_i \star_\tau \phi_j, \phi_k) = \sum_{d \in \NE(X)_\ZZ} \sum_{n=0}^\infty \frac{Q^d}{n!}  \correlator{\phi_i,\phi_j,\phi_k, \tau,\dots,\tau}_{0,n+3,d}^X 
\] 
or, equivalently, by 
\begin{equation} \label{eq:qprod_pushforward} 
\phi_i \star_\tau \phi_j = \sum_{d\in \NE(X)_\ZZ} \sum_{n=0}^\infty \frac{Q^d}{n!}  \inv^* \ev_{3,*} \left (\ev_1^*(\phi_i) \ev_2^*(\phi_j) \prod_{l=4}^{n+3} \ev_l^*(\tau) \cap [X_{0,n+3,d}]^{\rm vir} \right).  
\end{equation} 
Conditions~\eqref{condition:1} and \eqref{condition:2} in~\S\ref{sec:conditions} ensure that $\ev_3: X_{0,n+3,d} \to \overline{IX}$ is proper, and thus that the push-forward along $\ev_3$ is well-defined without inverting equivariant parameters.  It follows that: 
\[ 
\phi_i \star_\tau \phi_j \in H_{\CR,T}^\bullet(X) \otimes_{R_T} R_T[\![\tau,Q]\!]  
\] 
where $R_T[\![\tau,Q]\!] = R_T[\![\tau^0,\dots,\tau^N]\!][\![Q]\!]$.  The product $\star_\tau$ defines an associative and commutative ring structure on $H_{\CR,T}^\bullet(X)\otimes_{R_T} R_T[\![\tau,Q]\!]$.  The non-equivariant limit of $\star_\tau$ exists, and this limit defines the non-equivariant quantum cohomology $\big(H_{\CR}^\bullet(X) \otimes_\CC \CC[\![\tau,Q]\!],\star_\tau\big)$.  

\begin{remark} \label{rem:divisoreq_qprod} The divisor equation \cite[Theorem 8.3.1]{Abramovich--Graber--Vistoli:2} implies that exponentiated $H^2$-variables and the Novikov variable $Q$ play the same role: one has 
\[ 
(\phi_i \star_\tau \phi_j, \phi_k) = 
\sum_{d \in \NE(X)_\ZZ} \sum_{n=0}^\infty \frac{Q^de^{\<\sigma,d\>}}{n!}  \correlator{\phi_i,\phi_j,\phi_k, \tau',\dots,\tau'}_{0,n+3,d}^X 
\] 
where $\tau = \sigma + \tau'$ with $\sigma = \sum_{i=1}^r \tau^i \phi_i$ and $\tau' = \tau_0 \phi_0 + \sum_{i=r+1}^N \tau^i \phi_i$.  The String Equation (ibid.) implies that the right-hand side here is in fact independent of $\tau_0$.
\end{remark}

\subsection{Givental's Lagrangian Cone} 
\label{sec:Givental_cone} 
Let $S_T(\!(z^{-1})\!)$ denote the ring of formal Laurent series in $z^{-1}$ with coefficients in $S_T$.   Givental's symplectic vector space is the space
\[
\cH = H_{\CR,T}^\bullet(X)\otimes_{R_T} S_T(\!(z^{-1})\!)[\![Q]\!] 
\]
equipped with the non-degenerate $S_T[\![Q]\!]$-bilinear alternating form: 
\[
\Omega(f, g) = - \Res_{z=\infty} (f(-z), g(z)) dz 
\]
with $f,g \in \cH$. The space is equipped with a standard polarization 
\[
\cH = \cH_+ \oplus \cH_- 
\]
where 
\begin{align*}
\cH_+ := H^\bullet_{\CR,T}(X) \otimes_{R_T} S_T[z][\![Q]\!] 
&&\text{and} &&
\cH_- := z^{-1} H^\bullet_{\CR,T}(X) \otimes_{R_T} S_T[\![z^{-1}]\!] [\![Q]\!] 
\end{align*}
are isotropic subspaces for $\Omega$. The standard polarization identifies $\cH$ with the cotangent 
bundle of $\cH_+$.  The \emph{genus-zero descendant Gromov--Witten potential}  is a formal function $\cF^0_X \colon (\cH_+,{-z} 1) \to S_T[\![Q]\!]$ defined on the formal neighbourhood of $-z\cdot 1$ in $\cH_+$ and taking values in $S_T[\![Q]\!]$:
\[
\cF^0_X(-z 1 + \bt(z)) = \sum_{d\in \NE(X)_\ZZ} 
\sum_{n=0}^\infty 
\frac{Q^d}{n!} 
\correlator{\bt(\psi),\dots,\bt(\psi)}_{0,n,d}^X  
\]
Here $\bt(z) = \sum_{n=0}^\infty t_n z^n$ with $t_n \in H_{\CR,T}^\bullet(X)\otimes_{R_T} S_T[\![Q]\!]$.  Let $\{\phi^i\}\subset H^\bullet_{\CR,T}(X) \otimes_{R_T} S_T$ denote the basis Poincar\'{e} dual to $\{\phi_i\}$, so that $(\phi_i,\phi^j) = \delta_i^j$.

\begin{definition}[\cite{Givental:symplectic, CCIT}]
\label{def:Lag_cone}
\emph{Givental's Lagrangian cone} $\cL_X\subset (\cH,-z1)$ is the graph of the differential $d \cF^0_X \colon \cH_+ \to T^* \cH_+ \cong \cH$. It consists of points of $\cH$ of the form: 
\begin{equation} 
\label{eq:point_on_cone}
-z 1 + \bt(z) + \sum_{d\in \NE(X)_\ZZ} 
\sum_{n=0}^\infty \sum_{i=0}^N 
\frac{Q^d}{n!} 
\correlator{\frac{\phi_i}{-z-\psi}, \bt(\psi),\dots,\bt(\psi)}_{0,n+1,d} 
\phi^i 
\end{equation} 
where $1/(-z-\psi)$ in the correlator should be expanded as the power series $\sum_{k=0}^\infty \psi^k (-z)^{-k-1}$ in $z^{-1}$.  
In a more formal language, we define the notion of a `point on $\cL_X$' as follows. Let $x=(x_1,\dots,x_n)$ be formal parameters. 
An \emph{$S_T[\![Q,x]\!]$-valued point} on $\cL_X$ is an element of $\cH[\![x]\!]$ of the form \eqref{eq:point_on_cone} with $\bt(z) \in \cH_+[\![x]\!]$ satisfying 
\[
\bt(z)|_{Q=x =0} = 0. 
\]
It should be thought of as a formal family of elements on $\cL_X$ parametrized by $x$.
\end{definition} 

The submanifold $\cL_X$ encodes all genus-zero Gromov--Witten invariants \eqref{eq:Gromov--Witten}.  It has the following special geometric properties \cite{Givental:symplectic}: \emph{it is a cone, and a tangent space $T$ of $\cL_X$ is tangent to $\cL_X$ 
exactly along $zT$}.  Knowing Givental's Lagrangian cone $\cL_X$ is equivalent to knowing the data of the quantum product $\star_\tau$, i.e.~$\cL_X$ can be reconstructed from $\star_\tau$ and vice versa.  See Remark~\ref{rem:fundamentalsol_cone}. 

\subsection{The Equivariant Quantum Connection and its Fundamental Solution} 
\label{sec:quantum_connection}
Let $v\in H_{\CR,T}^\bullet(X)$.  The equivariant quantum connection
\[
\nabla_v \colon H_{\CR,T}^\bullet(X) \otimes_{R_T} 
R_T[z][\![\tau,Q]\!] \to z^{-1} H_{\CR,T}^\bullet(X) \otimes_{R_T} R_T[z][\![\tau,Q]\!]
\]
is defined by 
\[
\nabla_v f(\tau) = \partial_v f(\tau) + z^{-1} v \star_\tau f(\tau)  
\]
where $\partial_v f(\tau) = \frac{d}{ds}f(\tau+sv)|_{s=0}$ is the directional derivative.  We write  $\nabla_i$ for $\nabla_{\phi_i}$ and  $\nabla f$ for $\sum_{i=0}^N (\nabla_i f) d\tau^i$.   The associativity of $\star_\tau$ implies that the  connection $\nabla$ is flat, that is, $[\nabla_i,\nabla_j] = 0$ for all $i$,~$j$.  Let $\rho$ denote the equivariant first Chern class (in the untwisted sector): 
\[
\rho := c_1^T(TX) \in H^2_T(X) \subset H^2_{\CR,T}(X)
\] 
For $\phi \in H^\bullet_{\CR,T}(X)$, we write $\deg \phi$ for the age-shifted (real) degree of $\phi$, so that $\phi\in H^{\deg \phi}_{\CR,T}(X)$. The equivariant Euler vector field $\cE$ and the grading operator $\mu\in \End_\CC(H_{\CR,T}^\bullet(X))$ are defined by 
\begin{align}
\label{eq:Euler_mu}
\begin{split}   
\cE &:= \sum_{i=1}^m \lambda_i \parfrac{}{\lambda_i} 
+ \sum_{i=0}^N 
\left(1- \frac{\deg \phi_i}{2}\right) \tau^i \parfrac{}{\tau^i} 
+ \partial_\rho  \\
\mu(\phi) &:= \left( 
\frac{\deg\phi}{2} - \frac{\dim_\CC X}{2} \right) \phi  
\end{split} 
\end{align} 
where $\lambda_1,\dots,\lambda_m\in H^2_T({\rm pt})$ are generators of $R_T$ (see \S\ref{sec:standing_notation}). The grading operator on $H_{\CR,T}^\bullet(X)\otimes_{R_T}R_T[z][\![\tau,Q]\!]$ is defined by 
\[
\Gr(f(\tau,z) \phi) = \left(\Big( \textstyle z\parfrac{}{z} + \cE\Big) f(\tau,z)\right) 
\phi + f(\lambda,\tau,z) \mu (\phi)  
\]
where $\phi \in H_{\CR,T}^{\bullet}(X)$ and $f(\lambda,\tau,z) \in R_T[z][\![\tau,Q]\!]$.  The quantum connection is compatible with the grading operator in the sense that $[\Gr, \nabla_i]=\nabla_{[\cE,\partial_{\tau^i}]} = (\frac{1}{2} \deg \phi_i - 1) \nabla_i$, $i=0,\dots,N$.  This follows from the virtual dimension formula for the moduli space of stable maps.

\begin{notation} 
Let $v\in H^2_T(X)$ be a degree-two class in the 
untwisted sector. The action of $v$ on $H^\bullet_{\CR,T}(X)$ 
is defined by $v \cdot \alpha = q^*(v) \cup \alpha$, where 
$q\colon IX \to X$ is the natural projection. 
(This coincides with the action of $v$ via the 
Chen--Ruan cup product.) 
\end{notation} 

Consider the flat section equations for $\nabla$, and a fundamental solution 
\[
L(\tau,z) \in \End_{R_T}(H^\bullet_{\CR,T}(X)) \otimes_{R_T} R_T(\!(z^{-1})\!)[\![\tau,Q]\!]
\]
determined by the following conditions:
\begin{align}
\label{eq:flatness}
\nabla_i L(\tau,z) \phi & = 0 && \text{for $i=0,\dots, N$} 
& &\text{(flatness)}  \\ 
\label{eq:divisor} 
\left(  v Q \parfrac{}{Q} - \partial_v \right)L(\tau,z)\phi 
& = L(\tau,z) \frac{v}{z}\phi && \text{for $v\in H^2_T(X)$}
& & \text{(divisor equation)} \\ 
\label{eq:initial} 
L(\tau,z) |_{\tau=Q=0} & = \id &&
& & \text{(initial condition)} 
\end{align} 
Here $\phi\in H^\bullet_{\CR,T}(X)$ and $v Q \parfrac{}{Q}$ with $v\in H^2_T(X)$ acts 
on Novikov variables as $Q^d \mapsto \<v,d\> Q^d$ 
(it acts by zero when $v\in H^2_{T}({\rm pt})
\subset H^2_T(X)$). 
The flatness equation fixes  $L(\tau,z)$ up to 
right multiplication by an endomorphism-valued function 
$g(z;Q)$ in $z$ and $Q$; the divisor equation implies 
that the ambiguity $g(z;Q)$ is independent of $Q$ and 
commutes with $v \cup$, $v\in H^2_T(X)$; 
finally the initial condition fixes $L(\tau,z)$ uniquely. 
The fundamental solution satisfying these 
conditions can be written explicitly in terms of (descendant) Gromov--Witten invariants: 
\begin{equation} 
\label{eq:L_descendant} 
L(\tau,z) \phi_i = \phi_i + 
\sum_{j=0}^N \sum_{d\in \NE(X)_\ZZ} 
\sum_{\substack{n=0 \\ (n\ge 1\text{ if }d=0)}}^\infty 
\frac{Q^d}{n!} 
\correlator{\frac{\phi_i}{-z-\psi},\tau,\dots,\tau,\phi_j}_{0,n+2,d}^X 
\phi^j
\end{equation} 
This is defined over $R_T$ (without inverting equivariant parameters) because it can be rewritten in terms of the push-forward along the last evaluation map $\ev_{n+2}$ as in \eqref{eq:qprod_pushforward}.

\begin{proposition} 
\label{prop:fundsol} 
The fundamental solution $L(\tau,z)$ in \eqref{eq:L_descendant} 
satisfies the conditions (\ref{eq:flatness}--\ref{eq:initial}).
Furthermore it satisfies:  
\begin{align*}  
L(\tau,z) & = \id + O(z^{-1}) & & \text{\rm (regularity at $z=\infty$)}  \\
\Gr L(\tau,z) \phi  &= 
L(\tau,z) \left(\mu - \frac{\rho}{z} \right) \phi 
& & \text{\rm (homogeneity)} \\ 
(\alpha,\beta) & =(L(\tau,-z) \alpha, L(\tau,z)\beta)  
& & \text{\rm (unitarity)} 
\end{align*}
where $\phi,\alpha,\beta\in H^\bullet_{\CR,T}(X)$. 
\end{proposition} 
\begin{proof} 
The fundamental solution \eqref{eq:L_descendant} 
is well-known: see \cite[Corollary 6.3]{Givental:equivariant}, 
\cite[Proposition 2]{Pandharipande:afterGivental}, 
\cite[Proposition 2.4]{Iritani}. 
The flatness equation follows from the topological recursion 
relations as explained in \cite[Proposition 2]{Pandharipande:afterGivental}; 
the divisor equation follows from the one 
\cite[Theorem 8.3.1]{Abramovich--Graber--Vistoli:2} for descendant 
Gromov--Witten invariants; 
the initial condition is obvious from formula~\eqref{eq:L_descendant}. 
Regularity at $z=\infty$ is also clear straight from the definition. 
Decompose $\tau = \sigma + \tau'$ where 
$\sigma = \sum_{i=0}^r \tau^i \phi_i$ 
and $\tau' = \sum_{i=r+1}^N \tau^i \phi_i$ 
(recall that $\phi_0 = 1$ and that $\phi_1,\dots,\phi_r$ induce a 
basis of $H^2_T(X)/H^2_T({\rm pt})$). 
The string and divisor equations \cite[Theorem 8.3.1]{Abramovich--Graber--Vistoli:2} 
imply that one has $L(\tau,z) \phi_i = 
S(\tau',z; Q e^{\sigma}) (e^{-\sigma/z}\phi_i)$ with:
\[
S(\tau',z; Q e^\sigma) \alpha = \alpha + 
\sum_{j=0}^N \sum_{d\in \NE(X)_\ZZ} \sum_{\substack{n=0 \\ 
(n\ge 1\text{ if }d=0)}}^\infty 
\frac{e^{\<\sigma,d\>}Q^d}{n!} 
\correlator{\frac{\alpha}{-z-\psi},
\tau',\dots,\tau',\phi_j}_{0,n+2,d} \phi^j  
\]
The dimension axiom shows that $S$ is homogeneous: 
$\Gr \circ S(\tau',z; Q e^\sigma)  = 
S(\tau',z;Q e^\sigma) \circ \Gr$. 
The homogeneity 
equation for $L(\tau,z)$ follows from this.  
The flatness equation, together with 
the Frobenius property of $\star_\tau$, shows that 
\[
\partial_i 
\big(L(\tau,-z)\alpha,L(\tau,z)\beta\big) 
= \big(\overline{\nabla}_i L(\tau,-z)\alpha, L(\tau,z) \beta\big)
+ \big(L(\tau,-z) \alpha, \nabla_i L(\tau,z)\beta\big)=0
\]
for $\overline{\nabla}_i =  \partial_i - z^{-1} \phi_i\star_\tau$. 
Thus the pairing $(L(\tau,-z)\alpha,L(\tau,z) \beta)$ does not depend on 
$\tau$. Consider the adjoint 
$L(\tau,z)^\dagger$ of $L(\tau,z)$ with respect to $(\cdot,\cdot)$. 
Then $(L(\tau,-z)^\dagger \alpha, 
L(\tau,z)^\dagger \beta)$ is also independent of $\tau$.  
The divisor equation implies that: 
\begin{multline*} 
\left( v Q \parfrac{}{Q} - \partial_v \right)
(L(\tau,-z)^\dagger \alpha, L(\tau,z)^\dagger \beta) \\ 
= (-\frac{v}{z} L(\tau,-z)^\dagger \alpha, L(\tau,z)^\dagger \beta) 
+ (L(\tau,-z)^\dagger \alpha, \frac{v}{z} L(\tau,z)^\dagger\beta) =0
\end{multline*} 
Thus it is also independent of $Q$. 
The initial condition at $\tau = Q=0$ now implies 
the unitarity $(L(\tau,-z)^\dagger \alpha, L(\tau,z)^\dagger \beta) 
= (\alpha,\beta)$. 
\end{proof} 

\begin{remark} 
By solving the differential equations as power series in $\tau$ and $Q$, 
we find that the coefficient $L_{d,k}(z)$ of $L(\tau,z)$ 
in front of $Q^d \tau_0^{k_0} 
\cdots \tau_N^{k_N}$ is a rational function in $z$ and $\lambda$, 
i.e.~that each matrix entry of $L_{d,k}$ 
lies in $\Frac(R_T[z]) \cong \CC(\lambda_1,\dots,\lambda_m,z)$. 
The Laurent expansion of each matrix element at $z=\infty$ gives an element 
of $R_T(\!(z^{-1})\!)$. As a rational function in $z$, $L_{d,k}$ 
has singularities not only at $z=0$ but also at other finite values 
of $z$. This at first sight seems to contradict the fact that the flatness equation 
and the divisor equation have singularities only at $z=0$. 
In fact there is no contradiction: the equivariant 
quantum differential equation is \emph{resonant} at certain values of $z$. 
The divisor equation for $L(\tau,z)$ together with the flatness 
equation gives 
\begin{align*}
  v Q \parfrac{}{Q} L(\tau,z) =z^{-1}
  \left( L(\tau,z) v  - v\star_\tau L(\tau,z)  \right) 
  &&
  v\in H^2_T(X)
\end{align*}
and this has a logarithmic singularity at the `large radius limit point' $Q=0$. 
The residue at $Q=\tau=0$ is given by the commutator 
$[v/z,-]$ and the resonance occurs when 
the cup product by $v/z$ has eigenvalues in $\ZZ\setminus\{0\}$. 
The coefficients $L_{d,k}(z)$ have poles at resonant $z$. 
In non-equivariant Gromov--Witten theory,  $v/z$ is always nilpotent and 
 resonance does not occur. 
More generally, if the equivariant parameters $\lambda_i$ are  
sufficiently small compared to $|z|$ ($|\lambda_i|\ll |z|$), 
resonance does not occur and the coefficients $L_{d,k}(z)$ are regular. 
\end{remark} 

\begin{remark}[\!\!\cite{Givental:symplectic}] 
\label{rem:fundamentalsol_cone}
The fundamental solution $L(\tau,z)$ is determined by the quantum 
product $\star_\tau$ via differential equations 
\eqref{eq:flatness}--\eqref{eq:initial}.  
Then $\tau \mapsto 
T_\tau = L(\tau,-z)^{-1} \cH_+$ gives a versal family of 
tangent spaces to Givental's cone $\cL_X$. The cone 
$\cL_X$ 
is reconstructed as $\cL_X = \bigcup_\tau z T_\tau$. 
\end{remark} 

We now study $\nabla$-flat sections $s(\tau,z)$ 
that are homogeneous of degree zero: $\Gr (s(\tau,z)) = 0$. 
By Proposition \ref{prop:fundsol}, if a flat section 
$L(\tau,z) f(z)$ is homogeneous of degree zero, 
then:
\[
\left(z\parfrac{}{z} + \mu -\frac{\rho}{z}\right) f(z)=0
\]
This differential equation has the fundamental solution: 
\[
z^{-\mu} z^\rho = z^{\rho/z} z^{-\mu} 
= \exp(\rho \log(z)/z) z^{-\mu} 
\]
that belongs to $\End_{R_T}(H_{\CR,T}^\bullet(X)) \otimes_{R_T} 
R_T[\log z](\!(z^{-1/k})\!)$ for some $k\in \N$; 
here $k$ is chosen so that all the eigenvalues of $k \mu$ are integers. 
Note that homogeneous flat sections can be multi-valued 
in $z$ (as they contain $\log z$). 
We have: 

\begin{corollary} 
\label{cor:homogeneous_flat_sections} 
The sections $s_i(\tau,z) = L(\tau,z) z^{-\mu} z^\rho \phi_i$, $i=0,\dots,N$ 
satisfy $\nabla s_i(\tau,z) = \Gr s_i(\tau,z) = 0$ and 
give a basis of homogeneous flat sections.  
 They belong to $H_{\CR,T}^\bullet(X) 
\otimes_{R_T} R_T[\log z](\!(z^{-1/k})\!)[\![\tau,Q]\!]$ 
for a sufficiently large $k \in \N$.
\end{corollary}

\section{Equivariant Gamma-Integral Structure}
\label{sec:integral_structure}

In this section we introduce one of the main ingredients of our result: an integral structure for equivariant quantum cohomology.  This is a $K^0_T({\rm pt})$-lattice in the space of flat sections for the equivariant quantum connection on $X$ which is isomorphic to the integral equivariant $K$-group $K^0_T(X)$: it generalizes the integral structure for non-equivariant quantum cohomology constructed by Iritani~\cite{Iritani} and Katzarkov--Kontsevich--Pantev~\cite{KKP}.  Similar structures have been studied by Okounkov--Pandharipande~\cite{Okounkov-Pandharipande:Hilbert} in the case where $X$ is a Hilbert scheme of points in $\CC^2$, and by Brini--Cavalieri--Ross~\cite{Brini--Cavalieri--Ross} in the case where $X$ is a 3-dimensional toric Calabi--Yau stack.   We define the integral structure in~\S\ref{sec:equiv_int_str}.  In \S\ref{sec:specialization} we observe that the quantum product, flat sections for the quantum connection, and integral structure continue to make sense when the Novikov variable $Q$ (see~\S\ref{sec:QC}) is specialized to $Q=1$.

The integral structure is defined in terms of a $T$-equivariant characteristic class of $X$ called the $\hGamma$-class.  One of the key points in this section is that the $\hGamma$-class behaves like a square root of the Todd class: see equation~\ref{eq:gammaproduct}.  When combined with the Hirzebruch--Riemann--Roch formula, this leads to one of the fundamental properties of the integral structure: that the so-called framing map is pairing-preserving (Proposition~\ref{prop:K_framing_pairing} below)

\subsection{The Equivariant Gamma Class and the Equivariant Gamma-Integral Structure}
\label{sec:equiv_int_str}

Let $K_T^0(X)$ denote the Grothendieck group of $T$-equivariant vector bundles on $X$.  We write $H^{\bullet\bullet}_{T}(IX) := \prod_{p} H^{2p}_{T}(IX)$.  We introduce an orbifold Chern character map $\tch \colon 
K^0_T(X) \to H_{T}^{\bullet\bullet}(IX)$ as follows.  Let $IX = \bigsqcup_{v \in \sfB} X_v$ be the decomposition of the inertia stack $IX$ into connected components, let $q_v \colon X_v \to X$ be the natural map, and let $E$ be a $T$-equivariant vector bundle on $X$.  The stabilizer $g_v$ along $X_v$ acts on the vector bundle $q_v^* E \to X_v$, giving an eigenbundle decomposition
\begin{equation} 
\label{eq:E_vf} 
q_v^* E = \bigoplus_{0\le f < 1} E_{v,f}
\end{equation} 
where $g_v$ acts on $E_{v,f}$ by $\exp(2\pi\tti f)$.  The equivariant Chern character is defined to be
\[
\tch(E) = \bigoplus_{v\in \sfB} \sum_{0\le f<1} e^{2\pi\tti f} 
\ch^T(E_{v,f}) 
\]
where $\ch^T(E_{v,f})\in H^{\bullet \bullet}_T(X_v)$ is the $T$-equivariant Chern character.   Let $\delta_{v,f,i}$, $1 \leq i \leq \rank(E_{v,f})$ be the $T$-equivariant Chern roots of $E_{v,f}$, so that $c^T(E_{v,f}) = \prod_i (1+ \delta_{v,f,i})$. 
These Chern roots are not actual cohomology classes, but  symmetric polynomials in the Chern roots make sense as equivariant cohomology classes on $X_v$.  The $T$-equivariant orbifold Todd class $\tTd(E) \in H^{\bullet\bullet}_{T}(IX)$ is defined to be: 
\[
\tTd(E) = \bigoplus_{v\in \sfB} 
\left( \prod_{0<f<1} \prod_{i=1}^{\rank(E_{v,f})} 
\frac{1}{1- e^{-2\pi\tti f} e^{-\delta_{v,f,i}}}\right) 
\prod_{i=1}^{\rank E_{v,0}} 
\frac{\delta_{v,0,i}}{1-e^{-\delta_{v,0,i}}}. 
\]
We write $\tTd_X = \tTd(TX)$ for the orbifold Todd class of the tangent bundle.

Recall that, because we are assuming condition \eqref{condition:2} 
from \S\ref{sec:conditions}, all of the $T$-weights of $H^0(X,\cO)$ 
lie in a strictly convex cone in $\Lie(T)^*$. 
After changing the identification of $T$ with $(\Cstar)^m$ 
if necessary, we may assume that this cone is contained within 
the cone spanned by the standard characters $\lambda_1,\dots,
\lambda_m$ of $H^2_T({\rm pt}) = \Lie(T)^*$ defined in 
\S\ref{sec:standing_notation}. 
As is explained in~\cite{Coates--Iritani--Jiang--Segal}, 
under conditions (\ref{condition:1}--\ref{condition:2}) in \S\ref{sec:conditions} 
there is a well-defined equivariant Euler characteristic
\begin{equation} 
\label{eq:Euler_char}
\chi(E) := \sum_{i=0}^{\dim X} 
(-1)^i \ch^T\big(H^i(X,E)\big).  
\end{equation} 
taking values in
\[
\ZZ[\![e^{\lambda}]\!][e^{-\lambda}]_{\rm rat}
:= \left\{ f \in \ZZ[\![e^{\lambda_1},\dots,e^{\lambda_m}]\!]
[e^{-\lambda_1},\dots,e^{-\lambda_m}] : 
\begin{array}{l}
\text{$f$ is the Laurent expansion} \\ 
\text{of a rational function in} \\ 
\text{$e^{\lambda_1},\dots,
e^{\lambda_m}$ at} \\ 
\text{$e^{\lambda_1} = \cdots = e^{\lambda_m}=0$}
\end{array}\right\}
\]
and we expect that the following equivariant Hirzebruch--Riemann--Roch (HRR) 
formula should hold:  
\begin{equation}
\label{eq:equiv_HRR} 
\chi(E) = 
\int_{IX} \tch(E) \cup \tTd_X
\end{equation} 
This formula should be interpreted with care.  The right-hand side is defined via the localization 
formula, and lies in a completion $\hS_T$ of $S_T$:  
\[
\hS_T := \left\{
\sum_{n\in \ZZ} a_n : \text{$a_n \in S_T$, $\deg a_n =n$, there exists $n_0\in \ZZ$ such that $a_n = 0$ for all $n< n_0$} \right\}  
\]
There is an inclusion of 
rings $\ZZ[\![e^{\lambda}]\!][e^{-\lambda}]_{\rm rat} \hookrightarrow \hS_T$ given by Laurent expansion at $\lambda_1 = \cdots = \lambda_m=0$ (see \cite{Coates--Iritani--Jiang--Segal}), and \eqref{eq:equiv_HRR} asserts that $\chi(E)$ coincides with the right-hand side after this inclusion.

We now introduce a lattice in the space of homogeneous flat sections for the quantum connection which is identified with the equivariant $K$-group of $X$.   The key ingredient in the definition is the characteristic class, called the \emph{Gamma class}, defined as follows.
Let $E$ be a vector bundle on $X$ and consider 
the bundles $E_{v,f}\to X_v$ and their equivariant 
Chern roots $\delta_{v,f,i}$, $i=1,\dots,\rank(E_{v,f})$ as 
above (see \eqref{eq:E_vf}). 
The equivariant Gamma class 
$\hGamma(E)\in H^{\bullet\bullet}_{T}(IX)$ is defined to be: 
\[
\hGamma(E) = \bigoplus_{v\in \sfB} 
\prod_{0\le f<1} \prod_{i=1}^{\rank(E_{v,f})} 
\Gamma(1-f + \delta_{v,f,i})  
\]
Here the $\Gamma$-function on the right-hand side should be expanded 
as a Taylor series at $1-f$, and then evaluated at $\delta_{v,f,i}$. The identity
$\Gamma(1-z) \Gamma(1+z) = 2\pi\tti z e^{-\pi \tti z}/
(1- e^{-2\pi \tti z})$ implies that
\begin{align}
\label{eq:gammaproduct} 
\begin{split}  
\left[\hGamma(E^*) \cup \hGamma(E)\right]_v  
& = \prod_{i,f} 
\Gamma(1-\fbar - \delta_{v,f,i}) \Gamma(1-f+\delta_{v,f,i}) \\ 
& = (2\pi\tti)^{\rank((q_v^*E)^{\rm mov})} 
\left[e^{-\pi\tti (\age(q^*E)+ c_1(q^* E))}
(2\pi\tti)^{\frac{\deg_0}{2}} \tTd(E)\right]_{\inv(v)}  
\end{split} 
\end{align} 
where $\cup$ is the cup product on $IX$, $[\cdots]_v$ denotes the component in $H^\bullet_T(X_v)$, $0\le \fbar<1$ is the fractional part of $-f$, $(q_v^*E)^{\rm mov} = \bigoplus_{f\neq 0} E_{v,f}$ is the moving part of $q_v^*E$, $q\colon IX \to X$ is the natural projection, $\age(q^*E) \colon IX \to \QQ$ is the locally constant function given by $\age(q^*E)|_{X_v} = \sum_{f} f \rank(E_{v,f})$, $\deg_0 \colon H^{\bullet\bullet}_T(IX) \to H^{\bullet\bullet}_T(IX)$ is the degree operator defined by $\deg_0(\phi)= 2p \phi$ for $\phi\in H^{2p}_T(IX)$, and $\inv(v)\in \sfB$ corresponds to the component $X_{\inv(v)}$ of $IX$ defined by $\inv(X_v) = X_{\inv(v)}$.  Note that $\deg_0$ means the degree as a class on $IX$, not the age-shifted degree as an element of $H^{\bullet\bullet}_{\CR,T}(X)$.

\begin{definition} 
\label{def:K_framing} 
Define the \emph{$K$-group framing} 
\[
\frs \colon K_T^0(X) \to 
H^\bullet_{\CR,T}(X)\otimes_{R_T} R_T[\log z](\!(z^{-1/k})\!)[\![Q,\tau]\!]
\]
by the formula: 
\[
\frs(E)(\tau,z) = \frac{1}{(2\pi)^{\dim X/2}} 
L(\tau,z) z^{-\mu} z^\rho \left( \hGamma_X \cup 
(2\pi\tti)^{\frac{\deg_0}{2}} \inv^* \tch(E) \right) 
\]
where $k\in \N$ is as in Corollary \ref{cor:homogeneous_flat_sections} 
and $\hGamma_X \cup$ is the cup product in $H^{\bullet\bullet}_T(IX)$. 
Corollary \ref{cor:homogeneous_flat_sections} shows that the image of $\frs$ is contained in the space of $\Gr$-degree zero 
flat sections. 
Note that $z^{-\mu}$ maps $H^{\bullet\bullet}_{\CR,T}(X)$ 
into $H^{\bullet}_{\CR,T}(X)\otimes_{R_T} R_T(\!(z^{-1/k})\!)$. 
\end{definition} 
For $T$-equivariant vector bundles $E$, $F$ on $X$, 
let $\chi(E,F) \in \ZZ[\![e^\lambda]\!][e^{-\lambda}]_{\rm rat}$ denote the equivariant Euler pairing 
defined by: 
\begin{equation} 
\label{eq:Euler_pairing} 
\chi(E,F) := \sum_{i=0}^{\dim X} (-1)^i 
\ch^T\big(\Ext^i(E,F)\big)  
\end{equation} 
We use a $z$-modified version $\chi_z(E,F)$ that is 
given by replacing equivariant parameters $\lambda_j$ in $\chi(E,F)$ 
with $2\pi\tti \lambda_j/z$: 
\begin{equation} 
\label{eq:modified_Euler}
\chi_z(E,F) := (2\pi\tti z^{-1})^{\sum_{i=1}^m \lambda_i \partial_{\lambda_i}} 
\chi(E,F) \in \ZZ[\![e^{2\pi\tti \lambda/z}]\!][e^{-2\pi \tti \lambda/z}]_{\rm rat} 
\end{equation}

\begin{proposition}[cf.~{\cite[Proposition 2.10]{Iritani}}] 
\label{prop:K_framing_pairing}
Suppose that the equivariant HRR formula \eqref{eq:equiv_HRR} 
holds. 
For $E, F \in K_T^0(X)$, we have 
\[
\left(\frs(E)(\tau,e^{-\pi \tti}z), \frs(F)(\tau,z) \right) 
= \chi_z(E,F). 
\]
\end{proposition} 
\begin{proof} 
Set $\Psi(E) = \hGamma_X \cup (2\pi\tti)^{\frac{\deg_0}{2}} \inv^* 
\tch(E)$. 
Using the unitarity in Proposition \ref{prop:fundsol}, we have 
\begin{equation} 
\label{eq:frs_pairing} 
\left(\frs(E)(\tau,e^{-\pi \tti}z), \frs(F)(\tau,z) \right) 
= \frac{1}{(2\pi)^{\dim X}} \left(z^{-\mu} e^{\pi\tti \mu} 
z^\rho e^{-\pi\tti \rho} \Psi(E), z^{-\mu} z^\rho \Psi(F)\right). 
\end{equation} 
Write $\lambda \partial_\lambda = 
\sum_{i=1}^m \lambda_i \partial_{\lambda_i}$. 
Using $(z^{-\mu}\alpha,z^{-\mu} \beta) = z^{-\lambda\partial_\lambda} 
(\alpha,\beta)$, $e^{\pi\tti \mu} \rho = -\rho e^{\pi\tti\mu}$, 
$(z^{-\rho} \alpha,z^{\rho} \beta) = (\alpha,\beta)$, 
we have 
\begin{align*} 
\eqref{eq:frs_pairing} & = 
\frac{z^{-\lambda \partial_\lambda}}{(2\pi)^{\dim X}} 
\left(e^{\pi\tti \rho}e^{\pi\tti \mu} \Psi(E), \Psi(F)\right) \\ 
& = \frac{z^{-\lambda \partial_\lambda}}{(2\pi)^{\dim X}} 
\int_{IX} \left( e^{\pi\tti q^* \rho} e^{\pi\tti \mu} 
\hGamma_X (2\pi\tti)^{\frac{\deg_0}{2}}\inv^* \tch(E) \right) 
\cup \inv^*\left( \hGamma_X (2\pi\tti)^{\frac{\deg_0}{2}} 
\inv^* \tch(F) \right) 
\\
& =
\frac{z^{-\lambda \partial_\lambda}}{(2\pi)^{\dim X}} 
\sum_{v\in \sfB} 
\int_{X_v} 
e^{\pi\tti q_v^*\rho}
e^{\pi\tti (\iota_{v} - \frac{\dim X}{2})}
\left[\hGamma_X^* \hGamma_X \right]_{\inv(v)} 
(2\pi\tti)^{\frac{\deg_0}{2}} \left[\tch(E^*) \tch(F)\right]_v \\ 
& = z^{-\lambda \partial_\lambda}
\sum_{v\in \sfB}
\frac{1}{(2\pi \tti)^{\dim X_v}} 
\int_{X_v} 
(2\pi\tti)^{\frac{\deg_0}{2}} \left[
\tch(E^* \otimes F) \cup \tTd_X \right]_v 
\end{align*} 
where we set $\hGamma_X^* = \hGamma(T^*X)$ 
and used equation \eqref{eq:gammaproduct} in the last line. 
The last expression equals $\chi_z(E,F)$ by the 
HRR formula \eqref{eq:equiv_HRR}. 
\end{proof}

\begin{remark} 
  One could also consider the $K$-group framing for \emph{topological} equivariant $K$-theory.  For toric stacks, the topological and  algebraic $K$-groups coincide.  
\end{remark}

\begin{remark} 
Okounkov--Pandharipande \cite{Okounkov-Pandharipande:Hilbert} 
and Braverman--Maulik--Okounkov \cite{Braverman--Maulik--Okounkov} 
introduced shift operators $\bbS_i$ on quantum 
cohomology, which induce the shift $\lambda_i \to \lambda_i + z$ 
of equivariant parameters (see \cite[Chapter 8]{Maulik--Okounkov} for a detailed 
description). 
Our $K$-theoretic flat sections $\frs(E)$ are invariant under the 
shift operators, and our main result suggests that shift 
operators for toric stacks 
should be defined globally on the secondary toric variety. 
\end{remark} 

\subsection{Specialization of Novikov Variables} 
\label{sec:specialization} 

In this section we show that the quantum product, the flat sections for the quantum connection, and the $K$-group framing remain well-defined after the specialization $Q=1$ of the Novikov variable~$Q$.  Recall that $\tau^0,\dots,\tau^N$ are co-ordinates on $H_{\CR,T}^\bullet(X)$ dual to a homogeneous $R_T$-basis $\{\phi_0,\dots,\phi_N\}$ of $H_{\CR,T}^\bullet(X)$, and that:
\begin{itemize}
\item $\phi_0=1$;
\item $\phi_1,\dots,\phi_r\in H^2_T(X)$;
\item $\phi_1,\dots,\phi_r$ descend to a basis of $H^2(X) = H^2_T(X)/H^2_T({\rm pt})$. 
\end{itemize}
Without loss of generality we may assume that the images of $\phi_1,\dots,\phi_r$ in $H^2(X)$ are nef and integral.

It is clear from Remark~\ref{rem:divisoreq_qprod} that the specialization $Q=1$ 
of the quantum product is well-defined, and we have:
\[
\phi_i \star_\tau \phi_j \Big|_{Q=1} \in H_{\CR,T}^\bullet(X) 
\otimes_{R_T} R_T[\![e^{\tau^1},\dots,e^{\tau^r},
\tau^{r+1},\dots,\tau^N]\!]
\]
As discussed in Remark~\ref{rem:divisoreq_qprod}, the product $\phi_i\star_\tau\phi_j$ is independent of $\tau^0$.  We saw in the proof of Proposition~\ref{prop:fundsol} that the fundamental solution $L(\tau,z)$ factorizes as $L(\tau,z) = S(\tau',z; Qe^{\sigma}) e^{-\sigma/z}$. Thus the specialization $Q=1$ makes sense for $L(\tau,z)$ and:
\[
L(\tau,z)\Big |_{Q=1} \in \End\left(H_{\CR,T}^\bullet(X)\right) 
\otimes_{R_T} R_T[\tau^0,\tau^1,\dots,\tau^r][\![z^{-1}]\!][\![
e^{\tau^1},\dots,e^{\tau^r}, \tau^{r+1},\dots,\tau^{N}]\!]  
\]
The specialization $Q=1$ for homogeneous flat sections $\frs(E)$ 
in Definition \ref{def:K_framing} 
(as well as the homogeneous flat sections $s_i$ in Corollary \ref{cor:homogeneous_flat_sections}) 
also makes sense and we have 
\[
\frs(E)(\tau,z)\Big|_{Q=1} 
\in H_{\CR,T}^\bullet(X) \otimes_{R_T}  
R_T[\tau^0,\tau^1,\dots,\tau^r,\log z](\!(z^{-1/k})\!)[\![
e^{\tau^1},\dots,e^{\tau^r}, \tau^{r+1},\dots,\tau^{N}]\!]  
\]
where $k\in \N$ is such that all the eigenvalues of $k\mu$ are integral.

\section{Toric Deligne--Mumford Stacks as GIT Quotients}
\label{sec:notation}

In the rest of this paper we consider toric Deligne--Mumford stacks $X$ with semi-projective coarse moduli space such that the torus-fixed set $X^T$ is non-empty.  This is the class of stacks that arise as GIT quotients of a complex vector space by the action of a complex torus.  In this section we establish notation and describe basic properties of these quotients.  Good introductions to this material include~\cite[\S VII]{Audin}, \cite{Cox-Little-Schenck} and~\cite{Borisov--Chen--Smith}.

\subsection{GIT Data}  
\label{sec:GITdata} 
Consider the following data: 
\begin{itemize} 
\item $K \cong (\Cstar)^r$, a connected torus of rank $r$; 
\item $\LL = \Hom(\Cstar,K)$, the cocharacter lattice of $K$; 
\item $D_1,\ldots,D_m \in \LL^\vee = 
\Hom(K,\Cstar)$, characters of $K$.  
\end{itemize} 
The characters $D_1,\ldots,D_m$ define a map from $K$ 
to the torus $T = (\Cstar)^m$, and hence define an action of 
$K$ on $\CC^m$.  

\begin{notation}
  \label{not:cones}
  For a subset $I$ of $\{1, 2, \ldots, m\}$, write $\overline{I}$ 
  for the complement of $I$, and set 
  \begin{align*}
    \angle_I &= \big\{ \textstyle\sum_{i \in I} a_i D_i : \text{$a_i \in
      \RR$, $a_i > 0$} \big\} \subset
    \LL^\vee \otimes \RR, \\
    (\Cstar)^I \times \CC^{\overline{I}} 
    & = \big \{ (z_1,\dots,z_m) : z_i \neq 0 \text{ for } i\in I 
\big\} \subset \CC^m.  
   \end{align*} 
   We set $\angle_\emptyset := \{0\}$. 
\end{notation} 

\begin{definition} 
\label{def:toricstack} 
Consider now a \emph{stability condition} 
$\omega \in \LL^\vee \otimes \RR$, and set:
\begin{align*}  
    \cA_\omega &= 
    \big\{ I \subset \{1,2,\ldots,m\} : \omega \in \angle_I \big\} \\
    U_\omega &=\bigcup_{I\in \cA_\omega} 
(\Cstar)^I \times \CC^{\overline{I}} \\
    X_\omega &= \big[U_\omega \big/ K \big]
\end{align*}
The square brackets here indicate that $X_\omega$ is the stack quotient 
of $U_\omega$ (which is $K$-invariant) by $K$.  
We call $X_\omega$ the \emph{toric stack 
associated to the GIT data} $(K;\LL;D_1,\ldots,D_m; \omega)$. 
We refer to elements of $\cA_\omega$ as \emph{anticones}, 
for reasons which will become clear in \S \ref{sec:stacky_fan} 
below.  
\end{definition}

\begin{assumption} \label{assumption}
  We assume henceforth that:
  \begin{enumerate}
  \item $\{1,2,\ldots,m\} \in \cA_\omega$;
  \item for each $I \in \cA_\omega$, 
the set $\{D_i : i \in I\}$ spans $\LL^\vee \otimes \RR$ over $\RR$.
\end{enumerate}
These are assumptions on the stability condition $\omega$.  
The first ensures that $X_\omega$ is non-empty; 
the second ensures that $X_\omega$ is a Deligne--Mumford stack. 
Under these assumptions, $\cA_\omega$ is closed under 
enlargement of sets, i.e.~if $I\in \cA_\omega$ and 
$I\subset J$ then $J \in \cA_\omega$. 
\end{assumption}

Let $S\subset\{1,2,\dots,m\}$ 
denote the set of indices $i$ such that $\{1,\dots,m\} \setminus \{i\} 
\notin \cA_\omega$. It is easy to see that 
the characters $\{D_i : i \in S\}$ are linearly independent 
and that every element of $\cA_\omega$ contains $S$ 
as a subset. 
Therefore we can write 
\begin{align}
\label{eq:Uomega_factorization}
\begin{split}  
\cA_\omega & = \{ I \sqcup S : I \in \cA_\omega'\} \\
U_\omega &\cong U'_\omega \times (\CC^\times)^{|S|} 
\end{split} 
\end{align} 
for some $\cA_\omega' \subset 2^{\{1,\dots,m\} \setminus S}$ and 
an open subset $U_\omega'$ of $\CC^{m-|S|}$. 
The toric stack $X_\omega$ can be also written 
as the quotient $[U_\omega'/G]$ of $U_\omega'$ 
for $G = \Ker(K \to (\Cstar)^{|S|})$: this corresponds to 
the original construction of toric Deligne--Mumford stacks 
by Borisov--Chen--Smith \cite{Borisov--Chen--Smith}. 

The space of stability conditions $\omega\in \LL^\vee
\otimes \RR$ satisfying Assumption 
\ref{assumption} has a wall and chamber structure. 
The chamber $C_\omega$ to which $\omega$ belongs is given by 
\begin{equation} 
\label{eq:ext_amplecone} 
C_\omega = \bigcap_{I \in \cA_\omega} \angle_I.  
\end{equation} 
and $X_\omega \cong X_{\omega'}$ as long as $\omega' \in C_\omega$. 
The GIT quotient $X_{\omega'}$ changes when $\omega'$ crosses a codimension-one 
boundary of $C_\omega$. 
We call $C_\omega$ the \emph{extended ample cone}; 
as we will see in \S\ref{sec:amplecone} below, 
it is the product of the ample cone for $X_\omega$ with a simplicial cone. 
\begin{example}
Let $K = \Cstar$, so that $\LL = \Hom(\Cstar,K) \cong \ZZ$. 
Let $D_1 = D_2 = 2 \in \ZZ^\vee$, and set 
$\omega = 1\in \ZZ\otimes \RR=\RR$. 
Then $U_\omega=\CC^2\setminus \{(0,0)\}$, 
and $X_\omega$ is the weighted projective stack $\PP(2,2)$.  
\end{example}

\subsection{GIT Data and Stacky Fans}
\label{sec:stacky_fan}
In the foundational work of Borisov--Chen--Smith \cite{Borisov--Chen--Smith}, 
toric DM stacks are defined in terms of \emph{stacky fans}. 
Jiang \cite{Jiang} introduced the notion of an 
\emph{extended} stacky fan, which is a stacky fan with 
extra data. 
Our GIT data above are in one-to-one correspondence 
with extended stacky fans satisfying certain conditions, as we now explain.

An \emph{$S$-extended stacky fan} 
is a quadruple $\mathbf{\Sigma}= (\bN,\Sigma,\beta,S)$, where:
\begin{itemize} 
\item $\bN$ is a finitely generated abelian group\footnote{Note that $\bN$ may have torsion.}; 
\item $\Sigma$ is a rational simplicial fan in $\bN\otimes \RR$;   
\item $\beta \colon \ZZ^m \to \bN$ is a homomorphism; 
we write $b_i = \beta(e_i)\in \bN$ for the image of the $i$th standard 
basis vector $e_i\in\ZZ^m$,
and write $\overline{b}_i$ for the image of $b_i$ in $\bN\otimes \RR$; 
\item $S \subset \{1,\dots,m\}$ is a subset, 
\end{itemize} 
such that:
\begin{itemize} 
\item each one-dimensional cone of $\Sigma$
is spanned by $\overline{b}_i$ for a unique 
$i\in \{1,\dots,m\}\setminus S$, and
each $\overline{b}_i$ with 
$i\in \{1,\dots,m\} \setminus S$ spans a one-dimensional 
cone of $\Sigma$; 

\item for $i\in S$, $\overline{b}_i$ lies in the support $|\Sigma|$ 
of the fan. 
\end{itemize} 
The vectors $b_i$ for $i\in S$ are called \emph{extended vectors}. 
Stacky fans as considered by Borisov--Chen--Smith correspond to the cases where $S = \emptyset$. 
For an extended stacky fan $(\bN,\Sigma,\beta,S)$, 
the \emph{underlying stacky fan} is the triple $(\bN,\Sigma,\beta')$ 
where $\beta' \colon \ZZ^{m-|S|} \to \bN$ is 
obtained from $\beta$ by deleting the columns corresponding to $S
\subset \{1,\dots,m\}$.   The toric Deligne--Mumford stack associated to an extended stacky fan $(\bN,\Sigma,\beta,S)$ depends only on the underlying stacky fan.

To obtain an extended stacky fan from our GIT data, 
consider the exact sequence:
\begin{equation}\label{eq:exact}
  \xymatrix{
    0 \ar[r] &
    \LL \ar[r] &
    \ZZ^m \ar[r]^\beta & 
    \bN \ar[r] & 
    0 }
\end{equation}
where the map from $\LL$ to $\ZZ^m$ is given by $(D_1,\ldots,D_m)$ 
and $\beta \colon \ZZ^m \to \bN$ is the cokernel of the map $\LL \to \ZZ^m$. 
Let $b_i = \beta(e_i)\in \bN$ and $\overline{b}_i\in \bN\otimes \RR$ 
be as above and, given 
a subset $I$ of $\{1,\dots,m\}$, let $\sigma_I$ denote the cone in 
$\bN\otimes \RR$ generated by $\{\overline{b}_i : i\in I\}$. 
The extended stacky fan
$\mathbf{\Sigma}_\omega=(\bN, \Sigma_\omega, \beta,S)$
corresponding to our data 
consists of the group $\bN$ and the map $\beta$ defined 
above, together with a fan $\Sigma_\omega$ in $\bN\otimes\RR$ 
and $S$ given by\footnote{This is why we refer to the 
elements of $\cA_\omega$ as anticones.}: 
\begin{align*} 
\Sigma_\omega & = \{\sigma_{I} : \overline{I} \in \cA_\omega\}, \\ 
S & = \{ i \in \{1,\dots,m\} : \overline{\{i\}}  
\notin \cA_\omega\}. 
\end{align*} 
The quotient construction in \cite[\S2]{Jiang} 
coincides with that in Definition~\ref{def:toricstack}, 
and therefore $X_\omega$ is the toric Deligne-Mumford 
stack corresponding to $\mathbf{\Sigma}_\omega$. 
Extended stacky fans $(\bN, \Sigma_\omega, \beta,S)$ corresponding to GIT data 
satisfy the following conditions: 
\begin{itemize} 
\item[(1)] the support $|\Sigma_\omega|$ 
of the fan is convex and full-dimensional; 
\item[(2)] there is a strictly convex piecewise-linear function 
$f\colon |\Sigma_\omega| 
\to \RR$ that is linear on each cone of $\Sigma_\omega$; 
\item[(3)] the map $\beta \colon \ZZ^m \to \bN$ is surjective. 
\end{itemize} 
The first two conditions are geometric constraints on $X_\omega$: 
they are equivalent to saying 
that the corresponding toric stack $X_\omega$ 
is semi-projective and has a torus fixed point. 
The third condition can be always achieved by adding  
enough extended vectors.

Conversely, given an extended stacky fan 
$\mathbf{\Sigma}=(\bN, \Sigma,\beta,S)$ satisfying the  
conditions (1)--(3) just stated, we can obtain GIT data as follows.  
Define a free $\ZZ$-module $\LL$ by the 
exact sequence \eqref{eq:exact} 
and define $K := \LL\otimes \Cstar$. 
The dual of \eqref{eq:exact} is an exact sequence:
\begin{equation} 
\label{eq:divseq}
\xymatrix{
  0 \ar[r] &
  \bN^{\vee} \ar[r] &
  (\ZZ^m)^{\vee} \ar[r]& 
  \LL^{\vee}}
\end{equation} 
and we define the character $D_i \in \LL^\vee$ of $K$ to be the image 
of the $i$th standard basis vector in $(\ZZ^m)^\vee$ under 
the third arrow $(\ZZ^m)^\vee \to \LL^\vee$.  
Set:
\[
\cA_{\omega}=\left\{I\subset \{1,2,\cdots,m\}: 
S \subset I, \ \text{$\sigma_{\overline{I}}$ 
is a cone of $\Sigma$}\right\}
\]
and take the stability condition $\omega\in \LL^{\vee}\otimes\RR$ 
to lie in $\bigcap_{I \in \cA_{\omega}} \angle_{I}$; 
the condition (2) ensures that this intersection is non-empty.  
This specifies the data in Definition~\ref{def:toricstack}. 

\subsection{Torus-Equivariant Cohomology}
\label{sec:equiv_coh} 
The action of $T = (\Cstar)^m$ on $U_\omega$ descends 
to a $\cQ := T/K$-action on $X_\omega$. 
We also consider an ineffective $T$-action on $X_\omega$ 
induced by the projection $T \to \cQ$. 
The $\cQ$-equivariant and $T$-equivariant cohomology of $X_\omega$ 
are modules over $R_\cQ:= H^\bullet_\cQ({\rm pt};\CC)$ 
and $R_T := H^\bullet_T({\rm pt};\CC)$ respectively.  
By the exact sequence \eqref{eq:exact}, the Lie algebra of 
$\cQ$ is identified with $\bN\otimes \CC$ and 
$R_\cQ \cong \Sym^\bullet(\bN^\vee \otimes \CC)$. 
Let $\lambda_i\in R_T$ be the equivariant first Chern class 
of the irreducible $T$-representation given by the 
projection $T \cong (\Cstar)^m \to \Cstar$ to the $i$th factor.  
Then $R_T = \CC[\lambda_1,\dots,\lambda_m]$. 
It is well-known that: 
\begin{equation} 
\label{eq:Qequiv_cohomology} 
H_\cQ^\bullet(X_\omega;\CC) = R_\cQ
[u_1,\ldots,u_m]\big/(\mathfrak{I} + \mathfrak{J}) 
\end{equation} 
where $u_i$ is the $\cQ$-equivariant class Poincar\'e-dual to the
toric divisor:
\begin{equation}
  \label{eq:T-invariant_divisor}
  \big\{ (z_1,\ldots,z_m) \in U_\omega : z_i = 0 \big\} \big/ K
\end{equation}
and $\mathfrak{I}$ and $\mathfrak{J}$ are the ideals 
of additive and multiplicative relations: 
\begin{align*} 
\mathfrak{I} & = \big\langle \chi - \textstyle\sum_{i=1}^m 
\<\chi,b_i\> u_i : \chi\in \bN^\vee\otimes \CC \big\rangle, \\ 
\mathfrak{J} & = \big\langle \textstyle \prod_{i \not \in I} u_i  : 
I \not \in \cA_\omega \big \rangle. 
\end{align*} 
Note that $u_i=0$ for $i\in S$ because the corresponding 
divisor \eqref{eq:T-invariant_divisor} is empty (see equation~\ref{eq:Uomega_factorization}). Indeed, 
this relation is contained in the ideal $\mathfrak{J}$. 
The $T$-equivariant cohomology is given by the extension 
of scalars: 
\[
H_T^\bullet(X_\omega) \cong H_\cQ^\bullet(X_\omega) \otimes_{R_\cQ} 
R_T 
\]
where the algebra homomorphism $R_\cQ \to R_T$ is given by 
$\chi \mapsto \sum_{i=1}^m \<\chi,b_i\> \lambda_i$ 
for $\chi\in \bN^\vee \otimes \CC$. 

\begin{remark} 
We note that the assumptions at the beginning of 
\S \ref{sec:equivariant_qc} are satisfied for toric Deligne--Mumford 
stacks obtained from GIT data. 
First, all the $\cQ$-weights appearing in the $\cQ$-representation 
$H^0(X_\omega,\cO)$ are contained in the strictly convex 
cone $|\Sigma_\omega|^\vee = 
\{ \chi \in \bN^\vee\otimes \RR : \text{$\<\chi,v\> \ge 0$ for all $v\in |\Sigma_\omega|$}\}$. 
Second, $X_\omega$ is equivariantly formal 
since the cohomology group of $X_\omega$ is generated by 
$\cQ$-invariant cycles \cite{GKM}. 
Because each component of $IX_\omega$ is again a 
toric stack given by certain GIT data (see \S \ref{sec:inertia}), 
we have that $IX_\omega$ is also equivariantly formal. 
The same conclusions hold for the $T$-action. 
\end{remark} 

\subsection{Second Cohomology and Homology} 
\label{sec:second_cohomology} 
There is a commutative diagram:
\begin{equation} 
\label{eq:H2_diagram}
\begin{aligned}
  \xymatrix{
    & 0 \ar[d] & 0 \ar[d] & 0 \ar[d] \\
    0 \ar[r] & H^2_{\cQ}({\rm pt};\RR) \ar[r] \ar[d] & H^2_T({\rm pt};\RR) \ar[r] \ar[d] & H^2_K({\rm pt};\RR) \ar[r] \ar@{=}[d] & 0 \\
    0 \ar[r] & H^2_{\cQ}(X_\omega;\RR) \ar[r] \ar[d] & H^2_T(X_\omega;\RR) \ar[r] \ar[d] & H^2_K({\rm pt};\RR) \ar[r] \ar[d] & 0 \\
    0 \ar[r] & H^2(X_\omega;\RR) \ar@{=}[r] \ar[d] & H^2(X_\omega;\RR) \ar[r] \ar[d] & 0 \\
    & 0 & 0 } 
\end{aligned}
\end{equation}
with exact rows and columns. 
Note that we have $H^2_\cQ({\rm pt};\RR) \cong \bN^\vee \otimes \RR$, 
$H^2_T({\rm pt};\RR) \cong \RR^m$, $H^2_K({\rm pt};\RR) 
\cong \LL^\vee \otimes \RR$.  
The top row of \eqref{eq:H2_diagram} is identified with the exact sequence 
\eqref{eq:divseq} tensored with $\RR$. 
By \eqref{eq:Qequiv_cohomology}, $H^2_\cQ(X_\omega;\RR)$ 
is freely generated by the classes $u_i$, $i\in 
\{1,\dots,m\} \setminus S$ of toric divisors,
and hence $H^2_\cQ(X_\omega;\RR) \cong \RR^{m-|S|}$. 
The leftmost column is identified with the exact sequence
\[
\xymatrix{
0 \ar[r]& 
\bN^\vee \otimes \RR \ar[r]& 
\RR^{m-|S|} \ar[r]& 
\LL^\vee\otimes \RR\big/\sum_{i\in S} \RR D_i 
\ar[r] & 0
}
\]
induced by \eqref{eq:divseq}. 
In particular we have 
\[
H^2(X_\omega; \RR) \cong \LL^\vee \otimes \RR\big/
\textstyle \sum_{i\in S} \RR D_i  
\]
where the non-equivariant limit of $u_i$ is identified with 
the class of $D_i$. 
The homology group $H_2(X_\omega;\RR)$ is identified with 
$\bigcap_{i\in S} \Ker(D_i)$ in $\LL\otimes \RR$. 
The square at the upper left of \eqref{eq:H2_diagram} 
is a pushout and we have: 
\begin{align*} 
H^2_T(X_\omega;\RR) 
& \cong \bigoplus_{i\in \{1,\dots,m\} \setminus S} 
\RR u_i \oplus 
\bigoplus_{i=1}^m \RR \lambda_i \Big/ 
\big \langle \textstyle\sum_{i=1}^m \<\chi, b_i\> (u_i - \lambda_i)  
\colon \chi \in \bN^\vee \otimes \RR \big\rangle. 
\end{align*} 
It follows that the middle row of \eqref{eq:H2_diagram} 
splits canonically: we have a well-defined homomorphism\footnote
{More precisely $(-\theta)$ gives a splitting of the middle row 
of \eqref{eq:H2_diagram}.}  
\begin{equation} 
\label{eq:theta}
\theta \colon \LL^\vee \otimes \RR \cong H^2_K({\rm pt};\RR) 
\longrightarrow H^2_T(X_\omega;\RR)
\end{equation} 
such that $\theta(D_i) =u_i-\lambda_i$ and that 
\[
H^2_T(X_\omega;\RR) \cong H^2_\cQ(X_\omega;\RR) 
\oplus \theta(\LL^\vee\otimes \RR). 
\]
The class $\theta(p)$ can be written as the $T$-equivariant first 
Chern class of a certain line bundle $L(p)$ associated 
to $p$ (see \S \ref{sec:basis_K-theory}).  
One advantage of working with $T$-equivariant cohomology instead of $\cQ$-equivariant cohomology
is the existence of this canonical splitting. 

We also introduce a canonical splitting of the projection 
$\LL^\vee \otimes \RR \to \LL^\vee \otimes \RR/ 
\sum_{i\in S} \RR D_i \cong H^2(X_\omega,\RR)$. 
This is equivalent to choosing a complementary subspace 
of $H_2(X_\omega;\RR)$ in $\LL\otimes \RR$. 
Take $j\in S$. The corresponding extended vector $\overline{b_j} \in 
\bN\otimes \RR$ lies in the support of the fan. 
Let $\sigma_{I_j}\in \Sigma$, 
$I_j\subset \{1,\dots,m\} \setminus S$ be 
the minimal cone containing $\overline{b_j}$ and write 
$\overline{b_j} = \sum_{i\in I_j} c_{ij} \overline{b_i}$ 
for some $c_{ij} \in \RR_{>0}$. 
By the exact sequence \eqref{eq:exact}, there exists an 
element $\xi_j \in \LL\otimes \QQ$ such that 
\begin{equation} 
\label{eq:def_xi}
D_i \cdot \xi_j = \begin{cases} 
1 & \text{if }i=j; \\ 
-c_{ij} & \text{if } i \in I_j; \\ 
0 & \text{if } i\notin I_j \cup \{j\}. 
\end{cases} 
\end{equation} 
Note that one has $D_i \cdot \xi_j = \delta_{ij}$ for $i,j\in S$. 
Hence $\{\xi_i\}_{i\in S}$ spans a complementary subspace 
of $H_2(X_\omega;\RR) = 
\bigcap_{j\in S} \Ker(D_j) \subset \LL \otimes \RR$ 
and defines a splitting: 
\begin{equation} 
\label{eq:L_decomp} 
\LL\otimes \RR \cong H_2(X_\omega;\RR) \oplus \bigoplus_{j\in S} \RR 
\xi_j, 
\end{equation} 
or, for the dual space, 
\begin{equation} 
\label{eq:Lvee_decomp} 
\LL^\vee \otimes \RR \cong \bigcap_{j\in S} \Ker(\xi_j) \oplus 
\bigoplus_{j\in S} \RR D_j
\end{equation} 
with $\bigcap_{j\in S} \Ker(\xi_j) \cong H^2(X_\omega;\RR)$.  

The equivariant first Chern class of $TX_\omega$ is 
given by: 
\[
\rho = c_1^\cQ(T X_\omega) = c_1^T(T X_\omega) = 
\sum_{i\in \{1,\dots,m\} \setminus S} u_i. 
\]

\subsection{Ample Cone and Mori Cone} 
\label{sec:amplecone} 
Let $D'_i$ denote the image of $D_i$ in 
$\LL^\vee\otimes \RR/\sum_{i\in S} \RR D_i \cong H^2(X_\omega;\RR)$. 
This is the non-equivariant Poincar\'{e} dual of the toric 
divisor \eqref{eq:T-invariant_divisor}, that is, the non-equivariant limit of $u_i$. 
The cone of ample divisors of $X_\omega$ is given by 
\[
C'_{\omega}  = \bigcap_{I\in \cA_\omega'} \angle'_I 
\]
where $\cA_\omega'$ was introduced in equation 
\eqref{eq:Uomega_factorization} and 
$\angle'_I := \sum_{i\in I} \RR_{>0} D'_i$ is an 
open cone in $\LL^\vee \otimes \RR/\sum_{i\in S} \RR D_i$ 
(cf.~Notation \ref{not:cones}).  
Under the splitting \eqref{eq:Lvee_decomp} of $\LL^\vee \otimes \RR$, 
the extended ample cone $C_\omega$ defined in equation~\ref{eq:ext_amplecone} also splits 
\cite[Lemma 3.2]{Iritani}: 
\begin{equation} 
\label{eq:splitcone}
C_\omega \cong C'_\omega \times \left( \sum_{i\in S} \RR_{>0} D_i 
\right) \subset H^2(X_\omega;\RR) \times \bigoplus_{i\in S} \RR D_i.   
\end{equation} 
The Mori cone is the dual cone of $C'_\omega$: 
\[
\NE(X_\omega) = C'^\vee_\omega = 
\{ d\in H_2(X_\omega;\RR) : \text{$\eta \cdot d \ge 0$ for all $\eta \in C'_\omega$}\} 
\]

\subsection{Fixed Points and Isotropy Groups} 
\label{sec:fixed_points}

Fixed points of the $T$-action on $X_\omega$ are 
in one-to-one correspondence with minimal anticones, that is, with $\delta \in \cA_\omega$ such that $|\delta| = r$.  
A minimal anticone $\delta$ corresponds to the $T$-fixed point:
\[
\big\{(z_1,\ldots,z_n) \in U_\omega : \text{$z_i = 0$ if $i \not \in
  \delta$}\big\} \big/ K
\]
We now describe the isotropy of the Deligne--Mumford stack $X_\omega$, 
i.e.~those elements $g \in K$ such that the action of $g$ on $U_\omega$ 
has fixed points.  Recall that there are canonical isomorphisms 
$K \cong \LL \otimes \Cstar$ and $\Lie(K) \cong \LL \otimes \CC$, 
via which the exponential map $\Lie(K) \to K$ becomes 
$\id\otimes \exp(2 \pi \tti {-}) \colon \LL \otimes \CC 
\to \LL \otimes \Cstar$.  
The kernel of the exponential map is $\LL \subset \LL \otimes \CC$.  
Define $\KK \subset \LL \otimes \QQ$ to be the set 
of $f \in \LL \otimes \QQ$ such that:
\begin{equation}
  \label{eq:has_fixed_points}
 I_f :=  \Big\{ i \in \{1,2,\ldots,m\} : D_i \cdot f \in \ZZ \Big\}
  \in \cA_\omega
\end{equation}
The lattice $\LL$ acts on $\KK$ by translation, 
and elements $g \in K$ such that the action 
of $g$ on $U_\omega$ has fixed points correspond, 
via the exponential map, to elements of $\KK/\LL$.

\subsection{Floors, Ceilings, and Fractional Parts}

For a rational number $q$, we write:
\begin{align*}
& \text{$\lfloor q \rfloor$ for the largest integer $n$ such that $n \leq q$;} \\
& \text{$\lceil q \rceil$ for the smallest integer $n$ such that $q \leq n$; and} \\
& \text{$\< q \>$ for the fractional part $q - \lfloor q \rfloor$ of $q$.}
\end{align*}

\subsection{The Inertia Stack and Chen--Ruan Cohomology}
\label{sec:inertia} 

Recall the definition of the inertia stack $I X_\omega$ from~\S\ref{sec:conditions}. 
Components of $I X_\omega$ are indexed by elements of $\KK/\LL$: 
the component $X_\omega^f$ of $IX_\omega$ corresponding to 
$f \in \KK/\LL$ consists of the points $(x,g)$ in $I X_\omega$ 
such that $g = \exp(2\pi \tti f)$.  
Recall the set $I_f$ defined in \eqref{eq:has_fixed_points}.  
The component $X_\omega^{f}$ in the inertia stack 
$I X_\omega$ is the toric Deligne--Mumford stack 
with GIT data given by $K$, $\LL$, and $\omega$ 
exactly as for $X_\omega$, and characters 
$D_i \in \LL^\vee$ for $i \in I_f$. We have: 
\[
X_\omega^f = [\CC^{I_f} \cap U_\omega/K]. 
\]
The inclusion $\CC^{I_f} \subset \CC^m$ exhibits $X_\omega^f$ 
as a closed substack of the toric stack $X_\omega$. 
According to Borisov--Chen--Smith \cite{Borisov--Chen--Smith}, 
components of the inertia stack of $X_\omega$ are indexed 
by elements of the set $\cbox(X_\omega)$: 
\[
\cbox(X_\omega) = \left 
\{ v \in \bN : \overline{v} = 
\text{$\sum_{i\notin I} c_i \overline{b_i}$ in $\bN\otimes \RR$  for some $I\in \cA$ and $0\le c_i <1$} \right\}
\]
In fact, we have an isomorphism \cite[\S3.1.3]{Iritani}:
\begin{align} 
\label{eq:KL_Box}
\KK/\LL \cong \cbox(X_\omega) 
&& [f] \mapsto v_f = \sum_{i=1}^m 
\lceil - (D_i \cdot f) \rceil b_i \in \bN.  
\end{align}
When $j\in S$ and $b_j\in \cbox(X_\omega)$, the element 
$-\xi_j \in\LL\otimes \QQ$ defined in 
\eqref{eq:def_xi} belongs to $\KK$ and corresponds to $b_j$. 

The age $\iota_f$ of the component $X_\omega^f \subset I X_\omega$ 
is $\sum_{i\notin I_f}\langle D_i \cdot f\rangle$.   
The $T$-equivariant Chen--Ruan cohomology of $X_\omega$ is, 
as we saw in~\S\ref{sec:CR}, the $T$-equivariant cohomology of the inertia stack 
$IX_\omega$ with age-shifted grading:
\[
H_{\CR,T}^\bullet(X_\omega;\QQ) = \bigoplus_{f\in \KK/\LL} 
H_T^{\bullet-2\iota_f}\big(X_\omega^f;\QQ\big) 
\]
This contains the $T$-equivariant cohomology of $X_\omega$ 
as a summand, corresponding to the element $0 \in \KK/\LL$; 
furthermore the fact that each $X_\omega^f$ is a closed substack 
of $X_\omega$ implies that $H_{\CR,T}^\bullet(X_\omega;\QQ)$ 
is naturally a module over $H_{T}^\bullet(X_\omega;\QQ)$.  
We write $\fun_f$ for the unit class in $H_T^0\big(X_\omega^f;\QQ\big)$, 
regarded as an element of $H_{\CR,T}^{2\iota_f}(X_\omega;\QQ)$. 

Recall that the component $X_\omega^f$ of the inertia stack is 
the toric Deligne--Mumford stack with GIT data $(K;\LL;\omega;D_i : i \in I_f)$. 
In particular, therefore, the anticones for $X_\omega^f$ are given by 
$\{I \in \cA_\omega : I \subset I_f\}$.  
$T$-fixed points on the inertia stack $IX_\omega$ 
are indexed by pairs $(\delta,f)$ where $\delta$ 
is a minimal anticone in $\cA_\omega$, $f \in \KK/\LL$, 
and $D_i \cdot f \in \ZZ$ for all $i \in \delta$.  
The pair $(\delta,f)$ determines a $T$-fixed point 
on the component $X_\omega^f$ of the inertia stack: 
the $T$-fixed point that corresponds 
to the minimal anticone $\delta \subset I_f$.

\section{Wall-Crossing in Toric Gromov--Witten Theory}
\label{sec:wall-crossing}

In this section we consider crepant birational transformations $X_+ \dashrightarrow X_-$
between toric Deligne--Mumford stacks 
which arise from variation of GIT.  We use the Mirror Theorem for toric 
Deligne--Mumford stacks~\cite{CCIT,Cheong--Ciocan-Fontanine--Kim} to construct a global equivariant quantum connection over (a certain part of) the secondary toric variety for $X_\pm$; this gives an analytic continuation of the equivariant quantum connections for $X_+$ and $X_-$.

\subsection{Birational Transformations from Wall-Crossing} 
\label{sec:birational_transformations}

Recall that our GIT data in \S \ref{sec:GITdata} 
consist of a torus $K \cong (\Cstar)^r$, 
the lattice $\LL = \Hom(\Cstar,K)$ of 
$\Cstar$-subgroups of $K$, 
and characters $D_1,\ldots,D_m \in \LL^\vee$.  
Recall further that a choice of stability condition 
$\omega \in \LL^\vee \otimes \RR$ satisfying 
Assumption~\ref{assumption} determines a toric 
Deligne--Mumford stack $X_\omega = \big[U_\omega/K\big]$.  
The space $\LL^\vee \otimes \RR$ of stability conditions is divided into 
chambers by the closures of the sets $\angle_I$, $|I| = r-1$, 
and the Deligne--Mumford stack $X_\omega$ depends on $\omega$ 
only via the chamber containing $\omega$.  
For any stability condition $\omega$ satisfying Assumption~\ref{assumption}, 
the set $U_\omega$ contains the big torus $T=(\Cstar)^m$, 
and thus for any two such stability conditions $\omega_1$,~$\omega_2$ 
there is a canonical birational map 
$X_{\omega_1} \dashrightarrow X_{\omega_2}$, 
induced by the identity transformation between 
$T/K \subset X_{\omega_1}$ and 
$T/K \subset X_{\omega_2}$.  
Our setup is as follows.  Let $C_+$,~$C_-$ be chambers 
in $\LL^\vee \otimes \RR$ that are separated by a hyperplane wall $W$, 
so that $W \cap \overline{C_+}$ is a facet of $\overline{C_+}$, 
$W \cap \overline{C_-}$ is a facet of $\overline{C_-}$, 
and $W \cap \overline{C_+} = W \cap \overline{C_-}$.  
Choose stability conditions $\omega_+ \in C_+$, $\omega_- \in C_-$ 
satisfying Assumption~\ref{assumption} and set 
$X_+ := X_{\omega_+}$, $X_- := X_{\omega_-}$, and
\begin{align*}
  & \cA_\pm := \cA_{\omega_{\pm}} = 
\big\{ I \subset \{1,2,\ldots,m\} : 
\omega_\pm \in \angle_I \big\} 
\end{align*}
Then $C_\pm = \bigcap_{I\in \cA_\pm} \angle_I$. 
Let $\varphi \colon X_+ \dashrightarrow X_-$ be 
the birational transformation induced by the toric wall-crossing 
and suppose that 
\[
\sum_{i=1}^m D_i \in W
\]
As we will see below this amounts to requiring that $\varphi$ is crepant. 
Let $e \in \LL$ denote the primitive lattice vector in $W^\perp$ 
such that $e$ is positive on $C_+$ and negative on $C_-$. 

\begin{remark} 
The situation considered here is quite general.  We do not require $X_+$,~$X_-$ to have projective 
coarse moduli space (they are required to be semi-projective). 
We do not require that $X_+$,~$X_-$ are weak Fano, 
or that they satisfy the extended weak Fano condition in \cite[\S 3.1.4]{Iritani}. 
In other words, we do not require $\sum_{i=1}^m D_i\in W$ to lie in the 
boundary $W \cap \overline{C_+} = W \cap \overline{C_-}$ of 
the extended ample cones. 
\end{remark}

Choose $\omega_0$ from the relative interior of $W\cap \overline{C_+} 
= W \cap \overline{C_-}$. The stability condition $\omega_0$ 
does not satisfy our Assumption \ref{assumption}, but we can 
still consider: 
\begin{align*} 
\cA_0 & := \cA_{\omega_0} = 
\left\{ I \subset \{1,\dots,m\} : 
\omega_0 \in \angle_I \right\} 
\end{align*} 
and the corresponding toric (Artin) stack $X_0 := X_{\omega_0}
=[U_{\omega_0}/K]$ 
as given in Definition \ref{def:toricstack}. 
Here $X_0$ is not Deligne--Mumford, as 
the $\Cstar$-subgroup of $K$ 
corresponding to $e\in \LL$ (the defining equation 
of the wall $W$) has a fixed point in $U_{\omega_0}$. 
The stack $X_0$ contains both $X_+$ and $X_-$ as open substacks 
and the canonical line bundles of $X_{+}$ and $X_-$ 
are the restrictions of the same line bundle $L_0\to X_0$ 
given by the character $-\sum_{i=1}^m D_i$ of $K$. 
The condition $\sum_{i=1}^m D_i\in W$ 
ensures that $L_0$ comes from 
a $\QQ$-Cartier divisor on the underlying singular 
toric variety 
$\overline{X}_0 = \CC^m/\!\!/_{\omega_0} K$ 
associated to the fan $\Sigma_{\omega_0}$.  
On the other hand, in \S \ref{sec:FM}, we shall construct 
a toric Deligne-Mumford stack $\tX$ equipped 
with proper birational morphisms $f_\pm \colon 
\tX \to X_{\pm}$ 
such that the following diagram commutes: 
\begin{equation} 
\label{eq:crepant_diagram} 
\begin{aligned} 
\xymatrix{
 & \tX \ar[rd]^{f_-} \ar[ld]_{f_+} & \\ 
 X_+ \ar[rd]_{g_+} \ar@{-->}^{\varphi}[rr] &  & X_- \ar[ld]^{g_-} \\ 
 & \overline{X}_0 & 
}
\end{aligned} 
\end{equation} 
Then $f_+^\star(K_{X_+})$ and 
$f_-^\star(K_{X_-})$ coincide since they 
are the pull-backs of a $\QQ$-Cartier divisor on $\overline{X}_0$. 
This is what is meant by the birational map $\varphi$ being
\emph{crepant}\footnote
{This notion is also called $K$-equivalence: see the Introduction.}. 

Set: 
\begin{align*} 
M_{\pm} & = \{ i \in \{1,\dots,m\} : \pm D_i\cdot e >0 \}, \\ 
M_0  & = \{ i\in \{1,\dots,m\} : D_i \cdot e = 0\}. 
\end{align*} 
Our assumptions imply that both $M_+$ and $M_-$ are non-empty.  The following lemma is easy to check: 
\begin{lemma} 
\label{lem:Apm} 
Set: 
\begin{align*} 
\cA_0^{\rm thin} & := \{ I \in \cA_0 : I \subset M_0\} \\ 
\cA_0^{\rm thick} & := \{ I \in \cA_0: I \cap M_+ \neq \emptyset, 
I \cap M_- \neq \emptyset\}. 
\end{align*} 
Then one has $M_0\in \cA_0^{\rm thin}$ and 
\begin{align*} 
\cA_0 & = \cA_0^{\rm thin} \sqcup \cA_0^{\rm thick}, \\ 
\cA_\pm & = \cA_0^{\rm thick} \sqcup \left\{ I \sqcup J  : 
\emptyset \neq J \subset M_\pm, I \in \cA_0^{\rm thin}\right\}. 
\end{align*} 
\end{lemma}
\begin{remark} 
\label{rem:circuit} 
Let $\Sigma_\pm$ be the fans of $X_\pm$. 
In terms of fans, a toric wall-crossing can be described 
as a \emph{modification along a circuit} \cite{GKZ:book,Borisov--Horja:FM}, where `circuit' means a minimal linearly dependent set of vectors.  
In our wall-crossing, the relevant circuit is $\{b_i: i\in M_+ \cup M_-\}$: we have $\sum_{i\in M_+ \cup M_-} (D_i\cdot e) b_i=0$,
and every proper subset of $\{b_i : i\in M_+ \cup M_-\}$ is linearly 
independent. 
The partition of the circuit $M_+ \cup M_-$ into $M_+$ and $M_-$ 
is determined by the sign of the coefficients in 
a relation among $\{b_i : i\in M_+\cup M_-\}$. 
The modification along the circuit $M_+ \cup M_-$ 
turns the fan $\Sigma_+$ into $\Sigma_-$: it 
removes every cone $\sigma_I$ of $\Sigma_+$ such that 
$I$ contains $M_-$ but not $M_+$ and introduces 
cones of the form $\sigma_K$ where $K=(I\cup M_+) \setminus J$ for 
any non-empty subset $J \subset M_-$.  
This description matches with Lemma~\ref{lem:Apm},~\S\ref{sec:stacky_fan}, and~\S\ref{sec:GITdata}. 
\end{remark}

There are three types of possible crepant toric wall-crossings: 
(I) $X_+$ and $X_-$ are isomorphic in codimension one 
(``flop''), 
(II) $\varphi$ induces a morphism $X_{+} \to |X_{-}|$ 
or $X_- \to |X_+|$
contracting a divisor to a toric subvariety (``crepant resolution'') 
and 
(III) the rigidifications\footnote{See e.g.~\cite{FMN}.} $X_+^{\rm rig}$, $X_-^{\rm rig}$ 
are isomorphic 
(only the gerbe structures change; we call it a ``gerbe flop'').    
Define: 
\[
S_\pm = \{ i\in \{1,\dots,m\} : \overline{\{i\}}  
\not \in \cA_\pm\}. 
\]
\begin{proposition}
\label{prop:Spm_classif}
The intersection $S_0 := S_+ \cap S_-$ is contained in $M_0$. 
Moreover, one and only one of the following holds: 
\begin{itemize} 
\item[(I)] $S_+ = S_-$, $\sharp(M_+)\ge 2$ and $\sharp(M_-) \ge 2$;  
\item[(II-i)] there exists $i\in \{1,\dots,m\}$ such that 
$S_- = S_+ \sqcup \{i\}$, $M_- = \{i\}$ and $\sharp(M_+) \ge 2$; 
\item[(II-ii)] there exists $i\in \{1,\dots,m\}$ such that 
$S_+ = S_- \sqcup \{i\}$, $M_+ = \{i\}$ and $\sharp(M_-)\ge 2$; 
\item[(III)] there exist $i_+,i_- \in \{1,\dots,m\}$ such that 
$S_+ = S_0 \sqcup \{i_+\}$, 
$S_- = S_0 \sqcup \{i_-\}$, $M_+ = \{i_+\}$ 
and $M_- = \{i_-\}$.   
\end{itemize} 
\end{proposition} 

\begin{proof} 
First we show that $S_0\subset M_0$. 
Take $i \in S_0$. 
Suppose that $i\in M_+$. 
Since $M_0\in \cA_0^{\rm thin}$, we have 
$M_0 \cup M_-\in \cA_-$ by Lemma \ref{lem:Apm}. 
Thus $\overline{\{i\}} = 
M_0 \cup M_- \cup (M_+\setminus \{i\})$ also belongs 
to $\cA_-$. This contradicts the fact that $i\in S_-$. 
Thus we have $i\notin M_+$, and similarly that 
$i\notin M_-$. Hence $i\in M_0$. We have shown 
that $S_0 \subset M_0$. 

Next we claim that:
\begin{itemize} 
\item[(a)] if $S_-\setminus S_+$ is non-empty, 
then we have $\sharp(S_- \setminus S_+) = 1$ and 
$M_- = S_-\setminus S_+$;
\item[(b)] if $S_- \subset S_+$, then $\sharp(M_-) \ge 2$. 
\end{itemize} 
Take $i\in S_-\setminus S_+$. 
We have $\overline{\{i\}} \in \cA_+\setminus \cA_-$. 
Lemma \ref{lem:Apm} implies that an element of 
$\cA_+ \setminus \cA_-$ is of the form $I \sqcup J$ with 
$\emptyset \neq J \subset M_+$ and $I\subset M_0$, 
and in particular does not intersect with $M_-$. 
This implies that $\{i\} = M_-$. 
Therefore $S_-\setminus S_+ = M_-$ consists of 
only one element. This proves (a).   
Conversely, if $M_-=\{i\}$, 
it follows from Lemma \ref{lem:Apm} that 
$\overline{\{i\}}\in \cA_+\setminus \cA_-$ 
and thus $i \in S_- \setminus S_+$. 
This proves (b). 
The same claim holds if we exchange $+$ and $-$. 
It follows that one and only one of (I), (II-i), (II-ii), (III) happens. 
\end{proof} 

\begin{proposition} 
The loci of indeterminacy of $\varphi$ and $\varphi^{-1}$ are 
the toric substacks 
\begin{align*}
\bigcap_{j\in M_-} \{z_j=0\}\subset X_+ && \text{and} &&
\bigcap_{j\in M_+} \{z_j=0\}\subset X_-
\end{align*}
respectively.   With cases as in Proposition~\ref{prop:Spm_classif}, we have: 
\begin{itemize} 
\item[(I)] $X_+$ and $X_-$ are isomorphic 
in codimension one; 
\item[(II-i)] $\varphi$ induces a morphism 
$\varphi \colon X_+ \to |X_-|$ that contracts the 
divisor $\{z_i=0\}$ to the subvariety 
$\bigcap_{j\in M_+} \{z_j=0\}$; 
\item[(II-ii)] a statement similar to {\rm (II-i)} with $+$ and $-$ 
interchanged; 
\item[(III)] $\varphi$ induces an isomorphism 
$X_+^{\rm rig} \cong X_-^{\rm rig}$
between the rigidifications. 
\end{itemize} 
\end{proposition} 
\begin{proof} 
One can check that 
$U_{\omega_+} \cap U_{\omega_-} = 
U_{\omega_+} \setminus \bigcap_{i\in M_-} \{z_i=0\} 
= U_{\omega_-} \setminus \bigcap_{i\in M_+} \{z_i=0\}$ 
using Lemma \ref{lem:Apm}. 
The geometric picture in each case can be seen from 
the stacky fans: (I) the sets of one-dimensional 
cones are the same; 
(II-i) the fan $\Sigma_-$ is obtained by 
deleting the ray $\RR_{\ge 0} \overline{b}_i$ from $\Sigma_+$; 
$\sigma_{M_+} \in \Sigma_-$ is a minimal cone containing 
$\overline{b}_i$; 
$\varphi$ contracts the toric divisor $\{z_i=0\}$ 
to the closed subvariety associated with $\sigma_{M_+}$; (II-ii) similar;
(III) the stacky fan $\mathbf{\Sigma}_-$ is obtained 
from $\mathbf{\Sigma}_+$ by replacing $b_{i_-}$ with 
$b_{i_+}$; one has $(D_{i_+} \cdot e) b_{i_+} = 
- (D_{i_-} \cdot e) b_{i_-}$ by \eqref{eq:exact} 
and $D_{i_+} \cdot e + D_{i_-} \cdot e =0$; thus 
$b_{i_+}$ and $b_{i_-}$ differ only by a torsion element 
in $\bN$. 
\end{proof} 

\begin{example} \ 
  \begin{itemize}
  \item[(I)] Let $a_1,\dots,a_k, b_1,\dots,b_l$ be positive integers such that 
    $a_1+ \cdots + a_k = b_1+ \cdots + b_l$. 
    Consider the GIT data given by $\LL^\vee =\ZZ$, $D_1 = a_1,\dots,D_k = a_k$, 
    $D_{k+1} = -b_1,\dots,D_{k+l}= -b_l$. 
    If $k,l\ge 2$, we have a flop between 
    \begin{align*}
      X_+= \bigoplus_{i=1}^k \cO_{\PP(a_1,\dots,a_k)}(-b_i) 
      && \text{and} &&
      X_- = \bigoplus_{j=1}^l \cO_{\PP(b_1,\dots,b_l)}(-a_j). 
    \end{align*}

  \item[(II)] Consider the case where $l=1$ in (I). Setting $d = a_1+ \cdots + a_k = b_1$, we have that $X_+ = \cO_{\PP(a_1,\dots,a_k)}(-d)$ is a crepant (partial) resolution of $|X_-| = \CC^k/\bmu_d$ where $\bmu_d$ acts on $\CC^k$ by the weights $(a_1/d,\dots,a_k/d)$.

  \item[(III)] Consider the GIT data given by $\LL^\vee = \ZZ^2$, $D_1 = (1,0)$, $D_2 = (1,2)$, $D_3 = (0,2)$.  Take $\omega_+$ from the chamber $\{(x,y) : 0< y < 2x \}$ and $\omega_-$ from the chamber $\{(x,y) : 0< 2x < y \}$. Then we have a ``gerbe flop" between $X_+ = \PP(2,2)$ and $X_- = \PP^1 \times B\bmu_2$.  
  \end{itemize}
\end{example}

\subsection{Decompositions of Extended Ample Cones} 

Recall the decomposition \eqref{eq:Lvee_decomp} of the vector space 
$\LL^\vee \otimes \RR$ and the decomposition 
\eqref{eq:splitcone} of the extended ample cone. 
In the case at hand, we have two (possibly different) 
decompositions of $\LL^\vee \otimes \RR$ 
associated to the GIT quotients $X_+$ and $X_-$: 
\begin{equation} 
\label{eq:decomp_pm} 
\LL^\vee \otimes \RR = \bigcap_{j\in S_\pm} 
\Ker(\xi_j^\pm) \oplus \bigoplus_{j\in S_\pm} \RR D_j 
\end{equation} 
where elements $\xi_j^\pm \in \LL \otimes \RR$, $j\in S_\pm$ 
are as in \eqref{eq:def_xi} and 
$\bigcap_{j\in S_\pm} \Ker(\xi_i^+) \cong H^2(X_\pm,\RR)$. 
Under these decompositions, one has 
\[
C_\pm = C_\pm' \times \sum_{j\in S_\pm} \RR_{>0} D_j 
\]
where $C_\pm' \subset \bigcap_{i\in S_\pm} \Ker(\xi^\pm_i) \cong
H^2(X_\pm;\RR)$ is the ample cone of $X_\pm$. 
Let $\overline{C_W} := W \cap \overline{C_+} 
= W \cap \overline{C_-}$ be a common facet of $C_+$ and $C_-$, 
and write $C_W$ for the relative interior of $\overline{C_W}$. 
We now show that these decompositions of 
the cones $C_+$,~$C_-$ are compatible along the wall. 

\begin{proposition} 
\label{prop:wallcone_decomp}
We have $\xi_i^+|_W = \xi_i^-|_W$ for $i\in S_0 = S_+ \cap S_-$ 
and $\xi_i^\pm|_W = 0$ for $i\in S_\pm \setminus S_\mp$. 
Set $\xi_i^W = \xi_i^+|_W = \xi_i^-|_W \in \Hom(W,\RR)$. 
Then we have 
\[
W' := W \cap \bigcap_{i\in S_+} \Ker(\xi_i^+) = W \cap \bigcap_{i\in S_-} 
\Ker(\xi_i^-) = \bigcap_{i\in S_0} \Ker(\xi^W_i) 
\]
and so the decompositions \eqref{eq:decomp_pm} restrict to 
the same decomposition of $W$:  
\begin{equation} 
\label{eq:decomp_W} 
W = W' \oplus 
\bigoplus_{i\in S_0} \RR D_i. 
\end{equation} 
Under this decomposition of $W$, the cone $C_W$ decomposes as 
\[
C_W = C_W' \times \sum_{i\in S_0} \RR_{>0} D_i 
\]
for some cone $C_W'$ in $W'$. 
With cases as in Proposition \ref{prop:Spm_classif}, we have:
\begin{itemize} 
\item[(I)] $C_W'$ is a common facet of $C_+'$ and $C_-'$; 
\item[(II-i)] $C_W' = C_-'$, $C_W'$ is a facet of $C_+'$ 
and $C_- = C_W + \RR_{>0} D_i$;  
\item[(II-ii)] $C_W' = C_+'$, $C_W'$ is a facet of $C_-'$
and $C_+ = C_W + \RR_{>0} D_i$;  
\item[(III)] $C_W' = C_+' = C_-'$, 
$C_+ = C_W + \RR_{>0} D_{i_+}$ and 
$C_- = C_W + \RR_{>0} D_{i_-}$. 
\end{itemize} 
\end{proposition} 
\begin{proof} 
It suffices to show that $\xi_i^+|_W = \xi_i^-|_W$ 
for $i\in S_0$ and that $\xi_i^\pm|_W = 0$ for 
$i\in S_\pm \setminus S\mp$. The rest of the statements 
follow easily. 
Suppose that $i\in S_0$. 
Recall the definition of $\xi_i^\pm$ in \eqref{eq:def_xi}. 
Let $\sigma_{I} \in\Sigma_{+}$ be the minimal 
cone containing $\overline{b_i}$. 
Then $\overline{I} \in \cA_+$. 
If $\overline{I} \in \cA_-$, we have 
$\xi_i^+ = \xi_i^-$ by the definition of $\xi_i^\pm$. 
Suppose that $\overline{I}\notin \cA_-$. 
By Lemma \ref{lem:Apm}, $I$ contains $M_-$
but not $M_+$. We have a relation of the form:
\begin{equation} 
\label{eq:b_i_cone}
\overline{b_i} = \sum_{j\in I} c_j \overline{b_j} 
\end{equation} 
with $c_j>0$. 
By adding to the right-hand side of 
\eqref{eq:b_i_cone} 
a suitable positive multiple of the relation 
\[
\sum_{j \in M_+} (D_j \cdot e) \overline{b_j} - 
\sum_{j\in M_-} (-D_j\cdot e) \overline{b_j} =0 
\]
given by $e\in \LL$ via \eqref{eq:exact}, we obtain a 
relation of the form 
\[
\overline{b_i} = \sum_{j\in I'} c_j' \overline{b_j} 
\]
such that $c'_j>0$ and 
$I' = (I \cup M_+) \setminus J$ with $\emptyset \neq J\subset M_-$. 
Then $\overline{I'} \in \cA_-$ by Lemma \ref{lem:Apm} 
(see also Remark \ref{rem:circuit}). 
Note that $c_j = c_j'$ if $j \in I \cap M_0 = I' \cap M_0$. 
This implies that $D_j \cdot \xi_i^+ = D_j \cdot \xi_i^-$ 
for all $j\in M_0$.  Since $\{D_j: j\in M_0\}$ spans $W$, 
we have $\xi_i^+|_W = \xi_i^-|_W$. 

Now suppose that $i\in S_+ \setminus S_-$.  
Then $M_+ = \{i\}$ by Proposition \ref{prop:Spm_classif} 
and we have a relation $(D_i \cdot e) 
\overline{b_i} = \sum_{j\in M_-} (-D_j \cdot e) \overline{b_j}$ 
given by $e\in\LL$. This implies that $\overline{b_i}$ 
is contained in the cone $\sigma_{M_-}$ of $\Sigma_+$, 
and the definition of $\xi_i^+$ implies that 
$D_j \cdot \xi_i^+=0$ for all $j\in M_0$. 
Thus $\xi_i^+|_W= 0$. 
The case where $i\in S_- \setminus S_+$ is similar. 
\end{proof} 

\subsection{Global Extended K\"{a}hler Moduli} 
\label{sec:global_KM} 
Our next goal is to describe a global `moduli space' $\tcM$ and a flat connection over $\tcM$, together with two neighbourhoods in $\tcM$ such that the restriction of the flat connection to one of the neighbourhoods (respectively to the other neighbourhood) is isomorphic to the equivariant quantum connection for $X_+$ (respectively for $X_-$).  Thus the equivariant quantum connections for $X_+$ and $X_-$ can be analytically continued to each other.  Roughly speaking, the space $\tcM$ will be a covering of a neighbourhood of a certain curve in the \emph{secondary toric variety} for $X_\pm$; in this section we introduce notation for and local co-ordinates on this secondary toric variety.

The wall and chamber structure of $\LL^\vee \otimes \RR$ described in \S\ref{sec:birational_transformations} defines a fan in $\LL^\vee\otimes \RR$, called the \emph{secondary fan} or \emph{Gelfand--Kapranov--Zelevinsky (GKZ) fan}.  The toric variety associated to the GKZ fan is called the secondary toric variety.  We consider the subfan of the GKZ fan consisting of the cones $\overline{C_+}$, $\overline{C_-}$ and their faces, and consider the toric variety $\cM$ associated to this fan.  (Thus $\cM$ is an open subset of the secondary toric variety.)  In the context of mirror symmetry, $\cM$ arises as the moduli space of Landau--Ginzburg models mirror to $X_\pm$.  It contains the torus fixed points $P_+$ and $P_-$ associated to the cones $C_+$ and $C_-$, which are called the \emph{large radius limit points} for $X_+$ and $X_-$.  More precisely, because we want to impose only very weak convergence hypotheses on the equivariant quantum products for $X_\pm$, we restrict our attention to the formal neighbourhood of the torus-invariant curve $\cC \subset \cM$ connecting $P_+$ and $P_-$: $\cC$ is the closed toric subvariety associated to the cone $\overline{C_W} = W \cap \overline{C_+} = W \cap \overline{C_-}$.

Our secondary toric variety $\cM$ is covered by two open 
charts
\begin{align} 
\label{eq:charts_cM} 
\Spec \CC[C_+^\vee \cap \LL] && \text{and} &&
\Spec \CC[C_-^\vee \cap \LL] 
\end{align}
that are glued along $\Spec \CC[C_W^\vee \cap \LL]$. 
Since the cones $C_\pm$ are not necessarily simplicial, 
$\cM$ is in general singular.  
For our purpose, it is convenient to use a lattice structure 
different from $\LL$ and to work with a smooth cover 
$\cM_{\rm reg}$ of $\cM$. 
We will define the cover $\cM_{\rm reg}$ by choosing 
suitable co-ordinates. 
As in \S \ref{sec:fixed_points}, 
consider the subsets $\KK_\pm \subset \LL\otimes \QQ$: 
\[
\KK_\pm := \Big\{ f \in \LL \otimes \QQ : \big\{ i \in \{1,2,\ldots,m\} : 
D_i \cdot f \in \ZZ \big\}  \in \cA_\pm \Big\} 
\]
and define $\tLL_+$ (respectively $\tLL_-$) to be the free $\ZZ$-submodule of 
$\LL\otimes \QQ$ generated by $\KK_+$ (respectively by $\KK_-$). 
Note that $\tLL_+$ and $\tLL_-$ are overlattices of $\LL$. 
\begin{lemma} 
\label{lem:integrallattice_decomp} 
Set $\tLL_\pm^\vee = \Hom(\tLL_\pm,\ZZ)\subset \LL^\vee$. 
We have $D_j \in \tLL_\pm^\vee$ if $j\in S_\pm$.  
The decomposition \eqref{eq:decomp_pm} of $\LL^\vee \otimes \RR$ is 
compatible with the integral lattice $\tLL^\vee_\pm$: one has 
\begin{equation} 
\label{eq:lattice_tLvee_decomp} 
\tLL^\vee_\pm = 
\left(H^2(X_\pm;\RR) \cap \tLL^\vee_\pm \right)
\oplus \bigoplus_{j\in S_\pm} \ZZ D_j  
\end{equation} 
where we regard $H_2(X_\pm;\RR)$ as a subspace of 
$\LL^\vee\otimes \RR$ via the isomorphism $H^2(X_\pm;\RR) 
\cong \bigcap_{j\in S_\pm} \Ker(\xi_j^\pm)$. 
The lattices $\tLL^\vee_+$ and $\tLL^\vee_-$ are compatible 
along the wall; one has (see equation~\ref{eq:decomp_W}):
\begin{equation} 
\label{eq:compatible_int_lattice}
W \cap \tLL_+^\vee = W \cap \tLL_-^\vee
= (W' \cap \tLL_\pm^\vee) \oplus \bigoplus_{j\in S_0} \ZZ D_j.  
\end{equation} 
\end{lemma} 
\begin{proof} 
Equation \eqref{eq:lattice_tLvee_decomp} holds for both 
$X_+$ and $X_-$ and we omit the subscript $\pm$ 
in what follows.
Since every element in $\cA$ contains $S$, 
we have $D_j \cdot f \in \ZZ$ for all $j\in S$ and $f\in \KK$. 
This shows that $D_j\in \tLL^\vee$ for $j\in S$. 
Thus $\tLL^\vee \supset 
(H^2(X;\RR)\cap \tLL^\vee) \oplus \bigoplus_{j\in S} \ZZ D_j$. 
Conversely, for $v\in \tLL^\vee$, one has 
$v \cdot \xi_i \in \ZZ$ for 
all $i\in S$ because $\xi_i \in \KK$. 
Then 
$w = v- \sum_{i\in S} (v\cdot \xi_i) D_i$ lies in 
$\bigcap_{j\in S} \Ker(\xi_j)  \cap \tLL^\vee$ 
and $v = w + \sum_{i\in S} (v\cdot \xi_i) D_i$. 

Next we prove \eqref{eq:compatible_int_lattice}. 
First we claim that for every element $f\in \KK_+\setminus \KK_-$, 
there exists $\alpha \in \QQ$ such that $f + \alpha e \in \KK_-$. 
This follows easily from the definition of $\KK_\pm$ 
and Lemma \ref{lem:Apm}. 
It follows from the claim that for any $f\in \tLL_+$, there exists 
$\alpha \in \QQ$ such that $f + \alpha e \in \tLL_-$. 
Suppose that $v\in W \cap \tLL_-^\vee$. 
For any $f\in \tLL_+$, taking $\alpha\in \QQ$ as above, 
one has $v \cdot f = v \cdot (f + \alpha e) \in \ZZ$. 
Therefore $v \in W \cap \tLL_+^\vee$. 
This shows that $W \cap \tLL_-^\vee \subset 
W \cap \tLL_+^\vee$. The reverse inclusion follows similarly. 
The second equality in \eqref{eq:compatible_int_lattice} 
follows from \eqref{eq:lattice_tLvee_decomp} 
and Proposition \ref{prop:wallcone_decomp}. 
\end{proof} 
\begin{remark} 
We have $H_2(X_\pm;\RR) \cap \tLL_\pm 
= H_2(|X_\pm|;\ZZ)$. 
\end{remark} 

Set $\ell_\pm = \dim H^2(X_\pm;\RR) = r -\sharp(S_\pm)$ 
and $\ell = \dim W' = r-1 - \sharp(S_0)$. 
We have $\ell \le \min\{\ell_+,\ell_-\}$. 
With cases as in Proposition \ref{prop:Spm_classif}, we have: 
\begin{itemize} 
\item[(I)] $\ell_+ = \ell_-= \ell+1$; 
\item[(II-i)] $\ell_+ = \ell+1$, $\ell_-= \ell$; 
\item[(II-ii)] $\ell_- = \ell+1$, $\ell_+ = \ell$; 
\item[(III)] $\ell_+ = \ell_- = \ell$.  
\end{itemize} 
Using Lemma \ref{lem:integrallattice_decomp}, 
we can choose integral bases 
\begin{align} 
\label{eq:two_bases} 
\begin{split} 
& \{p_1^+,\dots,p_{\ell_+}^+\} \cup \{D_j : j\in S_+\}  
\subset \tLL^\vee_+ \\ 
& \{p_1^-,\dots,p_{\ell_-}^-\} \cup \{D_j : j\in S_-\} \subset \tLL^\vee_- 
\end{split} 
\end{align} 
of $\tLL_\pm^\vee$ such that 
\begin{itemize} 
\item $p_1^+,\dots, p_{\ell_+}^+$ lie in the nef cone 
$\overline{C'_+} \subset H^2(X_+;\RR)$; 
\item $p_1^-,\dots,p_{\ell_-}^-$ lie in the nef cone 
$\overline{C'_-} \subset H^2(X_-;\RR)$; 
\item $p_i^+ = p_i^- \in \overline{C_W'}$ for $i=1,\dots,\ell$. 
\end{itemize} 
These bases give co-ordinates on the toric charts 
\eqref{eq:charts_cM}. 
For $d\in \LL$, we write $\sfy^d$ for the corresponding 
element in the group ring $\CC[\LL]$. 
The homomorphisms 
\begin{align*} 
& \CC[C_+^\vee \cap \LL] \hookrightarrow 
\CC[y_1,\dots,y_{\ell_+}, \{x_j : j\in S_+\}], 
\qquad \sfy^d \mapsto \textstyle \prod_{i=1}^{\ell_+} y_i^{p_i^+ \cdot d} 
\cdot \prod_{j\in S_+} x_j^{D_j \cdot d} \\ 
& \CC[C_-^\vee \cap \LL] \hookrightarrow 
\CC[\ty_1,\dots,\ty_{\ell_-}, \{\tx_j : j\in S_-\}], 
\qquad \sfy^d \mapsto \textstyle \prod_{i=1}^{\ell_-} \ty_i^{p_i^- \cdot d} 
\cdot \prod_{j\in S_-} \tx_j^{D_j \cdot d}
\end{align*}  
define the two smooth co-ordinate charts 
\begin{align*}
(y_i, x_j : 1\le i\le \ell_+, j\in S_+) && \text{and} &&
(\ty_i,\tx_j: 1\le i\le \ell_-, j\in S_-)
\end{align*}
which are resolutions of (respectively) $\Spec \CC[C_+^\vee \cap \LL]$ and $\Spec \CC[C_-^\vee \cap \LL]$. 
We reorder the bases \eqref{eq:two_bases} 
\begin{align*} 
&  \{p_1^+,\dots,p_{\ell_+}^+\} \cup \{D_j : j\in S_+\}  
= \{\sfp_1^+,\dots,\sfp_{r-1}^+,\sfp_r^+\} \\ 
& \{p_1^-,\dots,p_{\ell_-}^-\} \cup \{D_j : j\in S_-\} 
= \{ \sfp_1^-,\dots,\sfp_{r-1}^-,\sfp_r^-\} 
\end{align*}  
in such a way that $\sfp_i^+ = \sfp^-_i\in W$ for $i=1,\dots,r-1$ 
and $\sfp_r^\pm$ is the unique vector (in each basis) 
that does not lie on the wall $W$. 
Let 
\begin{align*} 
\{y_i, x_j : 1\le i\le \ell_+, j\in S_+\} 
& = \{ \sfy_1,\dots,\sfy_r\}   \\
\{\ty_i, \tx_j : 1\le i \le \ell_-, j\in S_-\} 
&= \{\tsfy_1,\dots,\tsfy_r\}
\end{align*} 
be the corresponding reordering of the co-ordinates. 
Then the change of co-ordinates 
is of the form:  
\begin{align}
\label{eq:change_of_variables}
\tsfy_i =
\begin{cases}
  \sfy_i \sfy_r^{c_i} & 1\le i\le r-1 \\
  \sfy_r^{-c} & i=r
\end{cases}
\end{align} 
for some $c\in \QQ_{>0}$ and $c_i\in \QQ$. 
The numbers $c_i$, $c$ here arise from the transition 
matrix of the two bases \eqref{eq:two_bases}. 
We find a common denominator for $c$, $c_i$ and 
write $c = A/B$ and $c_i = A_i/B$, $1\le i\le r-1$ 
for some $A, B\in \ZZ_{>0}$ and $A_i\in \ZZ$. 
Then $\sfy_r^{1/B} = \tsfy_r^{-1/A}$. 
The smooth manifold $\cM_{\rm reg}$ is defined by 
gluing the two charts  
\begin{align*}
  U_+ = \Spec \CC[\sfy_1,\dots,\sfy_{r-1},\sfy_r^{1/B}] 
  && \text{and} &&
  U_- = \Spec \CC[\tsfy_1,\dots,\tsfy_{r-1},\tsfy_r^{1/A}]
\end{align*}
via the change of variables \eqref{eq:change_of_variables}.  
The large radius limit points $P_+\in U_+$ and $P_-\in U_-$ are given 
respectively by $\sfy_1= \cdots =\sfy_{r} =0$ and 
$\tsfy_1= \cdots = \tsfy_r=0$. 
Note that the last variables $\sfy_r$, $\tsfy_r$ 
correspond to the direction of $e\in \LL$: one has 
$\sfy^e = \sfy_r^{\sfp_r^+ \cdot e} = \tsfy_r^{\sfp_r^- \cdot e}$. 

The torus-invariant rational curve $\cC_{\rm reg} 
\subset \cM_{\rm reg}$ associated to $C_W$ is given by 
$\sfy_1= \cdots = \sfy_{r-1} =0$ on $U_+$ 
and by $\tsfy_1 = \cdots = \tsfy_{r-1}=0$ on $U_-$.  
Let $\hcM_{\rm reg}$ be the formal neighbourhood of 
$\cC_{\rm reg}$ in $\cM_{\rm reg}$. 
Since the global quantum connection is an analytic object, 
we need to work with a suitable analytification of $\hcM_{\rm reg}$: 
we include analytic functions in the last variable $\sfy_r$ 
in the structure sheaf 
and use the analytic topology on $\cC_{\rm reg} \cong \PP^1$. 
The underlying topological space of 
$\hcM_{\rm reg}$ is therefore $\PP^{1,\rm an}$; 
$\hcM_{\rm reg}$ is covered by two charts 
$\hU_+$ and $\hU_-$ with structure sheaves: 
\begin{equation} 
\label{eq:Upm_str_sheaf} 
\cO_{\hU_+} = \cO_{\CC_{+}}^{\rm an}[\![ \sfy_1,\dots,\sfy_{r-1}]\!] 
\quad \text{and} \quad
\cO_{\hU_-} = \cO_{\CC_{-}}^{\rm an}[\![ \tsfy_1,\dots,\tsfy_{r-1}]\!] 
\end{equation} 
where $\CC_{+}$ and $\CC_{-}$ denote the 
complex plane with co-ordinates $\sfy_r^{1/B}$ and 
$\tsfy_r^{1/A}$ respectively 
and the superscript ``an'' means analytic 
(space or structure sheaf). 
In other words, we can regard $\hcM_{\rm reg}$ as a sheaf 
of algebras over $\PP^{1,\rm an}$. 

The same construction works over an arbitrary $\CC$-algebra $R$. 
We define $\hcM_{\rm reg}(R)$ 
by replacing the structure sheaves in \eqref{eq:Upm_str_sheaf} 
with $(\cO_{\CC_{+}}^{\rm an} \otimes R)[\![\sfy_1,\dots,\sfy_{r-1}]\!]$ 
and $(\cO_{\CC_{-}}^{\rm an} \otimes R)[\![\tsfy_1,\dots,\tsfy_{r-1}]\!]$. 
In the equivariant theory, we use $R = R_T[z] = 
H^\bullet_T({\rm pt})\otimes \CC[z]$ for the ground ring. 
The global equivariant quantum connection will be defined over $R_T[z]$ and 
on (a formal thickening of) a simply-connected open subset of $\PP^{1,\rm an}$ 
containing $P_+$ and $P_-$. 

\begin{remark} 
\label{rem:overlattice} 
Taking an overlattice $\tLL_\pm$ of $\LL$ 
corresponds to taking a finite cover of $\cM$. 
This is necessary because the power series defining the $I$-function (see~\S\ref{sec:I_function})
is indexed by elements in $\tLL_\pm$. 
If one takes into consideration Galois symmetry \cite{Iritani}
of the quantum connection, one can see that the quantum connection 
(near $P_\pm$) 
descends to the secondary toric variety with respect 
to the original lattice $\LL$. 
\end{remark} 

\subsection{The $I$-Function}
\label{sec:I_function}

Recall Givental's Lagrangian cone introduced 
in Definition \ref{def:Lag_cone}. 
We consider the Givental cone $\cL_{X_\omega}$  
associated to the toric Deligne--Mumford stack $X_\omega$.
Under the decomposition \eqref{eq:L_decomp} of $\LL\otimes \RR$, 
we decompose $d\in \LL\otimes \RR$ as: 
\[
d = \overline{d} + \sum_{j\in S_\pm} (D_j \cdot d) \xi_j 
\]
where $\overline{d}$ is the $H_2(X_\omega;\RR)$-component of $d$. 
Define the $H_{\CR,T}^\bullet(X_\omega)$-valued hypergeometric series 
$I^{\rm temp}_\omega(\sigma,x,z) 
\in H^\bullet_{\CR,T}(X_\omega) 
\otimes_{R_T} R_T(\!(z^{-1})\!)[\![Q,\sigma,x]\!]$ by 
\begin{equation*} 
I^{\rm temp}_\omega(\sigma,x,z) := z e^{\sigma/z} 
\sum_{d\in \KK} e^{\sigma\cdot \overline{d}} 
Q^{\overline{d}} \prod_{j\in S} x_j^{D_j\cdot d}
\left( 
\prod_{j=1}^{m}
    \frac{\prod_{a : \<a\>=\< D_j\cdot d \>, a \leq 0}(u_j+a z)}
    {\prod_{a : \<a \>=\< D_j \cdot d \>, a\leq D_j \cdot d} (u_j+a z)} 
  \right) 
  \fun_{[{-d}]}
\end{equation*} 
where $\KK$ is introduced in \S \ref{sec:fixed_points}, 
$x =(x_j : j\in S)$ and $\sigma \in H^2_T(X_\omega)$ are 
variables, and $[-d]$ is the equivalence class 
of $-d$ in $\KK/\LL$ (recall from \S\ref{sec:inertia} 
that $\KK/\LL$ parametrizes inertia components).  The subscript `temp' reflects the fact that we are just about to change notation, by specializing certain parameters, and so this notation for the $I$-function is only temporary.
One can see that the summand of $I^{\rm temp}_\omega$ corresponding 
to $d\in \KK$ vanishes unless $d\in C_\omega^\vee$. 
Therefore the summation ranges over all $d\in \KK$ such that 
$\overline{d}$ lies in the Mori cone 
$\NE(X_\omega) = C_\omega'^\vee$ and $D_j \cdot d \ge 0$ 
for all $j\in S$.  
The Mirror Theorem for toric Deligne--Mumford stacks can be 
stated as follows: 
\begin{theorem}[\!\!\cite{CCIT,Cheong--Ciocan-Fontanine--Kim}] 
\label{thm:mirror_thm} 
$I^{\rm temp}_\omega(\sigma,x,-z)$ is an 
$S_T[\![Q,\sigma,x]\!]$-valued point on $\cL_{X_\omega}$. 
\end{theorem} 

We adapt the above theorem to the situation of toric wall-crossing. 
Let $I_\pm^{\rm temp}$ denote the $I$-function of $X_\pm$. 
We introduce a variant $I_\pm$ of the $I$-function which gives a 
cohomology-valued function on a neighbourhood of $P_\pm$ 
in $\hcM_{\rm reg}$. The $I$-function $I_\pm$ is obtained 
from $I_\pm^{\rm temp}$ by the following specialization: 
\begin{itemize} 
\item $Q=1$;
\item for $I_+$, $\sigma = \sigma_+ := 
\theta_+( \sum_{i=1}^r \sfp^+_i \log \sfy_i)
+ c_0(\lambda)$, 
\item  for $I_-$, $\sigma = \sigma_- := 
\theta_-( \sum_{i=1}^r \sfp_i^- \log \tsfy_i)
+ c_0(\lambda)$.  
\end{itemize} 
where 
$\theta_\pm \colon \LL^\vee\otimes \CC \to H^2_T(X_\pm;\CC)$ 
are the maps introduced in \eqref{eq:theta} and $c_0(\lambda) 
= \lambda_1+ \cdots + \lambda_m$. 
Note that we have 
\begin{equation} 
\label{eq:sigma_p}
\sigma_+ 
= \sum_{i=1}^{\ell_+} \theta_+(p_i^+) \log y_i  -  
\sum_{j\in S_{+}} \lambda_j \log x_j  + c_0(\lambda) 
\end{equation} 
since $\theta_+(D_j) = -\lambda_j$ for $j\in S_+$. 
More explicitly, one can write $I_+$ as: 
\begin{equation}
  \label{eq:I+}
  I_+(\sfy,z):=
z e^{\sigma_+/z} 
  \sum_{d \in \KK_+} 
  \sfy^d 
  \left(
    \prod_{j=1}^{m}
    \frac{\prod_{a : \<a\>=\< D_j\cdot d \>, a \leq 0}(u_j+a z)}
    {\prod_{a : \<a \>=\< D_j \cdot d \>, a\leq D_j \cdot d} (u_j+a z)} 
  \right) 
  \fun_{[{-d}]}
\end{equation}
where recall that $(\sfy_1,\dots,\sfy_r) =(y_i, x_j : 1\le i\le \ell_+, j\in S_+)$ 
are co-ordinates on 
$\hU_+\subset \hcM_{\rm reg}$ and that
\begin{align*} 
\sfy^d &= \sfy_1^{\sfp^+_1 \cdot d} \cdots \sfy_r^{\sfp^+_{r}\cdot d} 
= \textstyle \prod_{i=1}^{\ell_+} y_i^{p_i^+\cdot d} 
\prod_{j\in S_+} x_j^{D_j \cdot d}. 
\end{align*} 
The $I$-function $I_+$ belongs to the space: 
\[
I_+ \in H^\bullet_{\CR,T}(X_+) \otimes_{R_T} R_T[\log \sfy_1,\dots,\log \sfy_r]
(\!(z^{-1})\!)[\![\sfy_1,\dots,\sfy_r]\!]. 
\]
The series $e^{-\sigma_+/z} I_+(\sfy,z)$ is homogeneous of degree two 
with respect to the (age-shifted) grading on $H_{\CR,T}^\bullet(X_+)$ 
and the degrees for variables given by:
\begin{align}
\label{eq:deg_variable}
\deg z =2 && \text{and} &&
\sum_{i=1}^r (\deg \sfy_i) \sfp_i^+ = 2\sum_{i=1}^{m} D_i
\end{align}
Note that $\deg \sfy_r=0$ because 
$\sum_{i=1}^m D_i \in W$. 

\begin{remark} 
The extra factor $e^{c_0(\lambda)/z}$ in the $I$-function 
makes the mirror map compatible with Euler vector fields. 
\end{remark} 

We now show that $I_+(\sfy,z)$ is analytic 
in the last variable $\sfy_r$, so that it defines 
an analytic function on $\hcM_{\rm reg}$. 

\begin{lemma} 
\label{lem:I_analyticity} 
Expand the $I$-function as
\[
I_+(\sfy,z) = z e^{\sigma_+/z} 
\sum_{k_1=0}^\infty \cdots \sum_{k_{r-1} = 0}^\infty
\sum_{n\in \ZZ}  
I_{+; k_1,\dots,k_{r-1},n}(\sfy_r) 
\sfy_1^{k_1} \cdots \sfy_{r-1}^{k_{r-1}} z^{n} 
\]
Then each coefficient $I_{+;k_1,\dots,k_{r-1},n}(\sfy_r)$ is 
a convergent power series in $\sfy_r$ taking values in a homogeneous 
component of $H^\bullet_{\CR,T}(X_+)$. 
Moreover it can be analytically continued to the universal 
covering $\cU_+$ of the space $\{ \sfy_r \in \CC : 
\sfy_r^{\sfp^+_r \cdot e} \neq \frc\}$, where 
\begin{equation} 
\label{eq:conifold_point}
\frc = \prod_{j: D_j\cdot e \neq 0} (D_j \cdot e)^{D_j \cdot e} 
\end{equation} 
is the so-called ``conifold point''. 
\end{lemma} 
\begin{proof} 
The homogeneity of $I_{+;k_1,\dots,k_{r-1},n}(\sfy_r)$ 
follows from the homogeneity of the $I$-function mentioned 
above. Fix $d\in \KK_+$ such that $k_i = 
\sfp_i^+ \cdot d$ for all $i=1,\dots,r-1$ 
and let $f=[-d] \in \KK_+/\LL$ be the corresponding sector. 
The $f$-component of $I_{+;k_1,\dots,k_{r-1},n}(\sfy_r)$ 
is given by a certain coefficient (in front of some powers of $z^{-1}$)  
of the $z^{-1}$-expansion of the series: 
\[
\sfy_r^{\sfp_r^+ \cdot d} 
\sum_{k\in \ZZ} \sfy_r^{k(\sfp_r^+ \cdot e)} 
\left( 
\prod_{j=1}^m 
\frac{\prod_{a: \<a\> = \<d_j\>, a\le 0} (\frac{u_j}{z}+a)}{
\prod_{a: \<a\> = \<d_j\>, a\le d_j + k e_j} 
(\frac{u_j}{z} + a) } 
\right) \fun_f, 
\]
where $d_j = D_j\cdot d$ and $e_j = D_j \cdot e$.  
Since the natural restriction map 
$H^\bullet_{T}(X_f) \to H^\bullet_T(X_f^T)$ is injective, it 
suffices to check that the restriction of the above series to each 
$T$-fixed point is convergent and analytically continues to 
$\cU_+$. 
Consider a $T$-fixed point in $X_+^f$ corresponding 
to a minimal anticone $\delta\in \cA_+$ with $\delta \subset I_f$ 
(see equation~\ref{eq:has_fixed_points} for $I_f$). 
Let $\beta_j$ denote the restriction of $u_j/z$ to the fixed point 
and set $\sx = \sfy_r^{\sfp^+_r \cdot e}$. 
The restriction of the above series gives: 
\[
\Phi(\beta; \sx) = 
\sum_{k\in \ZZ} \sx^{k} 
\prod_{j=1}^m 
\frac{\prod_{a: \<a\> = \<d_j\>, a\le 0} (\beta_j+a)}{
\prod_{a: \<a\> = \<d_j\>, a\le d_j + k e_j} 
(\beta_j + a) }  
\]
ignoring the prefactor $\sfy_r^{\sfp^+_r \cdot d}$. 
Since $\beta_j = \<d_j\>=0$ for $j\in \delta$ 
and $\delta \cap M_+ \neq \emptyset$, 
it follows that the summand vanishes for $k\ll 0$. 
We now regard $\{\beta_j : j \notin \delta\}$ as 
small\footnote{Note that $\Phi(\beta;\sx)$ have poles at $\beta_j=-a$ 
for some $a>0$.} complex parameters and prove the analyticity of 
$\Phi(\beta;\sx)$. 
The conclusion follows by 
expanding $\Phi(\beta;\sx)$ in the parameters $\beta_j$ and 
evaluating the expansion in the $T$-equivariant cohomology of a point.  
By using the ratio test and the fact that $\sum_{i=1}^m e_i =0$, 
we see that the radius of convergence of $\Phi$ is positive. 
Moreover one sees that $\Phi$ satisfies the differential 
equation: 
\[
\left[ 
\prod_{j: e_j>0} \prod_{l=0}^{e_j-1} 
\left( e_j \sx \parfrac{}{\sx} + (d_j+\beta_j -l) 
\right) 
- \sx \prod_{j: e_j<0} \prod_{l=0}^{-e_j-1} 
\left( e_j \sx \parfrac{}{\sx} + (d_j + \beta_j - l) \right) 
\right] \Phi = 0 
\]
which has singularities only at $\sx = 0, \frc, \infty$. 
Thus $\Phi$ can be analytically continued to $\cU_+$.   
\end{proof} 

An entirely parallel statement holds for $I_-(\sfy,z)$. 

\begin{remark} 
\label{rem:I_analyticity} 
Set $I_{+;k_1,\dots,k_{r-1}}(\sfy_r,z) := \sum_{n\in \ZZ} 
I_{+;k_1,\dots,k_{r-1},n}(\sfy_r) z^n$. 
In the proof of Lemma \ref{lem:I_analyticity}, we 
observed that $\Phi(\beta;\sx)$ is analytic for sufficiently 
small $\beta_j$. 
Therefore, if we expand $I_{+;k_1,\dots,k_{r-1}}(\sfy_r,z) 
= \sum_{i=0}^N f_i(\lambda,z,\sfy_r) \phi_i$ for a 
suitable $R_T$-basis $\{\phi_i\}$ of $H_{\CR,T}^\bullet(X_+)$, 
the coefficient $f_i(\lambda,z,\sfy_r)$ is analytic 
on the region $\{(\lambda_1,\dots,\lambda_m, z, \sfy_r) 
\in \CC^m \times \Cstar\times \cU_+ : 
|\lambda_i| < \epsilon |z|\}$ for some $\epsilon>0$, 
where $\lambda_i$ are $T$-equivariant parameters. 
\end{remark} 

\subsection{Global Equivariant Quantum Connection} 
\label{sec:global_qconn} 
In this section we use the $I$-function $I_+$ 
to construct a global 
quantum connection on the universal cover 
\begin{equation*} 
\tcM_+ := 
\left( \left(\hU_+ \setminus 
\{\sfy^e =\frc\}\right) \bigr/ \bmu_B \right)\sptilde  
\end{equation*} 
where $\hU_+$ is the open chart \eqref{eq:Upm_str_sheaf} 
of $\hcM_{\rm reg}$ and $\sfy^e = \sfy_r^{\sfp^+_r \cdot e}$ 
is a co-ordinate on $\hU_+$; 
$\bmu_B$ acts on $\hU_+$ by deck transformations 
of $\sfy_r^{1/B} \mapsto \sfy_r$. 
As in Lemma \ref{lem:I_analyticity}, 
we denote by $\cU_+$ the universal cover of 
$\{\sfy_r \in \CC : \sfy_r^{\sfp^+_r \cdot e} \neq \frc\}$.  
The space $\cU_+$ is the underlying topological space of $\tcM_+$,  
and $\tcM_+$ is a formal thickening of $\cU_+$. 
In a neighbourhood of $P_+$, 
the global quantum connection that we will construct can be identified with the equivariant 
quantum connection of $X_+$. 
The main result in this section is: 

\begin{theorem} 
\label{thm:global_qconn} 
There exist the following data: 
\begin{itemize} 
\item an open subset $\cU_+^\circ\subset \cU_+$ such that 
$P_+ \in \cU_+^\circ$ and that the complement 
$\cU_+ \setminus \cU_+^\circ$ is a discrete set; 
we write $\tcM_+^\circ = \tcM_+|_{\cU_+^\circ}$;  

\item a trivial $H_{\CR,T}^\bullet(X_+)$-bundle 
$\bF^+$ over $\tcM_+^\circ(R_T[z])$: 
\[
\bF^+ = H_{\CR,T}^\bullet(X_+)\otimes_{R_T} 
(\cO_{\cU^\circ_+}\otimes R_T[z])[\![\sfy_1,\dots,\sfy_{r-1}]\!];  
\]
\item a flat connection 
$\bnabla^+ = d + z^{-1} \bA^+(\sfy)$ 
on $\bF^+$ of the form: 
\[
\bA^+(\sfy) = \sum_{i=1}^{\ell_+} B_i(\sfy) \frac{dy_i}{y_i} 
+ \sum_{j\in S_+} C_j(\sfy) dx_j 
- \sum_{j\in S_+} \lambda_j \frac{dx_j}{x_j} 
\]
with $B_i(\sfy), C_j(\sfy) 
\in 
\End(H_{\CR,T}^\bullet(X_+)) \otimes_{R_T} 
(\cO_{\cU^\circ_+}\otimes R_T)[\![\sfy_1,\dots,\sfy_{r-1}]\!]$; 
\item a vector field $\bE^+$ on $\tcM_+(R_T)$, called the Euler vector field, defined by:
\[
\bE^+ = \sum_{i=1}^m \lambda_i \parfrac{}{\lambda_i}
+ \sum_{i=1}^r \frac{1}{2}(\deg \sfy_i) \sfy_i 
\parfrac{}{\sfy_i}; 
\]
\item a mirror map $\tau_+ \colon \tcM_+(R_T) \to H_{\CR,T}^\bullet(X_+)$ 
of the form:
\begin{align*}
  \tau_+ = \sigma_+ + \ttau_+ &&&
  \ttau_+\in H^\bullet_{\CR,T}(X_+)\otimes_{R_T}  (\cO_{\cU^\circ_+}\otimes R_T)[\![\sfy_1,\dots,\sfy_{r-1}]\!] \\
  &&& \ttau_+|_{\sfy_1=\cdots = \sfy_r=0}= 0
\end{align*}
\end{itemize} 
such that $\bnabla^+$ equals the pull-back $\tau_+^*\nabla^+$ of the 
equivariant quantum connection $\nabla^+$ of $X_+$ by $\tau_+$, 
that is:
\begin{align*} 
B_i(\sfy) & = \sum_{k=0}^N \parfrac{\tau_+^k(\sfy)}
{\log y_i} (\phi_k \star_{\tau_+(\sfy)}) && 
1\le i\le \ell_+ \\
C_j(\sfy) & = 
\sum_{k=0}^N \parfrac{\ttau_+^k(\sfy)} 
{x_j} (\phi_k\star_{\tau_+(\sfy)}) &&
j\in S_+ 
\end{align*} 
and that the push-forward of $\bE^+$ by $\tau_+$ 
is the Euler vector field $\cE^+$  for $X_+$ defined in equation~\ref{eq:Euler_mu}. 
Moreover, there exists a global section $\Upsilon^+_0(\sfy,z)$ of 
$\bF^+$ such that 
\[
I_+(\sfy,z) = z L_+(\tau_+(\sfy),z)^{-1} \Upsilon^+_0(\sfy,z) 
\]
where $L_+(\tau,z)$ is the fundamental solution 
for the quantum connection of $X_+$ 
in Proposition~\ref{prop:fundsol}. 
\end{theorem} 

\begin{remark}
  Here the Novikov variables $Q$ in the quantum product and the fundamental solution have been specialized to $1$: see \S\ref{sec:specialization}.  
\end{remark}

\begin{remark} An entirely analogous result holds for $X_-$. 
\end{remark} 

\begin{remark} 
The data in Theorem \ref{thm:global_qconn} satisfy 
some compatibility equations. The connection matrices $B_i$, $C_i$ 
are self-adjoint with respect to the equivariant orbifold 
Poincar\'{e} pairing $(\cdot,\cdot)$.  Furthermore the grading operator 
$\bGr^+ = z \parfrac{}{z} + \bE^+ + \mu^+$ on $\bF^+$ 
(where $\mu^+$ is the grading operator on $H_{\CR,T}^\bullet(X_+)$ 
defined in equation~\ref{eq:Euler_mu}) satisfies 
$[ \bGr^+, \bnabla^+_v] = \bnabla^+_{[\bE^+,v]}$ 
for any vector field $v$. 
These properties are inherited from the quantum connection. 
\end{remark} 
\begin{remark}[{\cite[Remark 3.5]{Iritani}}]
By construction, 
the mirror map $\tau_+$ here depends on how much 
we have extended vectors $b_j, j\in S_+$ in the extended 
stacky fan. If we add sufficiently many extended vectors, 
we can make it submersive near $P_+$ and 
Theorem \ref{thm:global_qconn} gives an analytic 
continuation of the big quantum cohomology. 
In fact we have 
\[
\tau(\sfy) = c_0(\lambda) + \sum_{j\in S_+\setminus S_0} 
\lambda_j \log x_j + 
\sum_{i=1}^{\ell_+} \theta(p_i^+) \log y_i 
+ 
\sum_{j\in S_+} 
\alpha_j x_j + \text{higher order terms}. 
\]
Here $\alpha_j = \prod_{i\in I_j} u_i^{n_{ij}} \fun_{[-\xi_j]}$, 
where $\xi_j\in \KK_+$ is given in \eqref{eq:def_xi}, 
$I_j\subset \{1,\dots,m\} \setminus S_+$ is such that 
$\sigma_{I_j}$ contains $\overline{b_j}$, and 
$\overline{b_j} = \sum_{i\in I_j} (n_{ij} + \epsilon_{ij}) 
\overline{b_i}$ with 
$\epsilon_{ij} \in [0,1)$ and $n_{ij} \in \ZZ_{\ge 0}$. 
Note that $\fun_{[-\xi_j]}$ corresponds to the Box element 
$b_j - \sum_{i\in I_j} n_{ij} b_i\in \cbox(X_+)$. 
\end{remark} 

\begin{remark} 
The logarithmic singularity of $\bnabla^+$ 
along $\prod_{j\in S_+} x_j=0$ is not very important: 
this can be eliminated by shifting the mirror map 
$\tau$ by $\sum_{j\in S_+} \lambda_j \log x_j$;
see \eqref{eq:sigma_p}. 
\end{remark}

The rest of this section is devoted to the proof of 
Theorem \ref{thm:global_qconn}. 
First we recall how to compute the quantum connection 
of $X_+$ using the $I$-function (cf.~\cite{CCIT:applications}). 
By the Mirror Theorem~\ref{thm:mirror_thm},  
$I_+^{\rm temp}(\sigma,x,-z)$ is a point on the Givental
cone $\cL_+:= \cL_{X_+}$ for $X_+$.  
Recall from Remark \ref{rem:fundamentalsol_cone} 
that the cone $\cL_+$ is ruled by its tangent spaces 
(multiplied by $z$): 
\[
\cL_+ = \bigcup_{\tau \in H^\bullet_{\CR,T}(X_+)} z 
L_+(\tau,-z)^{-1} \cH_+. 
\]
This implies that one has: 
\[
I_+^{\rm temp}(\sigma, x, z) = z L_+(\tau,z)^{-1} \Upsilon^+_0
\]
for some $\tau= \tau(\sigma,x) \in H^\bullet_{\CR,T}(X_+) 
\otimes_{R_T} R_T[\![Q,\sigma,x]\!]$ and 
$\Upsilon^+_0 \in \cH_+[\![\sigma,x]\!] = H^\bullet_{\CR,T}(X_+) 
\otimes_{R_T} R_T[z][\![Q,\sigma,x]\!]$. 
The map $(\sigma,x) \mapsto \tau(\sigma,x)$ 
is called the \emph{mirror map}: this will be determined below. 
In Lemma~\ref{lem:I_derivatives} we will construct differential operators 
$P_i = P_i(z\partial)$, $i=0,\dots,N$ 
which depend polynomially on $z$ and on the vector fields $z \partial_v$, $v\in H^2_T(X_+)$,
and $z\partial_{x_j}$, $j\in S_+$, and which satisfy:
\begin{itemize} 
\item $\phi_i = z^{-1} P_i I_+^{\rm temp}|_{Q=\sigma=x=0}$, 
$0\le i\le N$ 
are independent of $z$; 
\item $\{\phi_i : 0\le i\le N\}$ form a basis of $H^*_{\CR,T}(X_+)$ 
over $R_T$;
\item $P_0 = 1$.
\end{itemize} 
Then:
\begin{equation} 
\label{eq:Birkhoff}
\begin{bmatrix} 
\vert &  & \vert \\ 
z^{-1}P_0 I^{\rm temp}_+ & \cdots & z^{-1} P_N I^{\rm temp}_+ \\ 
\vert &  & \vert  
\end{bmatrix} 
= L_+(\tau,z)^{-1} 
\begin{bmatrix} 
\vert & & \vert \\ 
\Upsilon^+_0 & \cdots & \Upsilon^+_N \\
\vert & & \vert 
\end{bmatrix} 
\end{equation} 
for $\Upsilon^+_i := P_i( z\tau^*\nabla) \Upsilon^+_0$. 
Here $\tau^*\nabla$ is the pull-back of the quantum connection 
of $X_+$ via the mirror map $\tau$, and 
we used the fact that one has $\partial_v \circ L_+(\tau,z)^{-1} 
= L_+(\tau,z)^{-1}\circ  (\tau^* \nabla)_v$ for any 
vector field $v$ on $(\sigma,x)$-space. 
Note that: 
\begin{itemize} 
\item $\Upsilon^+_i \in \cH_+[\![\sigma,x]\!]$ does not 
contain negative powers of $z$; 
\item $L_+(\tau,z)$ does not contain positive powers of $z$; and
\item $L_+(\tau,z)= \id + O(z^{-1})$. 
\end{itemize} 
Thus the right-hand side of \eqref{eq:Birkhoff} can be regarded 
as the \emph{Birkhoff factorization} of the left-hand side 
(see \cite{Pressley-Segal}), when 
we view both sides as elements in the loop group $LGL_{N+1}$ 
with $z$ the loop parameter.
The properties of $P_i$ listed above ensure that the left-hand side of \eqref{eq:Birkhoff}
is invertible at $Q=\sigma=x=0$, and that its Birkhoff factorization 
can be determined recursively in powers of $Q$, $\sigma$ and $x$ 
(see Lemma \ref{lem:formal_Birkhoff}). 
Thus the $I$-function determines $L_+(\tau,z)^{-1}$ as a function
of $(\sigma,x)$, via Birkhoff factorization. 
The mirror map $\tau=\tau(\sigma,x)$ is determined by the asymptotics 
\[
L_+(\tau,z)^{-1} 1 = 1 + \tau z^{-1} + O(z^{-2})
\]
and $L_+(\tau,z)^{-1}$ determines the pulled-back quantum 
connection $\tau^* \nabla$. 

We perform the above procedure globally on $\hcM_{\rm reg}$, 
using the $I$-function $I_+$ obtained from $I_+^{\rm temp}$ 
by the specialization $Q=1, \sigma= \sigma_+$. 
It will be convenient to assume the following condition. 

\begin{assumption} 
\label{assump:generation} 
The set $\bN \cap |\Sigma_+| = \{ v\in \bN : 
\overline{v} \in |\Sigma_+|\}$ of lattice points in the 
support $|\Sigma_+|$ of the fan is generated by 
$b_j, j=1,\dots,m$ as an additive monoid. 
\end{assumption} 

\begin{remark}
  Assumption~\ref{assump:generation}  is harmless: it can be always achieved by adding enough extended vectors to the extended stacky fan and in fact   Theorem \ref{thm:global_qconn} holds without this assumption (see Remark \ref{rem:removing_assumption}).
\end{remark}

Recall from \S\ref{sec:inertia} that $H_{\CR,T}^\bullet(X_+)$ 
is the direct sum of sectors $H_{T}^\bullet(X_+^f)$, $f\in \KK_+/\LL$ 
and recall from \S \ref{sec:equiv_coh} that each sector 
$H_T^\bullet(X_+^f)$ is generated by divisor classes. 
Thus we can take an $R_T$-basis of $H_{\CR,T}^\bullet(X_+)$ 
of the form: 
\begin{align*}
  \phi_{f,i} = F_{f,i}\left(\theta(p^+_1),\dots, \theta(p^+_{\ell_+})\right) \fun_f  
  && \text{$f\in \KK/\LL$,  $1\le i \le \dim H^\bullet(X_+^f)$}
\end{align*}
where $F_{f,i}(a_1,\dots,a_{\ell_+}) \in \CC[a_1,\dots,a_{\ell_+}]$ 
is a homogeneous polynomial. 
Recall from \S \ref{sec:inertia} that elements in $\KK_+/\LL$ 
are in one-to-one correspondence with elements in $\cbox(X_+)$. 
Let $v_f \in \cbox(X_+)$ be 
the element corresponding to $f\in \KK_+/\LL$. 
By Assumption \ref{assump:generation}, there exist
non-negative integers $n_{f,j}$, $j=1,\dots,m$,
such that 
\begin{equation} 
\label{eq:vf} 
v_f = \sum_{j=1}^m n_{f,j} b_j. 
\end{equation} 
On the other hand, taking a minimal cone $\sigma_f$ in $\Sigma_+$ 
containing $v_f$, we can write 
\[
\overline{v_f} = \sum_{j \notin S_+, \overline{b_j} \in \sigma_f} c_{f,j} 
\overline{b_j}  
\]
for some $c_{f,j} \in [0,1)$. 
We set $c_{f,j} = 0$ if $j\in S_+$ or $b_j \notin\sigma_f$. 
Then $\sum_{j=1}^m (n_{f,j}-c_{f,j}) \overline{b_j}=0$ 
and by \eqref{eq:exact}, 
there exists an element $d_f \in \LL\otimes \QQ$ such 
that $D_j \cdot d_f = n_{f,j} - c_{f,j}$. By definition 
of $\KK_+$, $d_f \in \KK_+$ and $[-d_f] = f$ in 
$\KK_+/\LL$ by \eqref{eq:KL_Box} and \eqref{eq:vf}. 
Set $D_j = \sum_{a=1}^{r} \mu_{ja} \sfp^+_a$  
for some $\mu_{ja} \in \ZZ$. 
Define differential operators $\calD_j$, $\Delta_f$ as 
\begin{align*} 
\calD_j & := \sum_{a=1}^{r} \mu_{ja} z \sfy_a \parfrac{}{\sfy_a} \\ 
\Delta_f &:= \sfy^{-d_f} 
\prod_{j=1}^m \prod_{\nu=0}^{n_{f,j}-1} 
\left (\calD_j + \lambda_j  - \nu z \right).  
\end{align*}

\begin{lemma} 
\label{lem:I_derivatives}
Let $F_{f,i}$,~$\phi_{f,i}$,~$\Delta_f$ be as above. 
Define the differential operator $P^+_{f,i}$ by 
\[
P^+_{f,i}:= F_{f,i}\left(
zy_1\parfrac{}{y_1},\dots, 
z y_{\ell_+} \parfrac{}{y_{\ell_+}} \right) 
\Delta_f. 
\]
Then we have: 
\[
P^+_{f,i} I_+(\sfy,z) = z e^{\sigma_+/z} 
(\phi_{f,i} + O(\sfy)).  
\]
\end{lemma} 
\begin{proof} 
The proof is parallel to \cite[Lemma 4.7]{Iritani}. 
Note that the vector field $z y_i \partial/\partial y_i$ applied 
to $e^{\sigma_+/z}$ yields the factor $\theta(p_i^+)$ 
for $1\le i\le \ell_+$ (see equation~\ref{eq:sigma_p}). 
Thus it suffices to show that $\Delta_f I_+ = z e^{\sigma_+/z} 
(\fun_{f} + O(\sfy))$. 
Note that we have
\begin{align*} 
\calD_j \left( e^{\sigma_+/z} \sfy^d\right) 
& = \left( u_j - \lambda_j  
+ z  (D_i \cdot d) \right) e^{\sigma_+/z} \sfy^d
\end{align*} 
where we use $\sum_{a=1}^{\ell_+} \mu_{ja} \theta(\sfp_a^+) 
= \theta(D_j) = u_j -\lambda_j$. 
Therefore one has: 
\[
\Delta_f I_+ = e^{\sigma_+/z} 
\sum_{d\in \KK_+} \sfy^{d-d_f}  
\prod_{j=1}^m 
\frac{\prod_{a\le 0, \<a\> = \<D_j \cdot d\>} (u_j + az)}
{\prod_{a\le D_j\cdot d - n_{f,j}, \<a\> = \<D_j\cdot d\>} 
(u_j + az)}  
\fun_{[-d]}
\]
Note that the summand equals $\fun_f$ when $d= d_f$. 
We claim that the summand for $d\in \KK_+$ vanishes 
if $d-d_f$ does not lie in the dual cone $C_+^\vee$ of $C_+$. 
Suppose that $d-d_f\notin C_+^\vee$. 
Note that the summand for $d$ contains the factor 
$\prod_{j: D_j\cdot d\in \ZZ, D_j \cdot (d-d_f)<0} u_j$. 
By the description \eqref{eq:Qequiv_cohomology} 
of $H^\bullet_T(X^{[-d]}_+)$, it vanishes in cohomology if 
$I = \{j : D_j\cdot d\in \ZZ, D_j\cdot (d-d_f)\ge 0\}\notin \cA_+$. 
It now suffices to check $I\notin \cA_+$. 
If $I\in \cA_+$, one has $C_+ \subset \angle_I$.  
But $d-d_f$ lies in the dual cone of $\angle_I$. 
Thus $d-d_f \in C_+^\vee$: this is a contradiction. 
The claim follows and the Lemma is proved. 
\end{proof} 

Applying the differential operators $P^+_{f,i}$, $f\in \KK_+/\LL$, 
$1\le i\le \dim H^\bullet(X_+^f)$, to $I_+$, 
we obtain a matrix of the form: 
\begin{equation} 
\label{eq:I_derivative_matrix}
\begin{bmatrix} 
 & \vert & \\ 
\cdots & z^{-1} P^+_{f,i} I_+ & \cdots \\ 
& \vert &  
\end{bmatrix} 
= e^{\sigma_+/z} \II_+(\sfy,z)
\end{equation} 
where $I_+$ is regarded as a column vector 
written in the basis $\{\phi_{f,i}\}$ of $H^\bullet_{\CR,T}(X_+)$ 
and $\II_+(\sfy,z)= \id + O(\sfy)$ is a square matrix. 
We may also view $\II_+(\sfy,z)$ as an 
$\End(H^\bullet_{\CR}(X_+))$-valued function 
via the basis $\{\phi_{f,i}\}$. 
By the homogeneity of $e^{-\sigma_+/z}I_+$ and $P^+_{f,i}$, 
we find that the endomorphism $\II_+(\sfy,z)$ is homogeneous 
of degree-zero with respect to the degree \eqref{eq:deg_variable} 
of variables and the grading on $H^\bullet_{\CR}(X_+)$, i.e.~that: 
\begin{equation} 
\label{eq:I+homogeneous}
\left(z \parfrac{}{z} + \bE^+ + \ad(\mu^+) \right) \II_+(\sfy,z) = 0
\end{equation} 
As in \eqref{eq:Birkhoff}, we consider the Birkhoff factorization 
of \eqref{eq:I_derivative_matrix}. 
Since $e^{\sigma_+/z} = \id + O(z^{-1})$, 
it suffices to consider the Birkhoff factorization of $\II_+(\sfy,z)$. 
Set: 
\[
\gamma(\sfy_r,z) := \II_+(\sfy,z)
\Big|_{\sfy_1 = \cdots = \sfy_{r-1}=0, \lambda_1 = \cdots 
= \lambda_m =0}. 
\]
By Lemma \ref{lem:I_analyticity}, $z\mapsto \gamma(\sfy_r,z)$ 
is a loop in $\End(H_{\CR}^\bullet(X_+))$ 
that depends analytically on $\sfy_r\in \cU_+$. 
We first consider the Birkhoff factorization of $\gamma(\sfy_r,z)$. 
Since $\gamma(\sfy_r,z)$ is homogeneous, 
it is a Laurent polynomial in $z$ and both factors  
of the Birkhoff factorization $\gamma(\sfy_r,z) = \gamma_-(z) 
\gamma_+(z)$ are also homogeneous if the factorization exists. 
Therefore the Birkhoff factorization is equivalent to 
the \emph{block LU decomposition} of $\gamma(\sfy_r,1)$: 
\begin{align*}
\gamma_-(1) = 
\left [
\begin{array}{cccc} 
I_{r_1} & & & \\    
 * & I_{r_2} & &\hsymb{0}  \\ 
  \vdots & \vdots & \ddots &  \\ 
*&  *   & \cdots & I_{r_k} 
\end{array}  
\right ]
&&
\gamma_+(1) = 
\left [
\begin{array}{cccc} 
* & *& \cdots & * \\    
  & * & \cdots &  *\\ 
  &  & \ddots & \vdots  \\ 
\hsymb{0} &    &  & *
\end{array}  
\right] 
\end{align*}
where each block corresponds to a homogeneous component of  
$H^\bullet_{\CR}(X_+)$ and $I_r$ denotes the identity 
matrix of size $r$. 
The block LU decomposition of $\gamma(\sfy_r,1)$ exists if and only if 
\[
H = (\gamma(\sfy_r,1) H^{\le p}) \oplus H^{>p} 
\] 
holds for all $p\in \QQ$, 
where $H = H^\bullet_{\CR}(X_+)$ and $H^{\le p}$ 
(resp.~$H^{>p}$) denotes the subspace of degree 
less than or equal to $p$ (resp.~greater than $p$). 
This is a Zariski open condition for $\gamma(\sfy_r,1)$. 
Since $\gamma(\sfy_r=0,1) = \id$, it follows that $\gamma(\sfy_r,z)$ 
admits a Birkhoff factorization on the complement  
$\cU^\circ_+$ of a discrete set in $\cU_+$. 
Clearly one has $P_+ \in \cU_+^\circ$. 

\begin{lemma} 
\label{lem:formal_Birkhoff} 
Let $\gamma(z)\in LGL_{N+1}(\CC)$ be a Laurent polynomial 
loop admitting a Birkhoff factorization $\gamma = \gamma_- \gamma_+$. 
Let $\Gamma(s,z) \in \End(\CC^{N+1}) \otimes 
\CC[z,z^{-1}][\![s_1,\dots,s_l]\!]$ be a formal loop 
such that $\Gamma|_{s=0} = \gamma$. 
Then $\Gamma(s,z)$ admits a unique Birkhoff factorization 
of the form 
\[
\Gamma(s,z) = \Gamma_-(s,z) \Gamma_+(s,z) 
\]
such that $\Gamma_-(s,z) \in \End(\CC^{N+1}) 
\otimes \CC[z^{-1}][\![s_1,\dots,s_l]\!]$, 
$\Gamma_-(s,\infty) = \id$ 
and $\Gamma_+(s,z) \in \End(\CC^{N+1}) 
\otimes \CC[z][\![s_1,\dots,s_l]\!]$.  
\end{lemma} 
\begin{proof} 
It suffices to show that $\Gamma' =
\gamma_-^{-1} \Gamma \gamma_+^{-1}$ 
admits a Birkhoff factorization $\Gamma' = \Gamma'_-\Gamma'_+$. 
Expanding $\Gamma'$ and $\Gamma'_\pm$ in power series in 
$s_1,\dots,s_l$, one can determine the coefficients 
recursively from the equation $\Gamma' = \Gamma'_- \Gamma'_+$. 
\end{proof} 

Applying the above lemma to $\II_+(\sfy,z)$, we see that 
$\II_+(\sfy,z)$ with $\sfy_r \in \cU^\circ_+$ admits a 
Birkhoff factorization 
\begin{equation} 
\label{eq:Birkhoff_I}
\II_+(\sfy,z) = \bL_+(\sfy,z)^{-1} \Upsilon^+(\sfy,z) 
\end{equation} 
where 
\begin{align*} 
& \bL_+(\sfy,z) \in \End(H_{\CR}^\bullet(X_+))
\otimes \cO_{\cU^\circ_+}[z^{-1}][\![\lambda_1,\dots,\lambda_m,
\sfy_1,\dots,\sfy_{r-1}]\!], \\ 
&\Upsilon^+(\sfy,z) \in \End(H_{\CR}^\bullet(X_+)) 
\otimes \cO_{\cU^\circ_+}[z][\![\lambda_1,\dots,\lambda_m, 
\sfy_1,\dots,\sfy_{r-1}]\!]  
\end{align*} 
and $\bL_+(\sfy,\infty) = \id$. 
Using the homogeneity equation \eqref{eq:I+homogeneous}, 
we find that the Birkhoff factors $\bL_+$, $\Upsilon^+$ 
are also homogeneous of degree zero. 
Also the chosen $R_T$-basis $\{\phi_{f,i}\}$ of $H_{\CR,T}^\bullet(X_+)$ 
defines a splitting $H_{\CR,T}^\bullet(X_+)  \cong 
H_{\CR}^\bullet(X_+) \otimes R_T$, and by the splitting, 
one may naturally regard $\bL_+$, $\Upsilon^+$ as 
$\End(H_{\CR,T}^\bullet(X_+))$-valued functions. 
It follows that: 
\begin{align*} 
& \bL_+(\sfy,z) \in \End(H_{\CR,T}^\bullet(X_+))
\otimes_{R_T} 
(\cO_{\cU^\circ_+}\otimes R_T)[\![z^{-1}]\!] 
[\![\sfy_1,\dots,\sfy_{r-1}]\!], \\ 
&\Upsilon^+(\sfy,z) \in \End(H_{\CR,T}^\bullet(X_+)) 
\otimes_{R_T} (\cO_{\cU^\circ_+}\otimes R_T)[z]
[\![ \sfy_1,\dots,\sfy_{r-1}]\!]. 
\end{align*} 
Comparing \eqref{eq:Birkhoff_I} with \eqref{eq:Birkhoff}, 
we obtain 
\begin{equation} 
\label{eq:fundsol_mirror} 
L_+(\tau_+(\sfy),z)^{-1}\bigl|_{Q=1} = e^{\sigma_+/z} 
\bL_+(\sfy,z)^{-1}. 
\end{equation} 
The mirror map $\tau_+(\sfy)$ is given by 
$\tau_+(\sfy) = \sigma_+ + \ttau_+(\sfy)$ with 
$\ttau_+(\sfy)$ determined by: 
\begin{align*} 
\bL_+(\sfy,z)^{-1} 1 = 1 + \ttau_+(\sfy) z^{-1}  + O(z^{-2}).   
\end{align*} 
We have $\ttau_+(0)=0$ and $\ttau_+(\sfy) \in 
H_{\CR,T}^\bullet(X_+) \otimes 
(\cO_{\cU^\circ_+}\otimes R_T)[\![\sfy_1,\dots,\sfy_{r-1}]\!]$. 
The first column of \eqref{eq:Birkhoff_I} gives 
$I_+(\sfy,z) =z L(\tau(\sfy),z)|_{Q=1} \Upsilon^+_0$, where 
$\Upsilon^+_0$ is the first column of $\Upsilon^+$. 
(Here we assume that the first column corresponds to 
the basis vector 
$\phi_{0,1} = 1$ and the differential operator 
$P^+_{0,1} =1$.)  
\begin{remark} 
Equation \eqref{eq:fundsol_mirror} is an equality in the ring: 
\[
\End(H_{\CR,T}^\bullet(X_+)) \otimes_{R_T} 
R_T[\log \sfy_1,\dots, \log \sfy_r][\![z^{-1}]\!][\![\sfy]\!]. 
\] 
Note that the substitution $\tau=\tau_+(\sfy)$ in $L_+|_{Q=1}$ 
makes sense: see \S \ref{sec:specialization}. 
\end{remark} 

By equation \eqref{eq:fundsol_mirror}, $\bL_+(\sfy,z)$ determines 
the quantum connection pulled-back by the mirror map $\tau_+(\sfy)$. 
Set $\tau_+^* \nabla^+ = d + z^{-1} \bA^+(\sfy)$. 
The connection 1-form $\bA^+(\sfy)$ is computed by:  
\begin{align*} 
\bA^+(\sfy) & := - 
z d (\bL_+(\sfy,z) e^{-\sigma_+/z}) e^{\sigma_+/z} \bL_+(\sfy,z)^{-1} \\ 
& = -z (d\bL_+(\sfy,z)) \bL_+(\sfy,z)^{-1} 
+ \bL_+(\sfy,z) (d \sigma_+) \bL_+(\sfy,z)^{-1}
\end{align*} 
where the term $d\sigma_+$ gives a logarithmic singularity 
(see equation~\ref{eq:sigma_p}): 
\[
d\sigma_+ = \sum_{i=1}^{\ell_+} \theta_+(p_i^+) \frac{dy_i}{y_i} 
- \sum_{j\in S_+} \lambda_j \frac{dx_j}{x_j}. 
\]
Thus the connection form $\bA^+(\sfy)$ is a global $1$-form on $\tcM_+$ 
satisfying the properties in Theorem \ref{thm:global_qconn}.  

\begin{remark} 
Note that $\bA^+(\sfy)$ is independent of $z$: in the formal 
neighbourhood of $P_+ = \{\sfy_1=\cdots=\sfy_r=0\}$ this follows from 
the fact that $d+ z^{-1} \bA^+(\sfy)$ is the pulled-back 
quantum connection, 
and this is true everywhere by analytic continuation. 
\end{remark} 

Finally we see that $\bE^+$ corresponds to $\cE^+$. 
Choose a homogeneous $R_T$-basis $\{\phi_i\}$ of $H_{\CR,T}^\bullet(X_+)$ 
such that $\phi_0 = 1$ and $\phi_i = \theta(p_i^+)$ 
for $1\le i\le \ell_+$ and write $\tau^i_+(\sfy)$ for the $i$th component of $\tau_+(\sfy)$ with respect to this basis. 
One needs to check that $\bE^+ \tau^i_+(\sfy) = (1- \frac{1}{2} \deg \phi_i) 
\tau^i_+(\sfy) + \rho^i$, where $\rho = \sum_{i=0}^N \rho^i 
\phi_i$. 
The homogeneity of $\bL_+^{-1}$ shows that $\ttau_+(\sfy)$ 
is homogeneous of (real) degree two: this implies that
$\bE^+ \ttau_+^i(\sfy) = (1 - \frac{1}{2} \deg \phi_i) 
\ttau_+^i(\sfy)$. 
If we set $\sigma_+ = \sum_{i=0}^{\ell_+} 
\sigma_+^i \phi_i$, we have:
\[
\bE^+ \sigma_+^i = 
\begin{cases} 
c_0(\lambda) 
- \sum_{j\in S_+} (\lambda_j \log x_j + \lambda_j \frac{1}{2} (\deg x_j) ) 
& i=0 \\ 
\frac{1}{2} \deg y_i & 1\le i\le \ell_+
\end{cases} 
\]
Thus we have: 
\[
\bE^+ \tau^i_+(\sfy) 
= \begin{cases}
\tau^0_+(\sfy) + c_0(\lambda)
- \sum_{j\in S_+} \lambda_j \frac{1}{2} (\deg x_j) 
& i=0 \\ 
\frac{1}{2} (\deg y_i) & 1\le i \le \ell_+ \\ 
\left(1 - \frac{1}{2} \deg \phi_i\right) \tau^i_+(\sfy)  
& i > \ell_+
\end{cases}
\]
On the other hand, we have 
\begin{align*} 
\rho & = u_1 + \cdots + u_m = \theta(D_1 + \cdots + D_m) + \lambda_1 + 
\cdots + \lambda_m \\ 
& = \sum_{i=1}^{\ell_+} \frac{1}{2}(\deg y_i) \theta(p_i^+) 
+  \sum_{j\in S_+} \frac{1}{2} (\deg x_j) (-\lambda_j) 
+ \lambda_1 + \cdots + \lambda_m 
\end{align*} 
and therefore: 
\[
\rho^i = \begin{cases} 
c_0(\lambda) - \sum_{j\in S_+} \lambda_j \frac{1}{2} (\deg x_j)& i=0 \\ 
\frac{1}{2} (\deg y_i) & 1\le i\le \ell_+ 
\end{cases} 
\]
Thus $\bE^+ \tau^i_+(\sfy) = (1- \frac{1}{2} \deg \phi_i) 
\tau^i_+(\sfy) + \rho^i$. 
The proof of Theorem \ref{thm:global_qconn} is complete. 

\begin{remark} 
\label{rem:removing_assumption}
For the existence of a global quantum connection 
in Theorem \ref{thm:global_qconn} and other main results in this paper, 
we do not need Assumption \ref{assump:generation}. 
Let us write $\tcM_+^S$ for $\tcM_+$ to emphasize the 
dependence on the extension data $S$. 
Then one has: 
\[
S \subset S' \ \Longrightarrow \ 
\tcM_{+}^S\subset \tcM_{+}^{S'}
\]
Suppose that an $S$-extended stacky fan does not satisfy 
Assumption~\ref{assump:generation}. 
By taking a bigger $S'\supset S$, we can achieve Assumption 
\ref{assump:generation} and construct a global quantum 
connection on $\tcM_+^{S'}$. 
Then we obtain a global quantum connection on $\tcM_+^S$ 
by restriction. 
In this way, the global quantum connections 
form a projective system over all extension data $S$. 
Assumption \ref{assump:generation} ensures that 
$\bF^+$ is generated by 
a section $\Upsilon_0^+$ and its covariant derivatives. 
For the convenience of discussion, we will sometimes 
use Assumption \ref{assump:generation} in the rest of the paper, 
but this does not affect the final conclusion. 
\end{remark} 

\section{The Crepant Resolution Conjecture}
\label{sec:CRC}

We now come to the main result in this paper.  In Theorem~\ref{thm:global_qconn}, 
we constructed a global quantum connection $(\bF^+,\bnabla^+, \bE^+)$ for $X_+$ 
on $\tcM_+^\circ$, where $\tcM_+^\circ$ is an open subset of 
the universal cover $\tcM_+$ of $(\hU_+\setminus \{\sfy^e = \frc\})/\bmu_B$. 
By applying Theorem~\ref{thm:global_qconn} to $X_-$ rather than $X_+$, 
we obtain a global quantum connection $(\bF^-,\bnabla^-, \bE^-)$ for $X_-$ 
on $\tcM_-^\circ$, where $\tcM_-^\circ$ is an open subset of 
the universal cover $\tcM_-$ of $(\hU_-\setminus \{\sfy^e = \frc\})/\bmu_A$. 
We now show that these two global quantum connections are gauge-equivalent 
on a common covering $\tcM$: the universal cover of 
$\hcM_{\rm reg} \setminus \{\sfy^e = 0,\frc,\infty\}$. 
\[
\xymatrix{
&  \tcM \ar[dl]_{\pi_+} \ar[dr]^{\pi_-} \ar[d]  & \\ 
\tcM_+ \ar[d] 
& \hcM_{\rm reg}\setminus \{\sfy^e = 0,\frc,\infty\} \ar[dl]\ar[dr] 
& \tcM_- \ar[d] \\ 
(\hU_+\setminus \{\sfy^e = \frc\})/
\bmu_B &  
& (\hU_-\setminus \{\sfy^e = \frc\})/\bmu_A }
\]
Moreover, we show that the analytic continuation of flat sections is induced 
by a Fourier--Mukai transformation $\FM \colon K_T^0(X_-) \to K_T^0(X_+)$ 
through the equivariant integral structure in~\S\ref{sec:equiv_int_str}.  We establish the gauge-equivalence of the two global quantum connections in several steps, beginning in~\S\ref{sec:gauge-equivalence} by expressing the gauge transformation involved as a linear sympletomorphism $\UU$ between the Givental spaces for $X_+$ and $X_-$.  In~\S\ref{sec:Mellin-Barnes} we use the Mellin--Barnes method to analytically continue the $I$-function $I_+$, deducing from this a formula for~$\UU$.  In~\S\ref{sec:FM} we construct a Fourier--Mukai transformation $\FM \colon K_T^0(X_-) \to K_T^0(X_+)$  associated to the toric birational transformation $X_+ \dashrightarrow X_-$.  Finally in~\S\ref{sec:FM_match} and~\S\ref{sec:end_of_the_proof} we complete the proof of gauge-equivalence, and of the Crepant Resolution Conjecture in the toric case, by showing that the symplectic transformation~$\UU$ coincides, via the equivariant integral structure, with the Fourier--Mukai transformation~$\FM$.

\subsection{The Global Quantum Connections are Gauge-Equivalent}
\label{sec:gauge-equivalence}
Let $\cU_\pm$ denote the underlying topological space of 
$\tcM_\pm$. The space $\cU_+$ is the universal cover 
of $\{\sfy_r \in \CC : \sfy_r^{\sfp_r^+ \cdot e} \neq \frc\}$ 
and $\cU_-$ is the universal cover of 
$\{\tsfy_r \in \CC: \tsfy_r^{\sfp_r^- \cdot (-e)} \neq \frc^{-1} \}$. 
The underlying topological space of $\tcM$ is the universal cover 
$\cU$ of $\cC_{\rm reg} \setminus \{\sfy^e = 0,\frc,\infty\}$. 
We have natural maps $\pi_\pm \colon \cU \to \cU_\pm$ 
and set 
\begin{align*} 
\cU^\circ & := \pi_+^{-1}(\cU^\circ_+) \cap \pi_-^{-1}(\cU^\circ_-)
\subset \cU \\ 
\tcM^\circ & := \tcM|_{\cU^\circ} 
\end{align*} 
where $\cU_\pm^\circ\subset \cU_\pm$ is the open 
dense subset from Theorem \ref{thm:global_qconn}. 
Note that $\cU \setminus \cU^\circ$ is a discrete set. 
Since we use $P_\pm \in \cC_{\rm reg}$ as base points 
of the universal covers $\cU_\pm$, we need to specify 
a path from $P_+$ to $P_-$ in $\cC_{\rm reg}\setminus\{\sfy^e = \frc\}$ 
in order to identify the maps $\cU \to \cU_\pm$ between universal covers. 
We consider a path in the $\log(\sfy^e)$-plane starting from 
$\log(\sfy^e) = -\infty$ and ending at $\log(\sfy^e) = \infty$ 
such that it avoids $\log(\frc) + 2\pi \tti \ZZ$.  
We use a path $\gamma$ as in Figure \ref{fig:path_ancont} 
passing through the interval 
\[
\bigl ( \log|\frc|  + \pi \tti (w-1), \log |\frc| + \pi\tti(w+1) \bigr ) 
\]
in the $\log(\sfy^e)$-plane, where 
$w := -1 - \sum_{j: D_j\cdot e<0} (D_j \cdot e) 
= -1 + \sum_{j:D_j \cdot e>0} (D_j\cdot e)$. 

\begin{figure}[htbp] 
\centering
\includegraphics[, width=0.6\textwidth,bb=135 573 484 729]{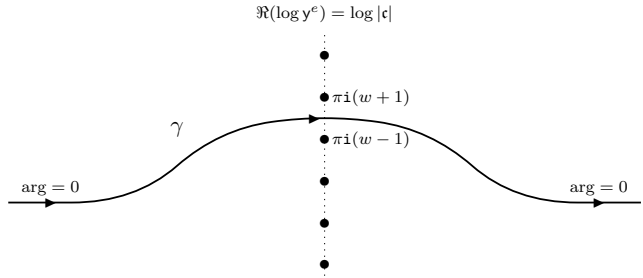} 
\caption{The path $\gamma$ of analytic continuation on the 
$\log(\sfy^e)$-plane}
\label{fig:path_ancont}
\end{figure}

\begin{theorem} 
\label{thm:U} 
Let $\cH(X_\pm) = H^\bullet_{\CR,T}(X_\pm) \otimes_{R_T} 
R_T(\!(z^{-1})\!)$ 
denote Givental's symplectic vector space for $X_\pm$ (see~\S\ref{sec:Givental_cone})
without Novikov variables, i.e.~with $Q$ specialized to~$1$.
There exists a degree-preserving\footnote{We use the 
usual grading on $H_{\CR,T}^\bullet(X_\pm)$, $R_T =H_T^\bullet({\rm pt})$ 
and set  $\deg z=2$.}   
$R_T(\!(z^{-1})\!)$-linear symplectic transformation 
$\UU \colon \cH(X_-) \to \cH(X_+)$ such that:
\begin{enumerate}
\item $I_+(\sfy,z) = \UU I_-(\sfy,z)$ after analytic continuation in $\sfy^e$
along the path $\gamma$ in Figure {\rm\ref{fig:path_ancont}}; 
\item $\UU \circ (g_-^\star v\cup)= (g_+^\star v\cup) \circ \UU$ for all 
$v\in H^2_T(\overline{X}_0)$, where  $g_\pm \colon X_\pm \to 
\overline{X}_0$ is the common blow-down appearing in the 
diagram \eqref{eq:crepant_diagram}; 
\item there exists a Fourier--Mukai transformation 
$\FM \colon K_T^0(X_-) \to K_T^0(X_+)$ such 
that the following diagram commutes: 
\begin{equation} 
  \label{eq:FM_U}
  \begin{aligned}
    \xymatrix{
      K_T^0(X_-) \ar[r]^{\FM} \ar[d]_{\tPsi_-} & K_T^0(X_+) \ar[d]^{\tPsi_+}\\ 
      \tcH(X_-) \ar[r]^{\UU} & \tcH(X_+) 
    }
  \end{aligned}
\end{equation}
where the vertical map $\tPsi_\pm\colon 
K_T^0(X_\pm) \to \tcH(X_\pm)$ is the map  
\[ 
\tPsi_\pm(E) = z^{-\mu^\pm} z^{\rho^\pm} 
\left( \hGamma_{X_\pm} \cup (2\pi\tti)^{\frac{\deg_0}{2}} 
\inv^* \tch(E) \right) 
\]
taking values\footnote{Cf.~Corollary~\ref{cor:homogeneous_flat_sections}.} in the ``multi-valued Givental space'':
\[
\tcH(X_\pm) = H_{\CR,T}^\bullet(X_\pm) 
\otimes_{R_T} R_T[\log z](\!(z^{-1/k})\!)
\]
Here $k\in \N$ is an integer such that all the eigenvalues 
of $k \mu^{+}, k\mu^-$ are integers.   
\end{enumerate} 
\end{theorem} 

Theorem~\ref{thm:U} will be proved in~\S\ref{sec:Mellin-Barnes} and~\S\ref{sec:FM}. 
The Fourier--Mukai kernel will be described in \S\ref{sec:FM}: it is given by 
a toric common blow-up $\tX$ of $X_\pm$. 

\begin{notation} 
\label{not:murho} 
In Theorem \ref{thm:U}, $\rho^\pm = c_1^T(TX_\pm)\in H^2_T(X_\pm)$, 
$\mu^\pm$ is the grading operator \eqref{eq:Euler_mu} 
on $H^\bullet_{\CR,T}(X_\pm)$ and 
$\deg_0\colon H^{\bullet\bullet}_T(IX_\pm) 
\to H^{\bullet\bullet}_T(IX_\pm)$ 
is the degree operator as in \S \ref{sec:equiv_int_str}. 
\end{notation} 

\begin{theorem} 
\label{thm:CTC_qconn}
Let $(\bF^\pm, \bnabla^\pm,\bE^\pm)$ be the global 
quantum connections for $X_\pm$ over $\tcM_\pm^\circ(R_T[z])$ 
from Theorem \ref{thm:global_qconn}. 
We have that $\bE^+ = \bE^-$ on $\tcM(R_T)$. 
There exists a gauge transformation 
\[
\Theta \in \Hom\big(H^\bullet_{\CR,T}(X_-), H^\bullet_{\CR,T}(X_+)\big) 
\otimes_{R_T} (\cO_{\cU^\circ}\otimes R_T)[z][\![\sfy_1,\dots,\sfy_{r-1}]\!] 
\]
over $\tcM^\circ(R_T[z])$ such that:
\begin{itemize} 
\item $\bnabla^-$ and $\bnabla^+$ are gauge-equivalent 
via $\Theta$, i.e.~$
\bnabla^+ \circ \Theta = \Theta \circ \bnabla^-$; 

\item $\Theta$ is homogeneous of degree zero, i.e.~$
\bGr^+ \circ \Theta = \Theta \circ \bGr^-$ with 
$\bGr^\pm := z\parfrac{}{z} + \bE^\pm + \mu^\pm$; 

\item $\Theta$ preserves the orbifold Poincar\'{e} pairing, i.e.~$(\Theta(\sfy,-z) \alpha, \Theta(\sfy,z) \beta) = (\alpha,\beta)$.
\end{itemize} 
Moreover, the analytic continuation of the $K$-theoretic flat sections 
in Definition \ref{def:K_framing} (with Novikov variables $Q$ 
set to be one, see \S \ref{sec:specialization}) is induced 
by the Fourier--Mukai transformation: 
\begin{align*}
  \Theta\Bigl( \frs(E)(\tau_-(\sfy),z) \Bigr) = \frs(\FM(E))(\tau_+(\sfy),z) &&
  \text{for all $E \in K_T^0(X_-)$}
\end{align*}
where $\tau_\pm$ are the mirror maps in Theorem~\ref{thm:global_qconn}.
\end{theorem} 

\begin{remark} 
The symplectic transformation $\UU$ in Theorem~\ref{thm:U} and the gauge transformation 
$\Theta$ in Theorem~\ref{thm:CTC_qconn} are related by 
\begin{equation} 
\label{eq:Theta-U}
L_+(\tau_+(\sfy),z)^{-1} \circ \Theta = 
\UU \circ L_-(\tau_-(\sfy),z)^{-1} 
\end{equation} 
where $L_\pm$ is the fundamental solution for the quantum 
connection of $X_\pm$ in Proposition \ref{prop:fundsol}.
The gauge transformation $\Theta$ sends 
the section $\Upsilon_0^-\in \bF^-$ to the section 
$\Upsilon_0^+\in \bF^+$, 
where $\Upsilon_0^\pm$ are as 
in Theorem \ref{thm:global_qconn}. 
\end{remark} 

\begin{remark} 
  Theorems~\ref{thm:U} and~\ref{thm:CTC_qconn} 
can be interpreted as the statement that the symplectic 
transformation $\UU$ matches up the Givental
cones $\cL_\pm$ associated to $X_\pm$ after
analytic continuation of $\cL_\pm$:  
\begin{equation} 
\label{eq:U_cones}
\UU(-z) \cL_- = \cL_+.  
\end{equation} 
In fact, Remark \ref{rem:fundamentalsol_cone} 
suggests that we may analytically continue the Lagrangian 
cones by the formula: 
\[
\cL_\pm \, \text{``$=$''} \, \bigcup_{\sfy\in \tcM^\circ} 
z L_\pm(\tau_\pm(\sfy), -z)^{-1} \cH_+(X_\pm) 
\]
and then equation \eqref{eq:Theta-U} would imply \eqref{eq:U_cones}. 
As discussed in the Introduction, to avoid subtleties in defining the analytic continuation of Givental cones in the equivariant setting, in this paper we state our results in terms of analytic continuation of the $I$-function (Theorem~\ref{thm:U}) or in terms of the equivariant quantum connection and gauge transformations (Theorem~\ref{thm:CTC_qconn}).
\end{remark} 

\begin{remark} 
 Theorem~\ref{thm:CTC_qconn} implies that the global quantum connections 
of $X_+$ and $X_-$ can be glued together to give a flat connection over 
$\tcM^\circ$. 
This flat connection descends to the formal neighbourhood 
$\hcM$ of $\cC$ in the secondary toric variety $\cM$ 
via Galois symmetry as in Remark \ref{rem:overlattice}. 
This global connection, or $D$-module, on 
$\hcM$ can be described by explicit GKZ-type differential 
equations: it is a completed version of Borisov--Horja's 
\emph{better-behaved GKZ system}\footnote
{The better-behaved GKZ system is in general generated by several elements. 
In our case, by adding enough extended vectors that 
Assumption \ref{assump:generation} is satisfied, 
we can make it generated by a single standard generator 
$1$, and in this case the better-behaved GKZ system is the same as the 
original GKZ system~\cite{GKZ:diffeq}.}
\cite{Borisov--Horja:GKZ}. 
In the papers \cite{Iritani, Reichelt--Sevenheck}, 
the toric quantum connection is described in terms of 
GKZ-type differential equations through mirror symmetry. 
The $I$-functions $I_\pm(q,z)$ are local solutions to 
these differential equations around the large radius limit points. 
\end{remark}

\begin{proof}[Proof that Theorem~\ref{thm:U} implies Theorem \ref{thm:CTC_qconn}]
One can easily check that the change of variables 
\eqref{eq:change_of_variables} preserves degree, and that 
$\bE^+ = \bE^-$. 
By Theorem \ref{thm:U}, we have 
\begin{equation} 
\label{eq:U_Ipm} 
I_+(\sfy,z) = \UU I_-(\sfy,z)
\end{equation} 
under analytic continuation along the path $\gamma$.  
The discussion in~\S\ref{sec:global_qconn} 
(see Lemma~\ref{lem:I_derivatives}, equation~\ref{eq:I_derivative_matrix}, and equation~\ref{eq:Birkhoff_I}) yields:
\begin{equation} 
\label{eq:ILUpsilon_+}
\begin{bmatrix} 
 & \vert &  \\ 
\cdots & z^{-1} P_{f,i}^+ I_+ & \cdots \\ 
 & \vert &   
\end{bmatrix} 
= e^{\sigma_+/z}\bL_+(\sfy,z)^{-1} \Upsilon^+(\sfy,z). 
\end{equation} 
Similarly, the discussion in~\S\ref{sec:global_qconn} 
applied to $X_-$ yields a global section $\Upsilon^-_0$ of $\bF^-$ 
and a (global) fundamental solution $\bL_-(\sfy,z)e^{-\sigma_-/z}$ 
for $\bnabla^-= d + z^{-1} \bA^-(\sfy)$ such that:
\begin{equation*}
z^{-1} I_-(\sfy,z) = e^{\sigma_-/z} \bL_-(\sfy,z)^{-1} 
\Upsilon^-_0(\sfy,z) 
\end{equation*} 
Applying the differential operators $P_{f,i}^+$ to $z^{-1} I_-(\sfy,z)$, 
we obtain 
\begin{equation} 
\label{eq:ILUpsilon_-} 
\begin{bmatrix} 
 & \vert &  \\ 
\cdots & z^{-1} P_{f,i}^+ I_- & \cdots \\ 
 & \vert &   
\end{bmatrix} 
= e^{\sigma_-/z} \bL_-(\sfy,z)^{-1} 
\begin{bmatrix} 
& \vert &  \\ 
\cdots & P_{f,i}^+(z\bnabla^-) \Upsilon^-_0 & \cdots \\ 
& \vert & 
\end{bmatrix} 
\end{equation} 
where $P_{f,i}^+(z\bnabla^-)$ is obtained from $P_{f,i}$ 
by replacing $z\partial_v$ with $z\bnabla^-_v$ for 
vector fields $v$. 
Let $\tUpsilon^-$ denote the matrix with column vectors 
$P_{f,i}^+(z\bnabla^-) \Upsilon^-_0$. 
Comparing \eqref{eq:ILUpsilon_+} with 
\eqref{eq:ILUpsilon_-} and using \eqref{eq:U_Ipm}, 
we obtain 
\[
e^{\sigma_+/z} \bL_+(\sfy,z)^{-1} \Upsilon^+ 
=\UU e^{\sigma_-/z} \bL_-(\sfy,z)^{-1}  \tUpsilon^- 
\]
since $\UU$ is independent of the base variables $\sfy$. 
In particular, it follows that $\tUpsilon^-$ is invertible. 
Setting $\Theta = \Upsilon^+ (\tUpsilon^-)^{-1}$, we obtain:
\begin{equation} 
\label{eq:Theta-U_relation} 
\left( e^{\sigma_+/z} \bL_+(\sfy,z)^{-1}\right)  \Theta(\sfy,z) 
=\UU \left( e^{\sigma_-/z} \bL_-(\sfy,z)^{-1}\right). 
\end{equation} 
Since $e^{\sigma_\pm/z} \bL_\pm^{-1}$ are fundamental 
solutions for $\bnabla^\pm$, 
$\Theta$ gives a gauge transformation 
between $\bnabla^-$ and $\bnabla^+$, i.e.~
$\Theta \circ \bnabla^- = \bnabla^+ \circ\Theta$. 
One may assume that the first columns of $\Upsilon^+$ 
and $\tUpsilon^-$ are given respectively by $\Upsilon_0^+$ 
and $\Upsilon^-_0$, and therefore $\Theta(\Upsilon_0^-) 
= \Upsilon_0^+$. 

Next we see that $\Theta$ preserves the grading and the pairing. 
Part (2) in Theorem \ref{thm:U} implies that 
$\UU \circ \theta_-(\sfp^-_i) = \theta_+(\sfp^+_i)\circ \UU$ 
for $i=1,\dots,r-1$, since $\sfp^+_i = \sfp^-_i$ lies on the 
wall $W$ for $1\le i\le r-1$. 
Therefore 
\begin{align*} 
e^{-\sigma_+/z} \circ 
\UU \circ e^{\sigma_-/z} 
& = e^{-\theta_+(\sfp^+_r) \log \sfy_r/z} 
\circ \UU \circ e^{\theta_-\left( \sum_{i=1}^r \sfp^-_i \log \tsfy_i 
- \sum_{i=1}^{r-1} \sfp^+_i \log \sfy_i
\right)/z}  \\ 
& =  e^{-\theta_+(\sfp^+_r) \log \sfy_r/z} 
\circ \UU \circ e^{\theta_-( \sfp_r^+) \log \sfy_r /z} 
\end{align*} 
where we used $\sum_{i=1}^r \sfp^+_i \log \sfy_i = 
\sum_{i=1}^r \sfp^-_i \log \tsfy_i$. 
This together with \eqref{eq:Theta-U_relation} implies that: 
\[
 \bL_+(\sfy,z)^{-1} 
\Theta(\sfy,z) 
= \left( 
e^{-\theta_+(\sfp_r^+) \log \sfy_r/z} \circ \UU \circ 
e^{\theta_-(\sfp_r^+) \log \sfy_r/z}\right)  
\bL_-(\sfy,z)^{-1}
\]
Since $\deg \sfy_r =0$, we know that all of the factors 
in this equation except for $\Theta$ are homogeneous 
of degree zero; thus $\Theta$ is also homogeneous of degree zero.  
The fundamental solutions $e^{\sigma_\pm/z} \bL_\pm^{-1}$ 
preserve the pairing by Proposition \ref{prop:fundsol} 
(we saw in \S \ref{sec:global_qconn} that 
they coincide with the fundamental solutions from 
Proposition \ref{prop:fundsol} via the mirror maps $\tau_\pm$) 
and $\UU$ also preserves the pairing. Thus $\Theta$ preserves the pairing. 

Finally we consider the analytic continuation of $K$-theoretic 
flat sections. Note that the flat section $\frs(E)(\tau_-(\sfy),z)$ 
is analytically continued along $\tcM^\circ$ by the right-hand 
side of the formula
\[
\frs(E)(\tau_-(\sfy),z) = \frac{1}{(2\pi)^{\dim X_-/2}}
\bL_-(\sfy,z) e^{- \sigma_-/z} \tPsi_-(E)
\]
where $\tPsi_-$ is the map in Theorem \ref{thm:U}. 
Using \eqref{eq:Theta-U_relation}, we obtain:
\[
\Theta(\frs(E)(\tau_-(\sfy,z))) = \frac{1}{(2\pi)^{\dim X_-/2}}
\bL_+(\sfy,z) e^{-\sigma_+/z} \UU \left( 
\tPsi_-(E) \right)
\]
Part (3) of Theorem \ref{thm:U} shows 
that this is equal to $\frs(\FM(E))(\tau_+(\sfy),z)$.  
\end{proof} 

\subsection{Mellin--Barnes Analytic Continuation}
\label{sec:Mellin-Barnes}

In this section, we compute the analytic continuation of 
the $I$-function and determine the linear transformation 
$\UU$ in Theorem \ref{thm:U}.

\subsubsection{The $H$-Function}

It will be convenient to introduce another cohomology-valued 
hypergeometric function called the $H$-function.  
Noting that the $I$-function can be written 
in terms of ratios of $\Gamma$-functions:
\[
I_+(\sfy,z):=
z e^{\sigma_+/z}
\sum_{d \in \KK_+} 
\frac{\sfy^d}{z^{(D_1+\cdots + D_m) \cdot d}}   
\left(
  \prod_{j=1}^{m}
  \frac{\Gamma\big(1 + \frac{u_j}{z} - \<{-D_j} \cdot d\>\big)}
  {\Gamma\big(1 + \frac{u_j}{z} + D_j \cdot d\big)}
\right) 
\frac{\fun_{[{-d}]}}{z^{\iota_{[{-d}]}}}
\]
we set:
\[
H_+(\sfy):= 
e^{\frac{\sigma_+}{2 \pi \tti} }
\sum_{d \in \KK_+} 
\sfy^d  
\left( 
\prod_{j=1}^{m}
\frac{1}
{\Gamma\big(1 + \frac{u_j}{2 \pi \tti} + D_j \cdot d\big)}
\right) 
\fun_{[d]}\]
and similarly for $H_-$.  
Formally speaking, $H_+$ belongs to the space:
\[
\prod_{p} \left( 
H^p_{T}(IX_+)[\log \sfy_1,\dots,\log\sfy_r] [\![\sfy_1,\dots,\sfy_r]\!] 
\right)
\]
Noting that the $T$-equivariant Gamma class of $X_+$ is given by 
\[
\hGamma_{X_+} = \bigoplus_{f\in \KK_+/\LL} 
\left( 
\prod_{j=1}^m \Gamma(1 + u_j - \<D_j \cdot f\>) \right) 
\fun_f
\]
we obtain the relationship between the $H$-function and the $I$-function: 
\begin{equation} 
\label{eq:IH_relation}
z^{-1} I_+(\sfy,z) = z^{-\frac{c_0(\lambda)}{2\pi\tti} - \frac{\dim X_+}{2}}
z^{-\mu^+} z^{\rho^+} 
\left(\hGamma_{X_+} \cup (2\pi\tti)^{\frac{\deg_0}{2}} 
\inv^* H\left(z^{-\frac{\deg \sfy}{2}} \sfy \right) \right)
\end{equation} 
where $\rho^+$, $\mu^+$, $\deg_0$ are as in Notation~\ref{not:murho} and 
\[
z^{-\frac{\deg \sfy}{2}}\sfy = (z^{-\frac{\deg \sfy_1}{2}}
\sfy_1 , \dots, z^{-\frac{\deg \sfy_r}{2}}\sfy_r ). 
\]
The relationship between $H_-$ and $I_-$ is similar. 
\begin{remark} 
The $H$-function $H_+$ has an analytic properties analogous 
to those of the $I$-function stated in Lemma~\ref{lem:I_analyticity} 
and Remark~\ref{rem:I_analyticity}. 
Namely $e^{-\sigma_+/(2\pi\tti)} H_+(\sfy)$ is a formal power 
series in $\sfy_1,\dots,\sfy_{r-1}$ with coefficients of the form $\sum_{i=0}^N f_i(\lambda,\sfy_r) \phi_i$ 
where $\{\phi_i\}$ is an $R_T$-basis of $H_{T}^\bullet(IX_+)$ 
and $f_i(\lambda,\sfy_r)$ is analytic in 
$(\lambda_1,\dots,\lambda_r, \sfy_r) \in \CC^m \times \cU_+$. 
Notice that the $H$-function has better analytic behaviour with respect 
to $\lambda$ since $\frac{1}{\Gamma(x)}$ is an entire function. 
The analytic continuation of $H$-functions performed below 
should be understood as analytic continuation of the coefficient functions $f_i(\lambda,\sfy_r)$.
\end{remark} 

\subsubsection{Restriction of the $H$-Function to $T$-Fixed Points}

Recall that the $T$-fixed points on $X_+$ are indexed by minimal 
anticones $\delta\in \cA_+$, and that the $T$-fixed points 
on the inertia stack $IX_+$ are 
indexed by pairs $(\delta,f)$ with $\delta \in \cA_+$ 
a minimal anticone and $f \in \KK_+/\LL$ 
satisfying $D_i \cdot f \in \ZZ$ for all $i \in \delta$.  
The minimal anticone $\delta$ determines a $T$-fixed point $x_\delta$ 
on $X_+$ and 
the pair $(\delta,f)$ determines a $T$-fixed point $x_{(\delta,f)}$ 
on the component $X_+^f$ of the inertia stack $IX_+$.  
Let $i_\delta$ and $i_{(\delta,f)}$ denote 
the inclusion maps $x_\delta \to X_+$ and $x_{(\delta,f)} \to IX_+$ 
respectively. 
Set $u_j(\delta) = i_{\delta}^\star u_j\in H^2_T({\rm pt})$, 
noting that $u_j(\delta) = 0$ if and only if $j \in \delta$.  
We have that:
\begin{equation}
  \label{eq:restriction_of_H_+}
 i_{(\delta,f)}^\star H_+ = 
  \sum_{d \in \KK_+:  [d]=f}
  \frac{\sfy^d}{\prod_{j \in \delta} \Gamma\big(1 + D_j \cdot d \big)}
  \frac{e^{\frac{1}{2\pi\tti} \sigma_+(\delta)}}
  {\prod_{j \not \in \delta}\Gamma\big(1 + \frac{u_j(\delta)}{2 \pi \tti} 
+ D_j \cdot d\big)}
\end{equation}
where $\sigma_+(\delta):= i_{\delta}^\star \sigma_+$. 
Consider the factor $\prod_{j \in \delta} \Gamma\big(1 + D_j \cdot d \big)^{-1}$ 
in the summand: since $d \equiv f \mod \LL$ and since $D_j \cdot f \in \ZZ$ 
for all $j \in \delta$, the term $D_j \cdot d$ here is an integer.  
Thus the factor $\prod_{j \in \delta} \Gamma\big(1 + D_j \cdot d \big)^{-1}$ 
vanishes unless $d \in \delta^\vee$, where 
\[
\delta^\vee := \{ d \in \LL\otimes \QQ : D_j \cdot d \in \ZZ_{\ge 0} 
\text{ for all } j\in \delta\}. 
\]
The $H$-function is a sum over the subset 
$\KKeff_+$ of $\KK_+$,
\[
\KKeff_+ = \Big\{ f \in \LL \otimes \QQ : 
\big\{ i \in \{1,2,\ldots,m\} : D_i \cdot f \in \ZZ_{\geq 0} \big\}  \in \cA_+ \Big\} 
\]
which is in general quite complicated, but the restriction 
$i_{(\delta,f)}^\star H_+$ of $H_+$ to a $T$-fixed point 
in $IX_+$ is a sum over the much simpler set $\delta^\vee$. 

\subsubsection{Analytic Continuation of the $H$-Function}  
The Localization Theorem in $T$-equivariant cohomology 
\cite{Berline--Vergne, Atiyah--Bott, GKM} implies that 
one can compute the analytic continuation of $H_+$ 
by computing the analytic continuation of the restriction 
$i_{(\delta,f)}^\star H_+$ to each $T$-fixed point 
$x_{(\delta,f)} \in I X_+$.  
The restriction $i_{(\delta,f)}^\star H_+$ is a 
$H_T^{\bullet\bullet}({\rm pt})$-valued function. During 
the course of analytic continuation, we regard the equivariant 
parameters $\lambda_1,\dots,\lambda_m$ as generic complex 
numbers.  
There are two cases: 
\begin{itemize}
\item $\delta \in \cA_+ \cap \cA_-$;
\item $\delta \in \cA_+$ but $\delta \not \in \cA_-$. 
\end{itemize}
The anticone $\delta$ determines a $T$-fixed point $x_\delta$ in $X_+$, 
and in the first case it also determines a fixed point in $X_-$.  
In the first case the birational transformation 
$\varphi \colon X_+ \dashrightarrow X_-$ 
is an isomorphism in a neighbourhood of $x_\delta$, 
and it is clear from \eqref{eq:restriction_of_H_+} 
that $i_{(\delta,f)}^\star H_+ = i_{(\delta,f)}^\star H_-$, 
noting that $u_j(\delta)$ is the same for $X_+$ and $X_-$. 
In the second case $x_\delta$ lies in the flopping locus of $\varphi$, 
and we will see that the analytic continuation of $i_{(\delta,f)}^\star H_+$ 
is a linear combination of restrictions $i_{(\delta_-,f_-)}^\star H_-$ 
for appropriate $\delta_- \in\cA_-$ and $f_- \in \KK_-$.  
Note that in the second case, $\delta$ has the form $\{j_1,\ldots,j_{r-1}, j_+\}$ 
with $D_{j_1},\ldots,D_{j_{r-1}} \in W$ 
and\footnote{Recall that $e \in \LL$ is the primitive lattice vector 
in $W^\perp$ such that $e>0$ on $C_+$ and $e<0$ on $C_-$.} 
$D_{j+} \cdot e>0$ (see Lemma \ref{lem:Apm}). 

\begin{definition} \label{def:next_to_X}
Let $\delta_+ \in \cA_+$ and $\delta_- \in \cA_-$ be minimal anticones.  
We say that $\delta_+$ is \emph{next to} $\delta_-$, 
written $\delta_+ | \delta_-$, if $\delta_+ = \{j_1,\ldots,j_{r-1},j_+\}$ 
and $\delta_- = \{j_1,\ldots,j_{r-1},j_-\}$ with 
$D_{j_1},\ldots,D_{j_{r-1}} \in W$, $D_{j_+} \cdot e > 0$, 
and $D_{j-} \cdot e < 0$. In this case $\delta_+\notin \cA_-$ 
and $\delta_- \notin \cA_+$. 
\end{definition}

\begin{definition} \label{def:next_to_IX}
Let $(\delta_+,f_+)$ index a $T$-fixed point on $IX_+$ and 
$(\delta_-, f_-)$ index a $T$-fixed point on $I X_-$.  
We say that $(\delta_+,f_+)$ is \emph{next to} $(\delta_-,f_-)$, 
written $(\delta_+,f_+) | (\delta_-,f_-)$, if $\delta_+ | \delta_-$ 
and there exists $\alpha \in \QQ$ 
such that $f_- = f_+ + \alpha e$ in $\LL\otimes \QQ/\LL$. 
\end{definition}

The analytic continuation of $i_{(\delta,f)}^\star H_+$ 
is a linear combination of  $i_{(\delta_-,f_-)}^\star H_-$ 
such that $(\delta,f)$ is next to $(\delta_-,f_-)$.

\begin{notation} 
\label{not:lift}
Fix lifts $\KK_+/\LL \to \KK_+$ and $\KK_-/\LL \to \KK_-$ 
such that, for any pairs $(d_+, d_-)\in \KK_+\times \KK_-$  
with $d_+ - d_- \in \QQ e$, the lifts of $[d_+]$ and $[d_-]$ 
differ by a rational multiple of $e$. 
\end{notation}

\begin{lemma} \label{lem:weights} 
Let $\delta_+ \in \cA_+$ and $\delta_- \in \cA_-$ be 
minimal anticones such that $\delta_+ | \delta_-$, and 
let $j_-$ be the element of $\delta_-$ such that $j_- \not \in \delta_+$.  
Then for any $j$, one has: 
  \[
u_j(\delta_+) = 
u_j(\delta_-) + \frac{D_j \cdot e}{D_{j_-} \cdot e} u_{j_-}(\delta_+). 
  \]
\end{lemma}

\begin{proof} 
Write $\delta_- =\{j_1,\dots,j_{r-1},j_-\}$. 
Since $D_{j_1},\dots,D_{j_{r-1}},D_{j_-}$ form a basis of 
$\LL^\vee \otimes \QQ$, we can write:
\[
D_j = c_1 D_{j_1} + \cdots + c_{r-1} D_{j_{r-1}} + c_- D_{j_-}
\]
Since $D_{j_1}, \dots, D_{j_{r-1}}$ are on the wall, pairing 
with $e$ yields:
\begin{equation}
\label{eq:dot_e}
D_j \cdot e= c_-  (D_{j_-} \cdot e). 
\end{equation} 
Applying the homomorphism $\theta_\pm$ from \eqref{eq:theta}, 
we obtain 
\[
u_j-\lambda_j = 
c_1 (u_{j_1} - \lambda_{j_1}) + \cdots + c_{r-1} (u_{j_{r-1}} - 
\lambda_{j_{r-1}}) + 
 c_- (u_{j_-} - \lambda_{j_-}) 
\]
on both $X_+$ and $X_-$. 
Restricting to $x_{\delta_+}\in X_+$ and $x_{\delta_-}\in X_-$, 
we get the two relations: 
\begin{align*} 
u_j(\delta_+) -\lambda_j & = - c_1 \lambda_{j_1} 
-\cdots - c_{r-1} \lambda_{j_{r-1}} + c_- (u_j(\delta_+) - \lambda_{j_-}), \\ 
u_j(\delta_-) -\lambda_j & = - c_1 \lambda_{j_1} 
-\cdots - c_{r-1} \lambda_{j_{r-1}} - c_- \lambda_{j_-}. 
\end{align*} 
Comparing the two equations, we get 
$u_j(\delta_+) = u_j(\delta_-) +c_- u_{j_-}(\delta_+)$. 
The conclusion now follows from equation~\ref{eq:dot_e}. 
\end{proof}

\begin{corollary} \ 
\label{cor:weights}  
\begin{enumerate}
\item Let $\delta$ be a minimal anticone such that $\delta\in 
\cA_+ \cap \cA_-$. Then 
$\sigma_+(\delta) = \sigma_-(\delta)$. 
\item Let $\delta_+ \in \cA_+$, $\delta_-\in \cA_-$ 
be minimal anticones such that $\delta_+|\delta_-$ 
and let $j_-\in \delta_-$ be an element such that $j_- \notin \delta_+$. 
Then: 
\[
\sigma_+(\delta_+) = \sigma_-(\delta_-) + 
\frac{\log \sfy^e }{D_{j_-}\cdot e} 
 u_{j_-}(\delta_+)
\]
\end{enumerate}
\end{corollary} 
\begin{proof} \ 
  \begin{enumerate}
  \item As we discussed, $u_j(\delta)$ is the same for $X_+$ and $X_-$ whenever $\delta \in \cA_+\cap\cA_-$. Therefore $i_\delta^\star \theta_+(D_j) = i_\delta^\star \theta_-(D_j)$ for all $j$. In particular $i_\delta^\star \theta_+(p) = i_\delta^\star \theta_-(p)$ for every $p\in \LL^\vee \otimes \CC$.  Setting $p= \sum_{i=1}^r \sfp_i^+ \log \sfy_i = \sum_{i=1}^r \sfp_i^- \log \tsfy_i$, we obtain (1).
  \item Lemma \ref{lem:weights} shows that
    \begin{equation} 
      \label{eq:weights_theta} 
      i_{\delta_+}^\star\theta_+(p)  = i_{\delta_-}^\star \theta_-(p) 
      + \frac{p\cdot e}{D_{j_-} \cdot e} u_{j_-}(\delta_+) 
    \end{equation} 
    for all $p\in \LL^\vee \otimes \CC$. 
    Setting again $p= \sum_{i=1}^r \sfp_i^+ \log \sfy_i = 
    \sum_{i=1}^r \sfp_i^- \log \tsfy_i$, we obtain (2). 
  \end{enumerate}
\end{proof} 

\begin{theorem}
  \label{thm:analytic_continuation}
Let $(\delta_+,f_+)$ index a $T$-fixed point on $I X_+$.  
If $\delta_+ \in \cA_+ \cap \cA_-$ then: 
  \[
  i_{(\delta_+,f_+)}^\star H_+ = i_{(\delta_+,f_+)}^\star H_-. 
  \]
Otherwise, after analytic continuation along the path 
$\gamma$ in Figure \ref{fig:path_ancont}, we have:
  \[
  i_{(\delta_+,f_+)}^\star H_+ = 
  \sum_{\substack{(\delta_-,f_-) : \\ (\delta_+,f_+)|(\delta_-,f_-)}}
  C_{\delta_+,f_+}^{\delta_-,f_-} \, 
  i_{(\delta_-,f_-)}^\star H_- 
  \]
  where:
  \begin{multline*}
    C_{\delta_+,f_+}^{\delta_-,f_-} = 
    e^{\frac{\pi\tti w}{D_{j_-}\cdot e} 
\big(\frac{u_{j_-}(\delta_+)}{2 \pi \tti} 
+ D_{j_-} \cdot (f_+-f_-)\big)}
    \\
\times    \frac{\sin  \pi\Big(\frac{u_{j_-}(\delta_+)}{2 \pi \tti} 
+ D_{j_-} \cdot (f_+-f_-)\Big)}
    {({-D_{j_-}} \cdot e) \sin \frac{\pi}{{-D_{j_-}}\cdot e} 
\Big(\frac{u_{j_-}(\delta_+)}{2 \pi \tti} + D_{j_-} \cdot (f_+-f_-)\Big)}
    \prod_{\substack{j : D_j \cdot e < 0 \\ j \ne j_-}}
    \frac{\sin  \pi\Big(\frac{u_j(\delta_+)}{2 \pi \tti} + D_j \cdot f_+\Big)}
    {\sin  \pi\Big(\frac{u_j(\delta_-)}{2 \pi \tti} + D_j \cdot f_-\Big)}
  \end{multline*}
with $w := -1 - \sum_{j:D_j \cdot e<0} D_j \cdot e 
= -1 + \sum_{j:D_j\cdot e>0} D_j \cdot e$ 
and $j_-\in \delta_-$ given by the unique element such that 
$D_{j_-} \cdot e <0$. 
\end{theorem}

\begin{remark} 
The coefficient $C_{\delta_+,f_+}^{\delta_-,f_-}$ 
does not depend on the choice of lifts $f_+ \in \KK_+$ and $f_-\in \KK_-$ 
such that $f_+ -  f_- \in \QQ e$ (see Notation \ref{not:lift}). 
\end{remark} 

\begin{proof}[Proof of Theorem \ref{thm:analytic_continuation}] 
The first statement follows immediately from \eqref{eq:restriction_of_H_+} 
and Corollary~\ref{cor:weights}. 
In this case, $i_{(\delta_+,f_+)}^\star H_+$ 
(respectively~$i_{(\delta_+,f_+)}^\star H_-$)  
is a formal power series in 
$\sfy_1,\dots,\sfy_{r-1}$ (respectively in~$\tsfy_1,\dots,\tsfy_{r-1}$) 
with coefficients that are \emph{polynomials} in $\sfy_r$ (respectively in~$\tsfy_r$),  
and the series $i_{(\delta_+,f_+)}^\star H_+$,  
$i_{(\delta_+,f_+)}^\star H_-$ match under the change 
\eqref{eq:change_of_variables} of co-ordinates. 
Consider now
  \[
  i_{(\delta_+,f_+)}^\star H_+ = e^{\frac{\sigma_+(\delta_+)}{2\pi\tti}}
  \sum_{\substack{d \in \delta_+^\vee: \\ [d]=f_+}}
  \sfy^d 
  \frac{1}
  {\prod_{j =1}^m\Gamma\big(1 + \frac{u_j(\delta_+)}{2 \pi \tti} 
    + D_j \cdot d\big)}
  \]
  where $\delta_+ \in \cA_+$ but $\delta_+ \not \in \cA_-$.  
We can write $d \in \delta_+^\vee$ uniquely as $d = d_+ + k e$ 
with $k$ a non-negative integer, $d_+ \in \delta_+^\vee$, 
and $d_+ - e \not \in \delta_+^\vee$.  Then: 
  \begin{equation}
    \label{eq:sum_over_d_+}
    i_{(\delta_+,f_+)}^\star H_+ = 
    \sum_{\substack{d_+ \in \delta_+^\vee : 
\\ d_+ - e \not \in \delta_+^\vee \\ [d_+]=f_+}} 
    \sfy^{d_+} 
    \sum_{k=0}^\infty 
    \frac{e^{\frac{\sigma_+(\delta_+)}{2 \pi \tti}} 
(\sfy^{e})^k } 
    {\prod_{j =1}^m 
\Gamma\big(1 + \frac{u_j(\delta_+)}{2 \pi \tti} + D_j \cdot d_+ 
+ k D_j \cdot e\big)} 
  \end{equation}
  Consider the second sum here.  This is: 
  \begin{multline}
    \label{eq:inner_sum}
    \sum_{k=0}^\infty
    {e^{\frac{\sigma_+(\delta_+)}{2 \pi \tti} } (\sfy^e)^k  }
    \prod_{j : D_j \cdot e < 0} 
\frac{(-1)^{k D_j \cdot e}  
\sin  \pi\big(-\frac{u_j(\delta_+)}{2 \pi \tti} - D_j \cdot d_+\big)}{\pi} 
\\ 
\times
    \frac{\prod_{j : D_j \cdot e < 0} 
\Gamma\big({-\frac{u_j(\delta_+)}{2 \pi \tti}} - D_j \cdot d_+ - k D_j \cdot e\big)}
    {\prod_{j: D_j \cdot e \geq 0} 
\Gamma\big(1 + \frac{u_j(\delta_+)}{2 \pi \tti} 
+ D_j \cdot d_+ + k D_j \cdot e\big)}
  \end{multline}
  where we used $\Gamma(y) \Gamma(1-y) = \pi/(\sin \pi y)$.  
Thus \eqref{eq:inner_sum} is:
  \begin{multline}
    \label{eq:inner_sum_as_residue_sum}
    \sum_{k=0}^\infty
    e^{\frac{\sigma_+(\delta_+)}{2 \pi \tti} }
    \Res_{s=k}  \Gamma(s) \Gamma(1-s) e^{\pi\tti s} (\sfy^e)^s
    \prod_{j : D_j \cdot e < 0} 
\frac{e^{\pi\tti s (D_j \cdot e)} 
        \sin  \pi\big(-\frac{u_j(\delta_+)}{2 \pi \tti} - D_j \cdot d_+\big)}{\pi} \\ 
\times
    \frac{\prod_{j : D_j \cdot e < 0} 
\Gamma\big({-\frac{u_j(\delta_+)}{2 \pi \tti}} - D_j \cdot d_+ - s D_j \cdot e\big)}
    {\prod_{j: D_j \cdot e \geq 0} 
\Gamma\big(1 + \frac{u_j(\delta_+)}{2 \pi \tti} 
+ D_j \cdot d_+ + s D_j \cdot e\big)} \, ds. 
\end{multline}

Consider now the contour integral
\begin{equation}
    \label{eq:contour_integral}
e^{\frac{\sigma_+(\delta_+)}{2\pi\tti}}    \int_C
    \Gamma(s) \Gamma(1-s) 
    \frac{\prod_{j : D_j \cdot e < 0} 
\Gamma\big({-\frac{u_j(\delta_+)}{2 \pi \tti}} 
- D_j \cdot d_+ - s D_j \cdot e\big) }
    {\prod_{j: D_j \cdot e \geq 0} 
\Gamma\big(1 + \frac{u_j(\delta_+)}{2 \pi \tti} 
+ D_j \cdot d_+ + s D_j \cdot e\big)} 
   \left(e^{-\pi\tti w}\sfy^e\right)^s \, ds
  \end{equation}
  where the contour $C$, shown in Figure~\ref{fig:contour}, 
is chosen such that the poles at $s=n$ are on the right of $C$ 
and the poles at $s = {-1}-n$ and at
  \begin{align}
    \label{eq:interesting_poles}
    s = \textstyle \frac{1}{{-D_{j_-}} \cdot e} 
\Big(  \frac{u_{j_-}(\delta_+)}{2 \pi \tti} + D_{j_-} \cdot d_+  - n \Big)
    && \text{for $j_-$ such that $D_{j_-} \cdot e < 0$}
  \end{align}
  are on the left of $C$; here $n$ is a non-negative integer.
  \begin{figure}
    \centering
    \includegraphics[width=0.5\textwidth]{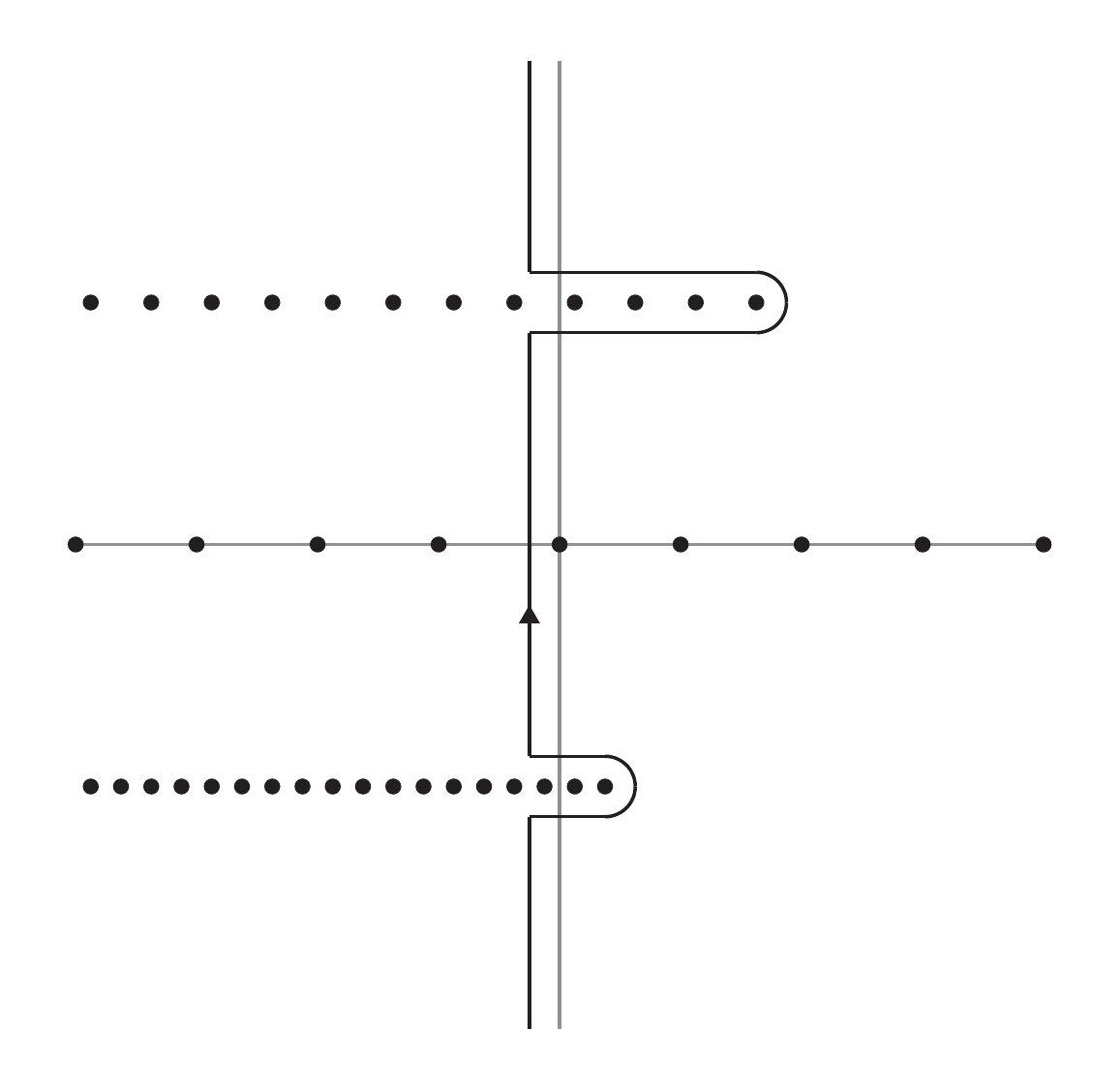}
    \caption{The contour $C$}
    \label{fig:contour}
  \end{figure}
Note that all poles of the integrand are simple.  
By assumption we have that $\sum_{j=1}^m D_j \in W$, 
and hence that $\sum_{j=1}^m D_j \cdot e = 0$.  
Let $\frc\in \CC$ be the conifold point \eqref{eq:conifold_point}. 
Lemma~A.6 in~\cite{Borisov--Horja:FM} implies that:
  \begin{itemize}
  \item the contour integral \eqref{eq:contour_integral} 
is convergent and analytic as a function of $\sfy^e$ in the domain 
$\{\sfy^e : |\arg(\sfy^e) - w \pi|< \pi \}$;
  \item for $|\sfy^e|<|\frc|$, the integral is equal to 
the sum of residues on the right of $C$; and
  \item for $|\sfy^e|>|\frc|$, the integral is equal to 
minus the sum of residues on the left of $C$.
  \end{itemize}
  The residues at $s = {-1}-n$ vanish, 
where $n$ is a non-negative integer: each such residue contains a factor
  \[
  \prod_{j \in \delta_+} \Gamma\big(1 + D_j \cdot\big(d_+ - (n+1) e\big) \big)^{-1}
  \]
  and $d_+ - (n+1) e \not \in \delta_+^\vee$, so at least 
one of the $\Gamma$-functions is evaluated at a negative integer.  
After analytic continuation in $x$, therefore, 
\eqref{eq:inner_sum_as_residue_sum} becomes 
minus the sum of residues at the poles \eqref{eq:interesting_poles}.  
The residue at the pole
  \[
  p = \textstyle \frac{1}{{-D_{j_-}} \cdot e} 
\Big(  \frac{u_{j_-}(\delta_+)}{2 \pi \tti} + D_{j_-} \cdot d_+  - n \Big)
  \]
is: 
\begin{multline}
    \label{eq:simplify_this_residue}
-e^{ \frac{\sigma_+(\delta_+)}{2\pi\tti} }
    (\sfy^e)^p
    e^{\pi\tti p(1 + D_{j_-} \cdot e)} 
    \frac{\sin  \pi\Big(\frac{u_{j_-}(\delta_+)}{2 \pi \tti} 
+ D_{j_-} \cdot d_+\Big)}{\sin \pi p}
    \frac{1}{({-D_{j_-}} \cdot e)} \frac{(-1)^n}{n!} \\
\prod_{\substack{j : D_j \cdot e < 0 \\ j \ne j_-}}
\frac{e^{\pi\tti p(D_j \cdot e)}  
          \sin \pi\Big( \frac{u_j(\delta_+)}{2 \pi \tti} + D_j \cdot d_+\Big)}
        {\sin  \pi\Big(\frac{u_j(\delta_+)}{2 \pi \tti}  + D_j \cdot d_+  
          + p (D_j \cdot e)\Big)}
\prod_{j : j \ne j_-}
\frac{1}{\Gamma\Big(1+\frac{u_j(\delta_+)}{2 \pi \tti} 
                + D_j \cdot d_+  + p (D_j \cdot e)\Big)} 
\end{multline}
  This simplifies dramatically.  
Set $n = k ({-D_{j_-}} \cdot e) + l$ with 
$0 \leq l < ({-D_{j_-}} \cdot e)$,
\[
d_- = d_+ + \frac{D_{j-} \cdot d_+ - l}{{-D_{j_-}} \cdot e}e
\]
and $\delta_- = \{j_1,\ldots,j_{r-1},j_-\}$, 
where $\delta_+ = \{j_1,\ldots,j_{r-1},j_+\}$ 
with $D_{j_1}\cdot e = \cdots = D_{j_{r-1}} \cdot e = 0$.  
Note that $D_{j_-} \cdot d_- =l \in \ZZ_{\ge 0}$ 
but $D_{j_-} \cdot (d_- + e) < 0$, 
and therefore $d_- \in \delta_-^\vee$ but $d_- + e \not \in\delta_-^\vee$.  
Lemma~\ref{lem:weights} implies that:
  \[
  \textstyle
  \frac{u_j(\delta_+)}{2 \pi \tti} + D_j \cdot d_+  + p (D_j \cdot e)
  = 
  \frac{u_j(\delta_-)}{2 \pi \tti} + D_j \cdot d_-  -k (D_j \cdot e)
  \]
and thus the residue \eqref{eq:simplify_this_residue} is: 
  \begin{multline}
    \label{eq:simplified_residue}
    {- e^{\frac{\sigma_-(\delta_-)}{2 \pi \tti} }}
    \sfy^{d_- - d_+ - k e} 
    e^{\frac{\pi \tti w}{D_{j_-}\cdot e} \big(\frac{u_{j_-}(\delta_+)}{2 \pi \tti} 
+ D_{j_-} \cdot (d_+-d_-)\big)}
    \\
\times   
\frac{\sin  \pi\Big(\frac{u_{j_-}(\delta_+)}{2 \pi \tti} + 
D_{j_-} \cdot (d_+-d_-)\Big)  
}
    {({-D_{j_-}} \cdot e) \sin \frac{\pi}{{-D_{j_-}}\cdot e} 
    \Big(\frac{u_{j_-}(\delta_+)}{2 \pi \tti} + D_{j_-} \cdot (d_+-d_-)\Big)}
    \\
\times  
\prod_{\substack{j : D_j \cdot e < 0 \\ j \ne j_-}}
    \frac{\sin  \pi\Big(\frac{u_j(\delta_+)}{2 \pi \tti} + D_j \cdot d_+\Big)}
    {\sin  \pi\Big(\frac{u_j(\delta_-)}{2 \pi \tti} + D_j \cdot d_-\Big)}
    \prod_{j=1}^m
    \frac{1}{\Gamma\Big(1+\frac{u_j(\delta_-)}{2 \pi \tti} 
                  + D_j \cdot d_-  - k (D_j \cdot e)\Big)} 
  \end{multline}
where we used 
$p = \frac{1}{- D_{j_-} \cdot e}
\left( \frac{u_{j_-}(\delta_+)}{2\pi\tti} 
+ D_{j_-} \cdot (d_+- d_-) \right) - k$ and Corollary \ref{cor:weights}. 

Let $f_-$ denote the equivalence class of $d_-$ in $\KK_-/\LL$, 
noting that $(\delta_+,f_+) | (\delta_-,f_-)$ and that 
  \[
  d_+ = f_+ - f_- + N e + d_-
  \]
for some integer $N$.  (Here we used Notation~\ref{not:lift}.)  
The dependence of \eqref{eq:simplified_residue} on $N$ cancels, giving:
  \[
  {- e^{\frac{\sigma_-(\delta_-)}{2 \pi \tti}}}
  \sfy^{d_- - d_+  - ke} 
  \frac{1}{\prod_{j=1}^m \Gamma\Big(1+\frac{u_j(\delta_-)}{2 \pi \tti} 
     + D_j \cdot d_-  - k (D_j \cdot e)\Big)}
  \,
  C_{\delta_+,f_+}^{\delta_-,f_-}
  \]
and minus the sum of these residues gives the analytic continuation 
of \eqref{eq:inner_sum_as_residue_sum}.  
After analytic continuation in $\sfy^e = \sfy_r^{\sfp^+_r \cdot e}$, 
therefore, we have that: 
  \begin{multline*}
    i_{(\delta_+,f_+)}^\star H_+ = 
    \sum_{\substack{(\delta_-,f_-) : \\ (\delta_+,f_+)|(\delta_-,f_-)}}
    \sum_{\substack{d_- \in \delta_-^\vee: \\ d_- + e \not \in \delta_-^\vee \\ [d_-]=f_-}}
    \sfy^{d_-} 
    \sum_{k=0}^\infty
    \frac{e^{\frac{\sigma_-(\delta_-)}{2\pi\tti}} (\sfy^e)^{-k} }
    {\prod_{j =1}^m\Gamma\big(1 + \frac{u_j(\delta_-)}{2 \pi \tti} 
      + D_j \cdot d_- - k D_j \cdot e\big)}
    \,
    C_{\delta_+,f_+}^{\delta_-,f_-}
  \end{multline*}
  Comparing with \eqref{eq:sum_over_d_+} gives the result.
\end{proof}

\subsubsection{Analytic Continuation of the $I$-Function and the Symplectic Transformation $\UU$} 
\label{sec:U} 

Set $\hR_T = H^{\bullet \bullet}_T({\rm pt})$ 
and let $\hS_T$ be the completion of $S_T$ in \S \ref{sec:equiv_int_str}. 
Define an $\hS_T$-linear transformation 
$\UU_H \colon H^{\bullet\bullet}_{T}(IX_-)\otimes_{\hR_T} \hS_T 
\to H^{\bullet\bullet}_{T}(IX_+) \otimes_{\hR_T} \hS_T$
by  
\begin{align}
\label{eq:UH} 
\begin{split}  
\UU_H (\alpha) = 
& \sum_{\substack{(\delta, f) : \\ \delta \in \cA_+ \cap \cA_-}} 
(i_{(\delta,f)}^\star \alpha) 
\cdot 
\frac{\fun_{\delta,f}}{e_{T}(N_{\delta,f})}  \\
& +  \sum_{\substack{(\delta_+,f_+): \\ \delta_+ \in \cA_+ \setminus \cA_-}} 
\sum_{\substack{(\delta_-,f_-): \\ (\delta,f_-) | (\delta_+,f_+)}}
C_{\delta_-,f_-}^{\delta_+,f_-} \cdot (i_{(\delta_-,f_-)}^\star \alpha) 
\cdot 
\frac{\fun_{\delta_+,f_+}}{e_{T}(N_{\delta_+,f_+})} 
\end{split} 
\end{align} 
where $(\delta,f)$ and $(\delta_+,f_+)$ index $T$-fixed points 
in $I X_+$, $\fun_{\delta,f} = i_{(\delta,f)\star} \fun$ and 
$N_{\delta,f} := T_{x_{(\delta,f)}} X^f_+$. 
Then Theorem \ref{thm:analytic_continuation} can be restated as:
\[
H_+ = \UU_H H_-
\]
Define the linear transformation $\UU$ so 
that the following diagram commutes: 
\begin{equation} 
\label{eq:UH_U} 
\begin{aligned}
  \xymatrix{
    H^{\bullet\bullet}_{T}(IX_-)\otimes_{\hR_T} \hS_T \ar[r]^{\UU_H} \ar[d]_{\tPsi'_-} &
    H^{\bullet\bullet}_{T}(IX_+) \otimes_{\hR_T} \hS_T \ar[d]^{\tPsi'_+} \\ 
    H^\bullet_{\CR,T}(X_-) \otimes_{R_T} S_T[\log z](\!(z^{-1/k})\!) \ar[r]^{\UU} & 
    H^\bullet_{\CR,T}(X_-) \otimes_{R_T} S_T[\log z](\!(z^{-1/k})\!) 
  }
\end{aligned}
\end{equation} 
where the vertical maps are defined by  $\tPsi'_\pm(\alpha)= 
z^{-\mu^\pm}z^{\rho^\pm} 
(\hGamma_{X_\pm} \cup (2\pi\tti)^{\frac{\deg_0}{2}}
\inv^* \alpha)$ and 
$k\in \N$ is as in Theorem \ref{thm:U}. 
The relationship \eqref{eq:IH_relation} between the $H$-function 
and the $I$-function implies part (1) of Theorem \ref{thm:U}: 
\begin{equation} 
\label{eq:UI} 
I_+ = \UU I_-.  
\end{equation} 
Since the $I$-function contains neither $\log z$ nor 
non-integral powers of $z$, it follows that $\UU$ is in fact 
a linear transformation:
\[
\UU \colon 
H_{\CR,T}^\bullet(X_-) \otimes_{R_T} S_T(\!(z^{-1})\!) 
\to 
H_{\CR,T}^\bullet(X_+) \otimes_{R_T} S_T(\!(z^{-1})\!)
\] 
Diagram \eqref{eq:UH_U} gives that $\UU$ is automatically degree-preserving. 
We show that $\UU$ satisfies part (2) of 
Theorem \ref{thm:U}. 
Noting that $\sfp_i ^+ = \sfp_i^-$, $i=1,\dots,r-1$ are 
on the wall $W$, it suffices to show that 
$\theta_+(\sfp_i^+) \circ \UU = \UU \circ \theta_-(\sfp_i^-)$ 
for $1\le i\le r-1$. 
This follows from equation \eqref{eq:UI} and the monodromy properties 
of the $I$-functions:
\begin{align*} 
I_{+} \big|_{\sfy_j \mapsto e^{2\pi\tti} \sfy_j} 
& = e^{2\pi\tti \theta_+(\sfp^+_j)/z} I_{+}  \\ 
I_{-} \big|_{\tsfy_j \mapsto e^{2\pi\tti} \tsfy_j} 
& = e^{2\pi\tti \theta_-(\sfp^-_j)/z} I_{-}
\end{align*} 
for $1\le j\le r-1$. 
Note that $\sfy_j \to e^{2\pi\tti} \sfy_j$ corresponds to 
$\tsfy_j \to e^{2\pi\tti} \tsfy_j$ under the change 
\eqref{eq:change_of_variables} of variables.  
It remains to show that: 
\begin{itemize} 
\item $\UU$ is symplectic; 
\item $\UU$ is defined over $R_T(\!(z^{-1})\!)$, 
i.e.~that $\UU$ admits a non-equivariant limit. 
\end{itemize}
These properties follow from the identification of $\UU_H$ with 
the Fourier--Mukai transformation defined in the next section. 
We will discuss these points in \S \ref{sec:end_of_the_proof} below. 

\subsection{The Fourier--Mukai Transform}
\label{sec:FM} 

We now construct a diagram \eqref{eq:common_blowup} canonically associated to the toric birational transformation 
$\varphi \colon X_+ \dashrightarrow X_-$, 
where $\tX$ is a toric Deligne--Mumford stack and 
$f_+$,~$f_-$ are toric blow-ups, and compute the 
Fourier--Mukai transformation:
\begin{align*}
  \FM \colon K^0_T(X_-) \to K^0_T(X_+) && \FM := (f_+)_\star (f_-)^\star
\end{align*}
In~\S\ref{sec:FM_match} below we will see that this transformation coincides, via the equivariant integral structure in 
Definition \ref{def:K_framing}, 
with the transformation $\UU$ from \S\ref{sec:U} 
given by analytic continuation. 

\subsubsection{The Common Blow-Up of $X_+$ and $X_-$}

Recall from \S\ref{sec:stacky_fan} that $X_+$ and $X_-$ are defined 
in terms of an exact sequence:
\[
\xymatrix{
  0 \ar[r] &
  \LL \ar[r] &
  \ZZ^m \ar[r]^\beta & 
  \bN \ar[r] & 
  0 }
\]
where the map $\LL \to \ZZ^m$ is given by $(D_1,\ldots,D_m)$.  
This sequence defines an action of $K = (\Cstar)^r$ on $\CC^m$, 
and $X_\pm = \big[ U_{\omega_\pm} \big/ K\big]$ for appropriate 
stability conditions $\omega_+$,~$\omega_- \in \LL^\vee \otimes \RR$. 
Let $b_1,\ldots,b_m$ denote the images 
of the standard basis elements for $\ZZ^m$ under the map $\beta$.  
Consider now the action of $K \times \Cstar$ on $\CC^{m+1}$ 
defined by the exact sequence:
\[
\xymatrix{
  0 \ar[r] &
  \LL \oplus \ZZ \ar[r] &
  \ZZ^m \oplus \ZZ \ar[r]^-{\tilde{\beta}} & 
  \bN \ar[r] & 
  0 }
\]
where the map $\LL \oplus \ZZ \to \ZZ^m \oplus \ZZ$ 
is given by $(\tD_1,\ldots,\tD_{m+1})$,
\[
\tD_j = 
\begin{cases}
  D_j \oplus 0 & \text{if $j < m+1$ and $D_j \cdot e \leq 0$} \\
  D_j \oplus ({-D_j} \cdot e) & \text{if $j < m+1$ and $D_j \cdot e > 0$} \\
  0 \oplus 1 & \text{if $j=m+1$}
\end{cases}
\]
The map $\tilde{\beta}$ is the direct sum of $\beta$ 
with the map $\ZZ \to \bN$ defined by the element
\[
b_{m+1} = \sum_{j : D_j \cdot e > 0} (D_j \cdot e) b_j
\]
so the images of the standard basis elements for $\ZZ^m \oplus \ZZ$ 
under the map $\tilde{\beta}$ are $b_1,\ldots,b_{m+1}$.  
Consider the chambers $\widetilde{C}_+$,~$\widetilde{C}_-$, 
and~$\widetilde{C}$ in $(\LL \oplus \ZZ)^\vee \otimes \RR$ 
that contain, respectively, the stability conditions
\begin{align*}
  \tomega_+ = (\omega_+,1) &&
  \tomega_- = (\omega_-,1) && \text{and} && 
  \tomega = (\omega_0, - \varepsilon)
\end{align*}
where $\omega_0$ is a point in the relative interior of 
$W \cap \overline{C_+} = W \cap \overline{C_-}$ 
as in \S \ref{sec:birational_transformations}, 
and $\varepsilon$ is a very small positive real number.  
Let $\tX$ denote the toric Deligne--Mumford stack 
defined by the stability condition $\tomega$. 

\begin{lemma} 
\label{lem:Atilde} 
Recall the notation $\cA_\pm$, $\cA_0$, $\cA_0^{\rm thick}$, 
$\cA_0^{\rm thin}$, $M_0$, $M_\pm$ in Lemma \ref{lem:Apm}. 
The set of anticones for the stability conditions $\tomega_\pm$, 
$\tomega$ are given by 
\begin{align*} 
\cA_{\tomega_\pm}  & = 
\left\{ I \sqcup \{m+1\} : I \in \cA_\pm \right\} \\ 
\cA_{\tomega} & = 
\left\{ I \sqcup \{ m+ 1\} : I \in \cA_0^{\rm thick}
\right\} \sqcup 
\left\{I \in\cA_0^{\rm thick} : I \cap M_0 \in \cA_0^{\rm thin} 
\right\}. 
\end{align*} 
\end{lemma} 

\begin{proof}
  Straightforward.
\end{proof}

\begin{lemma} 
\label{lem:toricblowup} 
We have the following statements. 
  \begin{enumerate}
  \item The toric Deligne--Mumford stack corresponding 
to the chamber $\widetilde{C}_+$ is $X_+$.
  \item The toric Deligne--Mumford stack corresponding 
to the chamber $\widetilde{C}_-$ is $X_-$.
  \item There is a commutative diagram as in \eqref{eq:common_blowup} , where:
    \begin{itemize}
    \item $f_+ \colon \tX \to X_+$ is a toric blow-up, 
arising from wall-crossing from the chamber $\widetilde{C}$ to $\widetilde{C}_+$; 
and
    \item $f_- \colon \tX \to X_-$ is a toric blow-up, 
arising from wall-crossing from the chamber $\widetilde{C}$ to $\widetilde{C}_-$.
    \end{itemize}
  \end{enumerate}
\end{lemma}

\begin{proof}
In view of \S\ref{sec:GITdata}, the description of $\cA_{\tomega_\pm}$ 
in Lemma \ref{lem:Atilde} proves (1) and (2). 
The birational transformations $f_+ \colon \tX 
\dashrightarrow X_+$ and $f_- \colon \tX \dashrightarrow X_-$ 
determined by the toric wall-crossings are each morphisms 
which contract the toric divisor defined by 
the $(m+1)$-st homogeneous co-ordinate. 
Indeed, $f_+$ is induced by the identity 
birational map $U_{\tomega} \dashrightarrow U_{\tomega_+}$, 
and a point $(z_1,\dots,z_m, z_{m+1}) \in U_{\tomega_+}$ is equivalent 
to the point $(z_1 z_{m+1}^{l_1},\dots,z_m z_{m+1}^{l_m},1) 
\in U_{\omega_+} \times \{1\}$ 
under the action of the $\Cstar$-subgroup of $K\times \Cstar$ 
corresponding to $e \oplus 1 \in \LL \oplus \ZZ$, 
where we set $l_i := \max(-D_i \cdot e, 0)$ for $1\le i\le m$. 
Therefore $f_+$ is induced by a morphism
\begin{equation} 
\label{eq:morphism_f+}
U_{\tomega} \to U_{\omega_+} \qquad (z_1,\dots,z_m,z_{m+1}) 
\mapsto (z_1 z_{m+1}^{l_1},\dots, z_m z_{m+1}^{l_m})  
\end{equation} 
which is equivariant with respect to the group homomorphism 
(quotient by the $\Cstar$-subgroup given by $e\oplus 1$) 
\begin{equation}
\label{eq:grouphom_f+}
\phi_+ \colon 
K\times \Cstar \to K \qquad (g,\lambda) \mapsto g \cdot \lambda^{-e}. 
\end{equation} 
Using Lemma \ref{lem:Atilde}, one can easily check that the map 
\eqref{eq:morphism_f+} 
indeed sends $U_\tomega$ to $U_{\omega_+}$. 
We obtain a similar description for $f_-$ 
by considering the $\Cstar$-subgroup given by $0\oplus 1 \in \LL \oplus \ZZ$ 
instead of $e\oplus 1$. 
\end{proof}

\begin{remark} 
Torus fixed points on $\tX$ lying on the exceptional 
divisor $\{z_{m+1}=0\}$ of $f_\pm$ correspond to minimal 
anticones $\tdelta \in \cA_{\tomega}$ such 
that $\tdelta \in \cA_0^{\rm thick}$ and 
$\tdelta \cap M_0\in \cA_0^{\rm thin}$. 
These minimal anticones take the form  
\[
\tdelta = \{j_1,\dots,j_{r-1}, j_+, j_-\} 
\]
where $D_{j_1},\dots,D_{j_{r-1}} \in W$, 
$D_{j_+} \cdot e>0$ and $D_{j_-} \cdot e <0$. 
The birational morphism $f_\pm$ maps 
the corresponding torus fixed point $x_{\tdelta}\in \tX$ 
to the torus fixed point $x_{\delta_\pm}\in X_\pm$ with 
\[
\delta_+ = \{j_1,\dots,j_{r-1}, j_+\}\in \cA_+, \quad 
\delta_- = \{j_1,\dots,j_{r-1}, j_-\} \in \cA_-. 
\]
Torus fixed points on $\tX$ lying away from the exceptional 
divisor $\{z_{m+1}=0\}$ corresponds to minimal anticones 
$\tdelta\in \cA_\tomega$ of the form 
$\tdelta = \delta \cup \{m+1\}$, $\delta \in \cA_0^{\rm thick} 
= \cA_+ \cap \cA_-$.  
The morphisms $f_\pm$ are isomorphisms in neighbourhoods 
of these fixed points, and the torus fixed point $x_\tdelta$ maps to 
the fixed point $x_\delta$ in $X_+$ or in $X_-$. 
\end{remark} 

\begin{remark}
The stacky fan $\widetilde{\mathbf{\Sigma}}$ for $\tX$ is obtained 
from the stacky fans $\mathbf{\Sigma}_\pm$ for $X_\pm$ by adding 
the extra ray 
$
b_{m+1} = \sum_{j : D_j \cdot e > 0} (D_j \cdot e) b_j
$
where
\[
\sum_{j : D_j \cdot e > 0} (D_j \cdot e) b_j 
= \sum_{j : D_j \cdot e < 0} ({-D_j} \cdot e) b_j
\]
is a minimal linear relation (or circuit) in $\Sigma_\pm$, 
see Remark \ref{rem:circuit}. 
So our discussion here is a rephrasing in terms of GIT data 
of the material in~\cite[\S5]{Borisov--Horja:FM}.
\end{remark}

\subsubsection{A Basis for Localized $T$-Equivariant $K$-Theory} 
\label{sec:basis_K-theory}

Recall that $T=(\Cstar)^m$ acts on $X_\pm$ through 
the diagonal $T$-action on $\CC^m$. 
We consider the $T$-action on $\tX$ induced from the 
inclusion $T = T\times \{1\} \subset T \times \Cstar$ 
and the $(T\times \Cstar)$-action on $\CC^{m+1}$. 
Then all the maps in the diagram \eqref{eq:common_blowup} 
are $T$-equivariant. 
The $T$-equivariant $K$-groups $K^0_T(X_\pm)$, $K^0_T(\tX)$ 
are modules over $K^0_T({\rm pt})$, which is the ring 
$\ZZ[T]$ of regular functions (over $\ZZ$) on the algebraic torus $T$. 

The $T$-invariant divisor $\{z_i=0\}$ 
on $X_\omega$ defined in \eqref{eq:T-invariant_divisor} 
determines a $T$-equivariant line bundle $\cO(\{z_i=0\})$ 
on $X_\omega$, and we denote the class of this line bundle in 
$T$-equivariant $K$-theory by $R_i$. 
For the spaces $X_+$, $X_-$, and $\tX$
we write these classes as 
\begin{align*}
  & \text{$R^+_1,\ldots,R_m^+\in K^0_T(X_+)$} &
  & \text{$R^-_1,\ldots,R_m^- \in K^0_T(X_-)$} && \text{and} &
  & \text{$\tR_1,\ldots,\tR_{m+1}\in K^0_T(\tX)$.} 
  \intertext{Let us write:}
  & S_j^+ := (R_j^+)^{-1} &
  & S_j^- := (R_j^-)^{-1} && \text{and} &
  & \tS_j := \tR_j^{-1}
\end{align*}
An irreducible $K$-representation 
$p\in \Hom(K, \Cstar) = \LL^\vee$ defines a line bundle 
$L(p)\to X_\omega$: 
\[
L(p) = U_{\omega} \times \CC\big /(z,v) \sim (g \cdot z, p(g) v), \ g\in K
\]
This line bundle is equipped with the $T$-linearization 
$[z,v] \mapsto [t \cdot z, v]$, $t\in T$ and thus defines a class 
in $K_T^0(X_\omega)$. 
We write $L_+(p)$ for the corresponding line bundle on $X_+$ 
and $L_-(p)$ for the corresponding line bundle on $X_-$. 
We have 
\[
R_i^\pm = L_\pm(D_i) \otimes e^{\lambda_i} 
\]
where $e^{\lambda_i}\in \CC[T]$ stands for the irreducible 
$T$-representation given by the $i$th projection $T\to \Cstar$. 
In particular we have $c_1^T(L_\pm(p)) = \theta_\pm(p)$ for the map 
$\theta$ in \eqref{eq:theta}. 
Similarly a character $(p,n)\in \Hom(K\times \Cstar, \Cstar) = 
\LL^\vee \oplus \ZZ$ defines a $T$-equivariant line 
bundle $L(p,n)\to \tX$ and we have:
\begin{align*}
  & \tR_i = L(\tD_i)\otimes e^{\lambda_i} & 1\le i\le m \\
  & \tR_{m+1} = L(\tD_{m+1}) = L(0,1)
\end{align*}
The classes 
$L_\pm(p)$ (respectively the classes $L(p,n)$) generate the equivariant $K$-group 
$K^0_T(X_\pm)$ (respectively $K^0_T(\tX)$) over $\ZZ[T]$. 

Let $\delta_- \in \cA_-$ be a minimal anticone, $x_{\delta_-}$ 
be the corresponding $T$-fixed point on $X_-$, 
$i_{\delta_-} \colon x_{\delta_-} \to X_-$ be the inclusion 
of the fixed point, and $G_{\delta_-}$ be the isotropy group 
of $x_{\delta_-}$.  
We have that $x_{\delta_-} \cong B G_{\delta_-}$, 
and that $i_{\delta_-}^\star R_i = 1$ for $i \in \delta_-$.  
A basis for $K^0_T(X_-)$, after inverting non-zero elements 
of $\ZZ[T]$, is given by: 
\begin{equation}
  \label{eq:fixed_point_basis}
  \big\{(i_{\delta_-})_\star \varrho : 
  \text{$\varrho$ an irreducible representation 
  of $G_{\delta_-}$, $\delta_- \in \cA_-$}\big\}
\end{equation}
We need to specify a $T$-linearization on $(i_{\delta_-})_\star \varrho$.  
Choosing a lift $\hvarrho \in \Hom(K,\Cstar) = \LL^\vee$ 
of each $G_{\delta_-}$-representation $\varrho \colon G_{\delta_-} \to \Cstar$, 
we write any element in \eqref{eq:fixed_point_basis} in the form: 
\begin{equation*}
  \label{eq:fixed_point_basis_using_Koszul}
  e_{\delta_-,\varrho} := L_- (\hvarrho) 
  \prod_{i \not \in \delta_-} \big(1 - S_i^-\big)
\end{equation*}
Then $\{e_{\delta_-,\varrho}\}$ is a basis for the localized $T$-equivariant 
$K$-theory of $X_-$.  There is an entirely analogous basis 
$\{e_{\delta_+,\varrho}\}$ for the localized 
$T$-equivariant $K$-theory of $X_+$.  
We will describe the action of the Fourier--Mukai transform 
in terms of these bases.

\subsubsection{Computing the Fourier--Mukai Transform}  
Consider the diagram \eqref{eq:common_blowup}  and 
the associated Fourier--Mukai transform 
$\FM \colon K^0_T(X_-) \to K^0_T(X_+)$.  In this section we prove:

\begin{theorem} \label{thm:Fourier--Mukai}
If $\delta_- \in \cA_-$ is a minimal anticone such that $\delta_- \in \cA_+$ then
  \[
  \FM(e_{\delta_-,\varrho}) = e_{\delta_-,\varrho}
  \]
where on the left-hand side of the equality $\delta_-$ is regarded as 
a minimal anticone for $X_-$ and on the right-hand side $\delta_-$ 
is regarded as a minimal anticone for $X_+$.  
If $\delta_-$ is a minimal anticone in $\cA_-$ 
such that $\delta_- \not \in \cA_+$ then 
$\FM(e_{\delta_{-},\varrho})$ is equal to 
  \[ 
  \frac{1}{l} 
  \sum_{t\in \cT}
  \left(
    \frac{1-S_{j_{-}}^+}{1-t^{-1} }
    \cdot 
    L_+(\hvarrho) t^{\hvarrho \cdot e} 
    \cdot 
    \prod_{\substack{j\notin \delta_{-}\\ D_j \cdot e <0}} 
    (1-S_j^+)
    \cdot 
    \prod_{\substack{i\notin \delta_{-} \\ D_i\cdot e\ge 0}}  
   \big(1-t^{-D_i \cdot e}S_i^+\big)
   \right)
  \]
where $j_-$ is the unique element of $\delta_-$ such 
that $D_{j_-} \cdot e<0$, $l = -D_{j_-} \cdot e$ and 
  \[
  \cT =\left\{\zeta \cdot (R_{j_{-}}^+)^{1/l} :  
\zeta \in\bmu_l \right\}. 
  \]
\end{theorem}

\begin{remark} 
We have 
\[
\frac{1}{l} \sum_{t\in \cT} t^n =  
\begin{cases} 
(R_{j_-}^+)^{n/l} & \text{if $l$ divides $n$;} \\
0 & \text{otherwise.}  
\end{cases}
\]
Thus $\frac{1}{l} \sum_{t\in \cT} f(t)$ makes 
sense as an element $K_T^0(X_+)$ for a Laurent polynomial 
$f(t)$ in $t$. Note that each summand appearing 
in the formula for $\FM(e_{\delta_-,\varrho})$ is in fact a Laurent 
polynomial in $t$, since the factor $1-t^{-1}$ divides 
$1- S_{j_-}^+ = 1 - t^{-l}$. 
\end{remark} 

Borisov--Horja have computed how non-equivariant versions of 
the classes $R_i^-$ change under 
pullback~\cite[Proposition~8.1]{Borisov--Horja:K}.  
We have parallel results in the equivariant setting. 

\begin{proposition}
  \label{pro:pullback}
  For $p \in \LL$, we have:
  \begin{align*} 
  f_-^\star(L_-(p)) = L(p, 0)  
  && \text{and} && 
  f_+^\star(L_+(p)) = L(p,-p\cdot e) \\ 
  \intertext{Let 
$k_i := \max(D_i \cdot e,0)$ and $l_i := \max({-D_i \cdot e},0)$.  
Then:} 
    f_-^\star R_i^- = \tR_i \tR_{m+1}^{k_i}
    && \text{and} &&
    f_+^\star R_i^+ = \tR_i \tR_{m+1}^{l_i}. 
  \end{align*}
\end{proposition}
\begin{proof} 
These statements follows from the description of $f_\pm \colon 
\tX \to X_\pm$ in the proof of Lemma \ref{lem:toricblowup};
see \eqref{eq:morphism_f+} and \eqref{eq:grouphom_f+}. 
\end{proof} 

We now analyze the push-forward of classes 
supported on torus fixed points of $\tX$. 

\begin{proposition}
\label{pro:pushforward}
Consider minimal anticones
\begin{align*}
    \tdelta &=\{j_1,\ldots,j_{r-1},j_+,j_-\} \in 
\cA_\tomega && \text{for $\tX$} \\
    \delta_+ &= \{j_1,\ldots,j_{r-1},j_+\} \in \cA_+  && \text{for $X_+$}
\end{align*}
such that $\{j_1,\dots,j_{r-1},j_+,j_-\} \subset \{1,\dots,m\}$, 
$D_{j_1} \cdot e = \cdots = D_{j_{r-1}} \cdot e = 0$, 
$D_{j_-} \cdot e<0$ 
and $D_{j_+} \cdot e > 0$.  
Let $i_{\tdelta} \colon BG_{\tdelta} \to \tX$ and 
$i_{\delta_+} \colon BG_{\delta_+} \to X_+$ 
denote the inclusions of the corresponding $T$-fixed points 
and let $f_{+,\tdelta} \colon BG_{\tdelta} \to B G_{\delta_+}$ 
denote the map induced on the fixed points: 
\[
\xymatrix{
  x_\tdelta \ar@{=}[r]&  BG_{\tdelta} \ar[r]^{i_\tdelta} \ar[d]_{f_{+,\tdelta}} & \tX \ar[d]^{f_+} \\ 
  x_{\delta_+} \ar@{=}[r]&  BG_{\delta_+} \ar[r]^{i_{\delta_+}} & X_+ 
}
\]
\begin{enumerate}
\item The map $f_{+,\tdelta}$ 
exhibits 
$BG_{\tdelta}$ as a $\bmu_l$-gerbe over $BG_{\delta_+}$, 
where $l = {-D_{j_-}} \cdot e$.
\item We have:
  \[
  (f_{+,\tdelta})_\star (i_{\tdelta})^\star L(p,n) =
  \begin{cases}
    (i_{\delta_+})^\star L_+(p) (R_{j_{-}}^+)^{(p\cdot e + n)/l} 
&  \text{if $l$ divides $p\cdot e + n$;} \\
    0 & \text{otherwise.} 
  \end{cases}
  \]
\item   Let $g$ be a Laurent polynomial in $m+1$ variables.  Then:
  \[
  (f_{+,\tdelta})_\star (i_\tdelta)^\star 
L(p,n) 
g\big(\tR_1,\ldots,\tR_{m+1}\big)
  =
  (i_{\delta_+})^\star 
  \frac{1}{l}
  \sum_{t\in \cT} 
  L_+(p) t^{p\cdot e + n} 
  g\big(t^{-l_1} R_1^+,\ldots,t^{-l_m} R_m^+,t\big)
  \]
\end{enumerate}
\end{proposition}
\begin{proof} 
The stabilizers $G_\tdelta$ and $G_{\delta_+}$ 
are given, as subgroups of $K\times \Cstar$ and $K$, 
by 
\begin{align*} 
G_\tdelta & = \{(g,\lambda) \in K\times \CC^\times : 
\text{$D_j(g) \lambda^{-D_j \cdot e} = 1$ for all $j\in \delta_+$, 
$D_{j_-}(g) = 1$} \} \\ 
G_{\delta_+} & = \{ h \in K : \text{$D_j(h) = 1$ for all $j \in \delta_+$}\} 
\end{align*} 
where we regard $D_j$ as a character of $K$. 
The homomorphism $G_\tdelta \to G_{\delta_+}$ is 
induced by $\phi_+ \colon (g,\lambda) \mapsto h= g \cdot \lambda^{-e}$ 
in \eqref{eq:grouphom_f+}. 
The kernel of the homomorphism is 
$\{(\lambda^{e},\lambda): \lambda\in \bmu_l\}$ 
and we obtain an exact sequence:
\[
\xymatrix{
  1 \ar[r] & \bmu_l \ar[r] & G_{\tdelta} \ar[r] & G_{\delta_+} \ar[r] &1
}
\]
Therefore $f_{+,\tdelta}$ exhibits $BG_{\tdelta}$ as a 
$\bmu_l$-gerbe over $BG_{\delta_+}$. 

For part (2), notice that $(f_{+,\tdelta})_\star$ maps 
a $G_{\tdelta}$-representation to its $\bmu_l$-invariant part. 
The character $(p,n)\in \Hom(K\times \Cstar,\Cstar)$ induces 
a $\bmu_l$-character $\lambda \mapsto \lambda^{p\cdot e + n}$ 
via the inclusion 
$\bmu_l \subset G_{\tdelta} \subset K\times \Cstar$.  
Therefore $(f_{+,\tdelta})_\star (\iota_{\tdelta})^\star L(p,n)$ 
vanishes if $l$ does not divide $p \cdot e + n$. 
On the other hand, Proposition \ref{pro:pullback} 
gives $(f_+)^\star R_{j_-}^+ = \tR_{j_-} \tR_{m+1}^l$ 
and hence, if $l$ divides $p\cdot e +n$, 
\begin{align*} 
(f_{+,\tdelta})^\star (i_{\delta_+})^\star 
L_+(p) (R_{j_-}^+)^{(p\cdot e+n)/l} 
& = (i_{\tdelta})^\star L(p, -p\cdot e) (\tR_{j_-})^{(p\cdot e+n)/l} 
(\tR_{m+1})^{p\cdot e + n} \\ 
& = (i_{\tdelta})^\star L(p,-p \cdot e) (\tR_{m+1})^{p\cdot e + n} 
= (i_{\tdelta})^\star L(p,n). 
\end{align*} 
Therefore the Projection Formula gives 
$(f_{+,\tdelta})_\star  (i_{\tdelta})^\star L(p,n) 
= (i_{\delta_+})^\star 
L_+(p) (R_{j_-}^+)^{(p\cdot e+n)/l}$. 
This proves (2). 

For part (3) it suffices to take $g$ to be a monomial: 
$g(\tR_1,\dots,\tR_{m+1}) = \prod_{i=1}^{m+1} \tR_i^{n_i}$. 
In this case: 
\begin{equation} 
\label{eq:Lpn_g}
L(p,n) g(\tR_1,\dots,\tR_{m+1}) = 
L(p+ \textstyle\sum_{i=1}^{m} n_i D_i, 
n + n_{m+1} - \sum_{i=1}^m n_i k_i ) \otimes e^{\sum_{i=1}^m n_i \lambda_i}
\end{equation} 
Part (2) can be restated as:
\[
(f_{+,\tdelta})_\star  (i_{\tdelta})^\star L(p,n) 
=(i_{\delta_+})^\star \frac{1}{l} \sum_{t\in \cT} 
 L_+(p) t^{p\cdot e +n} 
\]
Combining this with \eqref{eq:Lpn_g} yields (3). 
\end{proof}

\begin{proof}[Proof of Theorem~\ref{thm:Fourier--Mukai}]
  Suppose first that $\delta_- \in \cA_+ \cap \cA_-$.  
Then, as discussed, $\varphi$ gives an isomorphism 
between neighbourhoods of the fixed points corresponding 
to $\delta_-$.  Thus $\FM(e_{\delta_-,\varrho}) = e_{\delta_-,\varrho}$.

Suppose now that $\delta_- \in \cA_-$ but $\delta_- \not \in \cA_+$, 
so that $\delta_- = \{j_1,\ldots,j_{r-1},j_-\}$ with 
$D_{j_1} \cdot e = \cdots = D_{j_{r-1}} \cdot e = 0$ 
and $D_{j_-} \cdot e < 0$. Proposition~\ref{pro:pullback} gives:
\[
(f_-)^\star e_{\delta_-,\varrho} 
L(\hvarrho, 0) 
\prod_{i \not \in \delta_-} \big(1 - \tS_{m+1}^{k_i} \tS_i\big)
\]
where the index $i$ in the product satisfies $i\le m$.  
This restricts to zero at a fixed point $x_{\tdelta} \in \tX$ 
unless $x_{\tdelta}$ is in $f_+^{-1}(x_{\delta_-})$, that is, 
unless $\tdelta$ has the form $\delta_- \cup \{j_+\}$ 
with $D_{j_+} \cdot e > 0$.  
The Localization Theorem in $T$-equivariant $K$-theory~\cite{Coates--Iritani--Jiang--Segal} gives:
\begin{equation}
  \label{eq:actually_a_polynomial}
  (f_-)^\star e_{\delta_-,\varrho} = 
  \sum_{\tdelta}
  (i_{\tdelta})_\star (i_{\tdelta})^\star 
  \left[
    \frac{
    L(\hvarrho,0) 
     \prod_{i \not \in \delta_-} \big(1 - \tS_{m+1}^{k_i} \tS_i\big)
    }
{(1-\tS_{m+1}) 
\prod_{j \not \in \delta_-, j \ne j_+} (1-\tS_j) 
 }
\right]
\end{equation}
where $i$,~$j\le m$ and the sum runs over $\tdelta = \delta_- \cup \{j_+\}$ 
such that $D_{j_+} \cdot e > 0$.  
Restricted to such a $T$-fixed point, 
$\tS_{j_+}$ becomes trivial, so the numerator 
in \eqref{eq:actually_a_polynomial} contains a factor 
$(i_{\tdelta})^\star (1-\tS_{m+1}^{k_{j_+}})$ 
that is divisible by $(i_{\tdelta})^\star (1-\tS_{m+1})$.  
Thus \eqref{eq:actually_a_polynomial} depends polynomially 
on $\tS_{m+1}$. Now: 
\begin{align*}
  (f_+)_\star (f_-)^\star e_{\delta_-,\varrho}
   & =
   \sum_{\delta_+ : \delta_+ | \delta_-}
   (i_{\delta_+})_\star 
   (f_{+,\tdelta})_\star 
   (i_\tdelta)^\star 
   \left[
     \frac{
     L(\hvarrho, 0) 
     \prod_{i \not \in \delta_-} \big(1 - \tS_{m+1}^{k_i} \tS_i\big)
     }{
       (1-\tS_{m+1}) 
\prod_{j \not \in \delta_-, j \ne j_+} (1-\tS_j) 
     }
   \right] \\
   &=
   \sum_{\delta_+ : \delta_+ | \delta_-}
   (i_{\delta_+})_\star (i_{\delta_+})^\star 
    \left[ \frac{1}{l} \sum_{t \in \cT}
     \frac{
L_+(\hvarrho) t^{\hvarrho\cdot e} 
\prod_{i \not \in \delta_-} \big(1 - t^{l_i-k_i} S_i^+\big)
     }
{(1-t^{-1}) \prod_{
       j \not \in \delta_-, j \ne j_+} (1-t^{l_j} S_j^+) }
   \right] 
\end{align*}
where we used part (3) of Proposition~\ref{pro:pushforward}.  
This is:
\[
\sum_{\delta_+ : \delta_+ | \delta_-}
(i_{\delta_+})_\star (i_{\delta_+})^\star 
\left[
\frac{ \frac{1}{l} \sum_{t \in \cT} 
\frac{1-S_{j_{-}}^+}{1-t^{-1}} 
\cdot 
L_+(\hvarrho) t^{\hvarrho\cdot e} 
\cdot 
\prod_{j \not \in \delta_- } 
\big(1 - t^{-k_j} S_j^+\big)
} 
{
\prod_{j \not \in \delta_+} (1 - S_j^+)
}
\right].  
\]
Applying the Localization Theorem again gives the result. 
Here we need to check that the restriction of 
\[
\frac{1-S_{j_{-}}^+}{1-t^{-1}} 
\cdot 
\prod_{j \not \in \delta_-} 
\big(1 - t^{-k_j} S_j^+\big)
\]
to the fixed point corresponding to $\delta \in \cA_+ \cap \cA_-$ 
vanishes. If there exists $j\in \delta$ with $j\notin \delta_-$ 
and $D_j \cdot e \le 0$ then  the restriction vanishes since 
$i_\delta^\star S_j^+= 1$. 
Otherwise one has $\delta \setminus \delta_- \subset M_+$. 
In this case $j_-\in \delta$ and there exists $j_0\in \delta 
\cap M_+$. Thus the restriction contains the factor 
\[
i_\delta^\star \left[ \frac{1-S_{j_-}^+}{1-t^{-1}} 
(1 - t^{-D_{j_0} \cdot e} S_{j_0}^+) \right] 
=  i_\delta^\star 
\left[ (1-S_{j_-}^+) \frac{1-t^{-D_{j_0} \cdot e}}{1-t^{-1}} \right] 
\]
which vanishes.
\end{proof}

\subsection{The Fourier--Mukai Transform Matches With Analytic Continuation} 
\label{sec:FM_match}
We now show that the analytic continuation formula 
in Theorem~\ref{thm:analytic_continuation} matches 
with the Fourier--Mukai transform in Theorem~\ref{thm:Fourier--Mukai}. 
More precisely we show: 
\begin{theorem} 
\label{thm:UH-FM} 
Let $\UU_H$ be the linear transformation in \S \ref{sec:U} 
given by the analytic continuation of $H$-functions. 
Then $\UU_H$ induces a map $\UU_H \colon 
H^{\bullet\bullet}_{T}(IX_-) \to H^{\bullet\bullet}_{T}(IX_+)$ 
and the following diagram commutes: 
\begin{equation} 
\label{eq:FM_UH} 
\begin{aligned}
  \xymatrix{
    K_T^0(X_-) \ar[r]^{\FM} \ar[d]_{\tch} & K_T^0(X_+) \ar[d]_{\tch} \\ 
    H^{\bullet\bullet}_{T}(IX_-) \ar[r]^{\UU_H} & H^{\bullet\bullet}_{T}(IX_+)
  }
\end{aligned}
\end{equation}  
\end{theorem}

We start by computing the Chern characters
of certain line bundles. It is easy to see that: 
\begin{align*} 
\tch(L_\pm(\hvarrho)) & = \bigoplus_{f\in \KK_\pm/\LL} 
e^{2\pi\tti\hvarrho\cdot f} e^{\theta_\pm(\hvarrho)} \fun_f\\
\tch(S_j^\pm) & = \bigoplus_{f\in \KK_\pm/\LL} 
e^{-2\pi\tti D_j \cdot f} e^{-u_j} \fun_f
\end{align*} 
In view of this, we define 
\[
\tch(t) := \bigoplus_{f\in \KK_+/\LL} 
\zeta e^{2\pi\tti D_{j_-} \cdot f/l} e^{u_{j_-}/l}\fun_f 
\]
for $t= \zeta (R_{j_-}^+)^{1/l}\in \cT$ appearing 
in Theorem \ref{thm:Fourier--Mukai}. 
Here we fix lifts $\KK_+/\LL \to \KK_+$, 
$\KK_-/\LL \to \KK_-$ as in Notation \ref{not:lift} 
and identify $f \in \KK_+/\LL$ with its lift 
in $\KK_+$.

\begin{lemma} 
\label{lem:chernchar_matching} 
Suppose that $(\delta_+,f_+)$ indexes a $T$-fixed point on $X_+$, 
that $(\delta_-,f_-)$ indexes a $T$-fixed point on $X_-$,
and that $(\delta_+,f_+)|(\delta_-,f_-)$. 
Let $j_-\in \delta_-$ be the unique index such that $D_{j_-} \cdot e<0$ 
and write $l = {-D_{j_-}} \cdot e$. 
Setting $t = e^{- 2\pi\tti D_{j_-} \cdot f_-/l} 
(R_{j_-}^+)^{1/l}$, we have:
\begin{align}
\label{eq:t_chern}
\begin{split}  
i_{(\delta_+,f_+)}^\star 
\tch\left(L_+(\hvarrho) t^{\hvarrho\cdot e}\right) & = 
i_{(\delta_-,f_-)}^\star \tch\left(L_-(\hvarrho)\right)
\\ 
i_{(\delta_+,f_+)}^\star
\tch \left(S_j^+ t^{-D_j \cdot e} \right) 
& =  i_{(\delta_-,f_-)}^\star 
\tch(S_j^-)
\end{split} 
\end{align} 
We also have:  
\begin{align} 
\label{eq:ch_e_deltavarrho}
i_{(\delta_+,f_+)}^\star
\tch\left[ 
L_+(\hvarrho) t^{\hvarrho\cdot e} 
\prod_{j\notin \delta_-} (1- t^{-D_j \cdot e} S_j^+) 
\right] 
& = 
i_{(\delta_-,f_-)}^\star 
\tch(e_{\delta_-,\varrho}) \\
\intertext{and:}
\label{eq:coefficient_match}
 i_{(\delta_+,f_+)}^\star \orbich \left[ \frac{1-S_{j_{-}}^+}{l (1-t^{-1})}
    \cdot \prod_{\substack{j\notin \delta_{-}\\
        D_j \cdot e <0}}\frac{1-S_j^+}{1-S_j^+t^{{-D_j} \cdot e}}
  \right]
& =  C_{(\delta_+,f_+)}^{(\delta_-,f_-)}  
\end{align} 
where $C_{(\delta_+,f_+)}^{(\delta_-,f_-)}$ 
are the coefficients appearing in Theorem \ref{thm:analytic_continuation}. 
\end{lemma}

\begin{proof}
This is just a calculation. 
Recall from Notation \ref{not:lift} that 
$f_- = f_+ + \alpha e$ for some $\alpha \in \QQ$. 
Then $D_j \cdot (f_+ - f_-) = {- \alpha} D_j \cdot e$ 
and $D_{j_-} \cdot (f_+ - f_-) = l \alpha$.  
The formulae \eqref{eq:t_chern} easily follow from 
Lemma \ref{lem:weights} and \eqref{eq:weights_theta}. 
The formula \eqref{eq:ch_e_deltavarrho} is an easy consequence of 
\eqref{eq:t_chern}. 
To see \eqref{eq:coefficient_match}, we calculate, using \eqref{eq:t_chern}, 
\begin{align*} 
\text{LHS} & = \frac{1}{l} 
\frac{1- e^{-u_{j_-}(\delta_+) - 2\pi \tti(D_{j_-} \cdot f_+)}}
{1- e^{- \frac{1}{l}( u_{j_-}(\delta_+) + 2\pi\tti D_{j_-} \cdot (f_+-f_-))}} 
\prod_{\substack{j\notin \delta_-\\ D_j \cdot e<0}}
\frac{1- e^{-u_j(\delta_+) - 2\pi\tti D_j\cdot f_+}}
{1- e^{-u_j(\delta_-) - 2\pi\tti D_j \cdot f_-}} \\ 
& = \frac{1}{l} 
\frac{\sin \pi 
        \left( \frac{u_{j_-}(\delta_+)}{2\pi\tti} + D_{j_-} \cdot (f_+-f_-)
        \right)}
{\sin \frac{\pi}{l} \left( \frac{u_{j_-}(\delta_+)}{2\pi\tti} 
+ D_{j_-}\cdot (f_+ - f_-) \right) }
\prod_{\substack{j\notin \delta_- \\ D_j \cdot e<0}}
\frac{\sin \pi \left( \frac{u_j(\delta_+)}{2\pi\tti} + D_j \cdot f_+\right) }
{\sin \pi \left ( \frac{u_j(\delta_-)}{2\pi\tti} + D_j \cdot f_- \right)} \\ 
& \quad \times e^{
- \frac{1}{2}(1-\frac{1}{l}) (u_{j_-}(\delta_+ ) + 2\pi\tti 
D_{j_-}\cdot (f_+-f_-) ) +  
\sum_{j\notin \delta_-, D_j\cdot e<0} 
\left( 
\frac{1}{2}(u_j(\delta_-) - u_j(\delta_+)) + \pi\tti  D_j \cdot (f_- - f_+)
\right)}
\end{align*} 
where we used the fact that $D_{j_-} \cdot f_- \in \ZZ$. 
Using Lemma \ref{lem:weights} again to calculate the 
exponential factor, we arrive at the expression for 
$C_{(\delta_+,f_+)}^{(\delta_-,f_-)}$ 
in Theorem \ref{thm:analytic_continuation}. 
\end{proof} 

\begin{proof}[Proof of Theorem \ref{thm:UH-FM}] 
We first show that the commutative diagram holds over $\hS_T$. 
Then it follows that $\UU_H$ has a non-equivariant 
limit, as $\FM$ does. 
Consider the element $e_{\delta,\varrho} 
\in K_T^0(X_-)$ with $\delta \in \cA_+ \cap \cA_-$. 
Theorem \ref{thm:Fourier--Mukai} and the definition 
\eqref{eq:UH} of $\UU_H$ show that 
\[
\tch( \FM ( e_{\delta,\varrho}) ) = \tch(e_{\delta,\varrho}) 
= \UU_H ( \tch(e_{\delta,\varrho})). 
\]
Consider now $e_{\delta_-,\varrho}\in K_T^0(X_-)$ for 
$\delta_-\in \cA_-\setminus \cA_+$. 
It is clear that $\tch(\FM(e_{\delta_-,\varrho}))$ is supported only 
on fixed points $x_{(\delta_+,f_+)}\in IX_+$ such that 
$\delta_+| \delta_-$. 
By the definition \eqref{eq:UH} of $\UU_H$, it suffices 
to show that:
\begin{equation} 
\label{eq:i_ch_FM}
i_{(\delta_+,f_+)}^\star \tch\left(\FM(e_{\delta_-,\varrho})\right) 
= \sum_{\substack{f_-\in \KK_-/\LL: \\ 
(\delta_+,f_+) | (\delta_-,f_-)}} 
C_{\delta_+,f_+}^{\delta_-,f_-} \cdot 
i_{(\delta_-,f_-)}^\star \tch(e_{\delta_-,\varrho})
\end{equation} 
We may rewrite the result in 
Theorem \ref{thm:Fourier--Mukai} as 
\begin{equation} 
\label{eq:FM_restated} 
\FM(e_{\delta_-,\varrho}) = 
\frac{1}{l} 
\sum_{t\in \cT} 
\left( 
\frac{1-S_{j_-}^+}{1-t^{-1}} 
\prod_{\substack{j\notin \delta_- \\ D_j \cdot e<0}} 
\frac{1-S_j^+}{1-t^{-D_j\cdot e} S_j^+} 
\cdot 
L_+(\hvarrho) t^{\hvarrho\cdot e} 
\prod_{i\notin \delta_-} 
(1- t^{-D_i \cdot e} S_i^+)
\right) 
\end{equation} 
We have a one-to-one correspondence between the 
 index of summation $f_-$ in \eqref{eq:i_ch_FM} 
and the index of summation $t\in \cT$ in \eqref{eq:FM_restated} 
given by 
\[
f_- \ \longleftrightarrow \ t = e^{-2\pi\tti D_{j_-}\cdot f_-/l} (R_{j_-}^+)^{1/l}
\]
where $j_-\in \delta_-$ is the unique element satisfying 
$D_{j_-} \cdot e<0$ and $l = - D_{j_-}\cdot e$. 
Therefore \eqref{eq:i_ch_FM} follows from \eqref{eq:FM_restated}, 
\eqref{eq:ch_e_deltavarrho} and \eqref{eq:coefficient_match}. 
The Theorem is proved.
\end{proof}

\subsection{Completing the Proof of Theorem \ref{thm:U}} 
\label{sec:end_of_the_proof} 
Combining the commutative diagrams \eqref{eq:UH_U} 
and \eqref{eq:FM_UH}, we obtain the commutative diagram 
\eqref{eq:FM_U} in Theorem \ref{thm:U}. 
Since the Fourier--Mukai transformation $\FM$ can be defined 
non-equivariantly, $\UU$ also admits a non-equivariant limit. 
Finally we show that $\UU$ is symplectic, i.e.~that $(\UU(-z) \alpha, \UU(z)\beta) 
= (\alpha,\beta)$ for all $\alpha$,~$\beta$. 
Since $\FM$ is induced by an equivalence of derived categories~\cite{Coates--Iritani--Jiang--Segal}, 
it preserves the Euler pairing $\chi(E,F)$ 
given in \eqref{eq:Euler_pairing}. 
The proof of Proposition \ref{prop:K_framing_pairing} 
shows that the vertical maps $\tPsi_\pm$ in \eqref{eq:FM_U} 
preserve the pairing in the sense that: 
\[
\left(\tPsi_\pm(E)|_{z\to e^{-\pi\tti} z}, \tPsi_\pm(F) \right) 
=  \chi_z(E,F).  
\]
The commutative diagram \eqref{eq:FM_U} now shows that 
$\UU$ is symplectic.  This completes the proof of Theorem~\ref{thm:U}.

\begin{remark}
  The reader who would prefer to prove that the transformation $\UU$ is symplectic without using the machinery of derived categories can argue as follows.  It suffices to show 
that the transformation $\FM$ preserves the Euler pairing on the equivariant $K$-groups. 
Using Grothendieck duality, one finds that the adjoint of $\FM$ 
with respect to the Euler pairing is given by 
\[
\FM^*(\alpha) =  f_{-\star} ( (f_-^\star K_-^{-1})\tK \otimes f_+^\star \alpha) 
\]
(see, for example,~\cite[Lemma 1.2]{Bondal--Orlov:semiorth})
where $K_-$ and $\tK$ are respectively the canonical line bundles
of $X_-$ and $\tX$. 
It suffices to see that $\FM^*$ corresponds to 
the analytic continuation along the path inverse 
to $\gamma$ in Figure \ref{fig:path_ancont}. 
Consider the Fourier--Mukai transformation in the 
reverse direction: 
\[
\FM' = f_{-\star} f_+^\star \colon K_T^0(X_+) \to K_T^0(X_-)
\]
Exchanging the roles of $X_+$ and $X_-$ in Theorem \ref{thm:U}, 
we find that $\FM'$ corresponds to the analytic continuation 
along the path $\gamma'$ in Figure \ref{fig:path_rev_ancont}. 
We claim that the difference between the paths $\gamma^{-1}$ 
and $\gamma'$ exactly matches with the difference between 
the Fourier--Mukai transformations $\FM^*$ and $\FM'$. 
Using $\tK = \tS_1 \cdots \tS_{m+1}$, $K_- = S_1^- \cdots S_m^-$ 
and Proposition \ref{pro:pullback}, we have 
\[
(f_-^\star K_-^{-1})\tK  = L(0,w) 
= (f_+^\star L_+(-w p)) \otimes (f_-^\star L_-(wp)) 
\] 
for $p\in \LL$ with $p\cdot e = 1$. Therefore 
\[
\FM^* = L_-(wp) \circ \FM' \circ L_+(-wp).  
\]
On the other hand, one can check that the monodromy of the $K$-theoretic flat 
section $\frs(E)$ with respect to the shift $T_{-2\pi\tti w p} \colon 
\log \sfy^d \mapsto \log\sfy^d - 2\pi\tti w (p\cdot d)$ 
corresponds to $E \mapsto L_\pm(wp)\otimes E$ 
(cf.~the Galois action in \cite[Proposition 2.10 and equation~61]{Iritani}), 
and we have the equality
\[
\gamma^{-1} = T_{-2\pi\tti w p} \circ \gamma' \circ T_{2\pi\tti w p}
\]
in the fundamental groupoid of $\{\sfy^e \in \Cstar : \sfy^e \neq \frc\}$.  This shows that $\FM^*$ corresponds to the analytic continuation along $\gamma^{-1}$.

\begin{figure}[htbp]  
\centering
\includegraphics[bb=131 545 483 729]{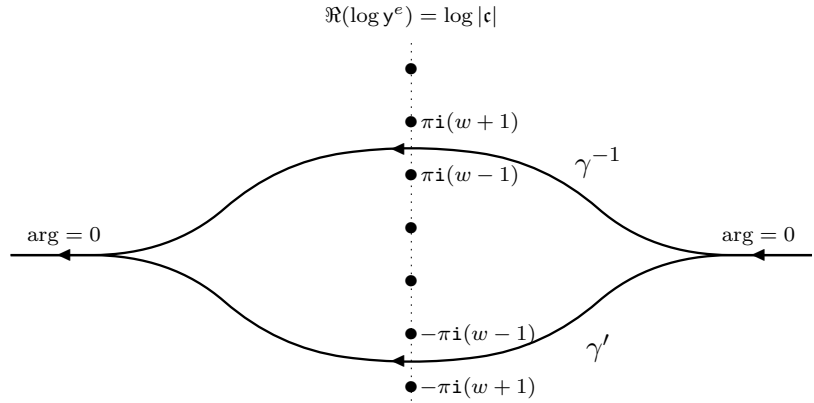} 
\caption{The paths $\gamma'$ and $\gamma^{-1}$}
\label{fig:path_rev_ancont}
\end{figure}
\end{remark}

\section{Toric Complete Intersections}
\label{sec:CRC_ci}

We now turn to the Crepant Transformation Conjecture for toric complete intersections.  Consider toric Deligne--Mumford stacks $X_\pm$ of the form $\big[\CC^m /\!\!/_\omega K \big]$, where $K = (\Cstar)^r$ is a complex torus, and consider a $K$-equivalence $\varphi \colon X_+ \dashrightarrow X_-$ determined by a wall-crossing in the space of stability conditions~$\omega$ as in~\S\ref{sec:wall-crossing}.   We use notation as there, so that $\LL = \Hom(\Cstar,K)$ is the lattice of cocharacters of $K$; the space of stability conditions is $\LL^\vee \otimes \RR$; and the birational map $\varphi$ is induced by the wall-crossing from a chamber $C_+ \subset \LL^\vee \otimes \RR$ to a chamber $C_- \subset \LL^\vee \otimes \RR$, where $C_+$ and $C_-$ are separated by a wall $W$.  Consider characters $E_1,\ldots,E_k$ of $K$ such that:
\begin{equation}
  \label{eq:conditions_on_line_bundles}
  \parbox{0.92\textwidth}{
    \begin{itemize}
    \item each $E_i$ lies in $W \cap \overline{C_+} = W \cap \overline{C_-}$;
    \item for each $i$, the line bundle $L_{X_+}(E_i) \to X_+$ corresponding to $E_i$ is a pull-back from the coarse moduli space $|X_+|$;
    \item for each $i$, the line bundle $L_{X_-}(E_i) \to X_-$ corresponding to $E_i$ is a pull-back from the coarse moduli space $|X_-|$;
    \end{itemize}
  }
\end{equation}
where $L_{X_\pm}(E_i)$ are the line bundles on $X_\pm$ associated to the character $E_i$ in~\S\ref{sec:basis_K-theory}.  Let:
\begin{align*}
  E_+ := \bigoplus_{i=1}^k L_{X_+}(E_i) &&
  E_- := \bigoplus_{i=1}^k L_{X_-}(E_i)
\end{align*}
Let $s_+$,~$s_-$ be regular sections of, respectively, the vector bundles $E_+ \to X_+$ and $E_- \to X_-$ such that:
\begin{itemize}
\item  $s_+$ and $s_-$ are compatible via $\varphi \colon X_+ \dashrightarrow X_-$;
\item the zero loci of $s_\pm$ intersect the flopping locus of $\varphi$ transversely;
\end{itemize}
and let $Y_+ \subset X_+$, $Y_- \subset X_-$ be the complete intersection substacks defined by $s_+$,~$s_-$.  The birational transformation $\varphi$ then induces a $K$-equivalence $\varphi \colon Y_+ \dashrightarrow Y_-$.  In this section we establish the Crepant Transformation Conjecture for $\varphi \colon Y_+ \dashrightarrow Y_-$.

\subsection{The Ambient Part of Quantum Cohomology}
\label{sec:ambient_part}
Under our standing hypotheses on the ambient toric stacks $X_\pm$, the complete intersections $Y_\pm$ automatically have semi-projective coarse moduli spaces, and so the (non-equivariant) quantum products on $H_{\CR}^\bullet(Y_\pm)$ are well-defined.  Thus we have a well-defined quantum connection
\begin{equation} 
\label{eq:qconn_non_equivariant} 
\nabla = d + z^{-1} \sum_{i=0}^N (\phi_i \star_\tau) d\tau^i
\end{equation}
where $\star_\tau$ is the non-equivariant big quantum product, defined exactly as in~\eqref{eq:qprod_pushforward}.  This is a pencil $\nabla$ of flat connections on the trivial $H_{\CR}^\bullet(Y_\pm)$-bundle over an open set in $H_{\CR}^\bullet(Y_\pm)$; here, as in the equivariant case, $z \in \Cstar$ is the pencil variable, $\tau \in H_{\CR}^\bullet(Y_\pm)$ is the co-ordinate on the base of the bundle, $\phi_0,\dots,\phi_N$ are a basis for $H_{\CR}^\bullet(Y_\pm)$, and $\tau^0,\dots, \tau^N$ are the corresponding co-ordinates of $\tau \in H_{\CR}^\bullet(Y_\pm)$, so that $\tau = \sum_{i=0}^N \tau^i \phi_i$.  We consider now a similar structure on the \emph{ambient part} of $H_{\CR}^\bullet(Y_\pm)$, that is, on:
\[
H_{\amb}^\bullet(Y_\pm) := \im \iota^\star_\pm \subset H_{\CR}^\bullet(Y_\pm)
\]
where $\iota_\pm \colon Y_\pm \to X_\pm$ are the inclusion maps.  If $\tau \in H_{\amb}^\bullet(Y_\pm)$ then the big quantum product $\star_\tau$ preserves $H_{\amb}^\bullet(Y_\pm)$~\cite[Corollary~2.5]{Iritani:periods}, and so \eqref{eq:qconn_non_equivariant} restricts to give a well-defined quantum connection on the ambient part of $H_{\CR}^\bullet(Y_\pm)$.  The restriction of the fundamental solution $L_\pm(\tau,z)$ for \eqref{eq:qconn_non_equivariant}, defined exactly as in \eqref{eq:L_descendant}, gives a fundamental solution $L^\amb_\pm(\tau,z)$ for the quantum connection on the ambient part.

There is also an ambient part of $K^0(Y_\pm)$, given by $K^0_\amb(Y_\pm) := \im \iota_\pm^\star$, and an ambient $K$-group framing (cf.~Definition~\ref{def:K_framing})
\[
\frs \colon K_\amb^0(Y_\pm) \to 
H^\bullet_{\amb}(Y_\pm)\otimes \CC[\log z](\!(z^{-1/k})\!)[\![Q,\tau]\!]
\]
given by
\[
\frs(E)(\tau,z) = \frac{1}{(2\pi)^{\dim {Y_\pm}/2}} 
L^\amb_\pm(\tau,z) z^{-\mu} z^\rho \left( \hGamma_{Y_\pm} \cup 
(2\pi\tti)^{\frac{\deg_0}{2}} \inv^* \tch(E) \right) 
\]
where $\mu$ and $\rho$ are the grading operator and first Chern class for $Y_\pm$, $k \in \N$ is such that the eigenvalues of $k \mu$ are integers, and $\hGamma_{Y_\pm}$ is the non-equivariant $\hGamma$-class of $Y_\pm$.  As in~\S\ref{sec:integral_structure}, the image of $\frs$ is contained in the space of flat sections for the quantum connection on the ambient part of $H_{\CR}^\bullet(Y_\pm)$ which are homogeneous of degree zero.

\subsection{$I$-Functions for Toric Complete Intersections}

Recall from~\S\ref{sec:I_function} that the GIT data for $X_+$ determine a cohomology-valued hypergeometric function $I_+$.  The $I$-function $I_{X_+} := I_+$ is a multi-valued function of $\sfy_1,\ldots,\sfy_r$, depending analytically on $\sfy_r$ and formally on $\sfy_1,\ldots\sfy_{r-1}$, defined near the large-radius limit point $(\sfy_1,\ldots,\sfy_r) = (0,\ldots,0)$ in $\hcM_{\rm reg}$.  The GIT data for the total space of $E_+^\vee$ (regarded as a non-compact toric stack) is obtained from the GIT data for $X_+$ by adding extra toric divisors ${-E_1},\ldots,{-E_k}$.  It is easy to see that the corresponding $I$-function $I_{E_+^\vee}$ is also a multi-valued function of $\sfy_1,\ldots,\sfy_r$, depending analytically on $\sfy_r$ and formally on $\sfy_1,\ldots\sfy_{r-1}$, which is defined near the same large-radius limit point $(\sfy_1,\ldots,\sfy_r) = (0,\ldots,0)$ in $\hcM_{\rm reg}$.  The global quantum connections for $X_+$ and $E_+^\vee$ were constructed, in \S\ref{sec:global_qconn}, using the $I$-functions $I_{X_+}$ and $I_{E_+^\vee}$.  We now introduce a closely-related $I$-function, defined in terms of GIT data for $X_+$ and the characters $E_1,\ldots,E_k$, that will allow us to globalize the quantum connection on the ambient part of $H^\bullet_\CR(Y_\pm)$.

With notation as in~\S\ref{sec:I_function}, except with $u_i$ now denoting the non-equivariant class Poincar\'e-dual to the $i$th toric divisor \eqref{eq:T-invariant_divisor} and with $v_j \in H^2(X_+)$, $1 \leq j \leq k$, given by the non-equivariant first Chern class of the line bundle corresponding to the character $E_j$, define a $H_{\CR}^\bullet(X_+)$-valued hypergeometric series 
$I^{\rm temp}_{X_+,Y_+}(\sigma,x,z) 
\in H^\bullet_{\CR}(X_+) \otimes\CC(\!(z^{-1})\!)[\![Q,\sigma,x]\!]$ by:
\begin{multline*} 
 I^{\rm temp}_{X_+,Y_+}(\sigma,x,z)  = z e^{\sigma/z} 
\sum_{d\in \KK} e^{\sigma\cdot \overline{d}} 
Q^{\overline{d}} \prod_{j\in S} x_j^{D_j\cdot d}
\left( 
  \prod_{j=1}^{m}
  \frac{\prod_{a : \<a\>=\< D_j\cdot d \>, a \leq 0}(u_j+a z)}
  {\prod_{a : \<a \>=\< D_j \cdot d \>, a\leq D_j \cdot d} (u_j+a z)} 
\right) \\
\times 
\left(
  \prod_{j=1}^{k}
  \prod_{a=1}^{E_j \cdot d} (v_j+a z)
\right) 
\fun_{[{-d}]}
\end{multline*}
Note that for each $d \in \KK$ and each $j \in \{1,2,\ldots,k\}$, $E_j \cdot d$ is a non-negative integer.  (The subscript `temp' here again reflects the fact that  this notation for the $I$-function is only temporary: we are just about to change notation, by specializing certain parameters.)  Under our hypotheses \eqref{eq:conditions_on_line_bundles} on the line bundles $L_{X_+}(E_j)$, we have a Mirror Theorem for the toric complete intersection $Y_+$:

\begin{theorem}[\!\!\cite{CCIT:applications}] 
\label{thm:ci_mirror_thm} 
$\iota_+^\star I^{\rm temp}_{X_+,Y_+}(\sigma,x,-z)$ is an 
$\CC[\![Q,\sigma,x]\!]$-valued point on $\cL_{Y_+}$. 
\end{theorem} 

We define the $I$-function $I_{X_+,Y_+}$ to be the function obtained from $I^{\rm temp}_{X_+,Y_+}$ by the specialization $Q=1$, $\sigma = \sigma_+ := 
\theta_+( \sum_{i=1}^r \sfp^+_i \log \sfy_i)$ where $\theta_+$ is as in \eqref{eq:theta}.  Thus:
\[
I_{X_+,Y_+}(\sfy,z):=
z e^{\sigma_+/z} 
\sum_{d \in \KK_+} 
\sfy^d 
\left(
  \prod_{j=1}^{m}
  \frac{\prod_{a : \<a\>=\< D_j\cdot d \>, a \leq 0}(u_j+a z)}
  {\prod_{a : \<a \>=\< D_j \cdot d \>, a\leq D_j \cdot d} (u_j+a z)} 
\right) 
\left(
  \prod_{j=1}^{k}
  \prod_{a=1}^{E_j \cdot d} (v_j+a z)
\right)  
\fun_{[{-d}]}
\]
where $(\sfy_1,\dots,\sfy_r)$ are as in~\S\ref{sec:I_function}.  Repeating the analysis in Lemma~\ref{lem:I_analyticity} shows that $I_{X_+,Y_+}$, just like $I_{X_+}$ and $I_{E_+^\vee}$, is a multi-valued function of $\sfy_1,\ldots,\sfy_r$ that depends analytically on $\sfy_r$ and formally on $\sfy_1,\ldots\sfy_{r-1}$, defined near the large-radius limit point $(\sfy_1,\ldots,\sfy_r) = (0,\ldots,0)$ in $\hcM_{\rm reg}$.  The arguments in~\S\ref{sec:global_qconn} can now be applied verbatim to $I_{Y_+} := \iota_+^\star I_{X_+,Y_+}$, and thus we construct a global version of the quantum connection on the ambient part $H^\bullet_\amb(Y_+)$, defined over the base $\tcM_+^\circ$.  The analog of Theorem~\ref{thm:global_qconn} holds, with the same proof:

\begin{theorem} 
\label{thm:global_qconn_ci} 
There exist the following data: 
\begin{itemize} 
\item an open subset $\cU_+^\circ\subset \cU_+$ such that 
$P_+ \in \cU_+^\circ$ and that the complement 
$\cU_+ \setminus \cU_+^\circ$ is a discrete set; 
we write $\tcM_+^\circ = \tcM_+|_{\cU_+^\circ}$;  

\item a trivial $H_{\amb}^\bullet(Y_+)$-bundle 
$\bF^+$ over $\tcM_+^\circ(\CC[z])$: 
\[
\bF^+ = H_{\amb}^\bullet(Y_+)\otimes
\cO_{\cU^\circ_+}[z][\![\sfy_1,\dots,\sfy_{r-1}]\!];  
\]
\item a flat connection 
$\bnabla^+ = d + z^{-1} \bA^+(\sfy)$ 
on $\bF^+$ of the form: 
\[
\bA^+(\sfy) = \sum_{i=1}^{\ell_+} B_i(\sfy) \frac{dy_i}{y_i} 
+ \sum_{j\in S_+} C_j(\sfy) dx_j 
\]
with $B_i(\sfy), C_j(\sfy) 
\in 
\End(H_{\amb}^\bullet(Y_+)) \otimes
\cO_{\cU^\circ_+}[\![\sfy_1,\dots,\sfy_{r-1}]\!]$; 
\item a vector field $\bE^+$ on $\tcM_+$, called the Euler vector field, defined by:
\[
\bE^+ = \sum_{i=1}^r \frac{1}{2}(\deg \sfy_i) \sfy_i 
\parfrac{}{\sfy_i}; 
\]
\item a mirror map $\tau_+ \colon \tcM_+ \to H_{\amb}^\bullet(Y_+)$ 
of the form:
\begin{align*}
  \tau_+ = \iota_+^\star \sigma_+ + \ttau_+ &&&
  \ttau_+\in H^\bullet_{\amb}(Y_+)\otimes\cO_{\cU^\circ_+}[\![\sfy_1,\dots,\sfy_{r-1}]\!] \\
  &&& \ttau_+|_{\sfy_1=\cdots = \sfy_r=0}= 0
\end{align*}
\end{itemize} 
such that $\bnabla^+$ equals the pull-back $\tau_+^*\nabla^+$ of the 
(non-equivariant) quantum connection $\nabla^+$ on the ambient part of $H^\bullet_{\CR}(Y_+)$ by $\tau_+$, 
that is:
\begin{align*} 
B_i(\sfy) & = \sum_{k=0}^N \parfrac{\tau_+^k(\sfy)}
{\log y_i} (\phi_k \star_{\tau_+(\sfy)}) && 
1\le i\le \ell_+ \\
C_j(\sfy) & = 
\sum_{k=0}^N \parfrac{\ttau_+^k(\sfy)} 
{x_j} (\phi_k\star_{\tau_+(\sfy)}) &&
j\in S_+ 
\end{align*} 
and that the push-forward of $\bE^+$ by $\tau_+$ 
is the (non-equivariant) Euler vector field $\cE^+$  on the ambient part $H^\bullet_{\amb}(Y_+)$.
Moreover, there exists a global section $\Upsilon^+_0(\sfy,z)$ of 
$\bF^+$ such that 
\[
I_{Y_+}(\sfy,z) = z L_+^\amb(\tau_+(\sfy),z)^{-1} \Upsilon^+_0(\sfy,z) 
\]
where $L_+^\amb(\tau,z)$ is the ambient fundamental solution from \S\ref{sec:ambient_part}
\end{theorem} 

\begin{remark}
  Here, as in Theorem~\ref{thm:global_qconn}, the Novikov variable $Q$ has been specialized to $1$.
\end{remark}

\begin{remark}
  Entirely parallel results hold for $Y_-$.
\end{remark}
\subsection{Analytic Continuation of $I$-Functions}
\label{sec:ci_analytic_continuation}

To prove the Crepant Transformation Conjecture in this context, we need to establish the analog of Theorem~\ref{thm:U}.  To do this, we will compare the analytic continuation of the $I$-functions $I_{X_\pm,Y_\pm}$ with the analytic continuation of $I_{E_\pm^\vee}$.  Let $T = (\Cstar)^m$ denote the torus acting on $X_\pm$, and $\widetilde{T} = (\Cstar)^{m+k}$ denote the torus acting on $E_\pm^\vee$.  The splitting $\widetilde{T} = T \times (\Cstar)^k$ gives $R_{\widetilde{T}} = R_T [\kappa_1,\ldots,\kappa_k]$ where $\kappa_j$, $1 \leq j \leq k$, is the character of $(\Cstar)^k$ given by projection to the $j$th factor of the product $(\Cstar)^k$.  We regard $\widetilde{T}$ as acting on $X_\pm$ via the given action of $T \subset \widetilde{T}$ and the trivial action of $(\Cstar)^k \subset \widetilde{T}$, so that:
\begin{align*}
  \ZZ[\widetilde{T}] = \ZZ[T] [e^{\pm\kappa_1},\ldots,e^{\pm\kappa_k}] 
&& \text{and} &&
  K^0_{\widetilde{T}}(X_\pm) 
= K^0_T(X_\pm) \otimes_{\ZZ[T]} \ZZ[\widetilde{T}]
\end{align*}

\begin{lemma}
  \label{lem:FM_E_X}
  The Fourier--Mukai transformations
  \begin{align*}
    \FM \colon K^0(X_-) \to K^0(X_+) &&
    \FM \colon K^0(E_-^\vee) \to K^0(E_+^\vee) 
  \end{align*}
  coincide under the natural identification of $K^0(X_\pm)$ with $K^0(E_\pm^\vee)$.  The same statement holds equivariantly.
\end{lemma}

\begin{proof}
  Consider the fiber diagram:
  \[
  \xymatrix{
     & \widetilde{E}^\vee \ar[ld] \ar[rd] \ar[d] \\
      E_-^\vee\ar[d]  &\widetilde{X} \ar[ld]^-{f_-} \ar[rd]_-{f_+} & E_+^\vee \ar[d]\\
      X_-  && X_+
  }
  \]
  where the bottom triangle is \eqref{eq:common_blowup}  and the top triangle is the analog of \eqref{eq:common_blowup}  for $E_\pm^\vee$, 
and apply the flat base change theorem.
\end{proof}

Let $\UU_{E^\vee}$  be the symplectic transformation from Theorem~\ref{thm:U} applied to $E_\pm^\vee$.
Combining Lemma~\ref{lem:FM_E_X} with Theorem~\ref{thm:U} gives a commutative diagram:
\begin{equation}
  \label{eq:U_is_independent_of_kappa_1}
  \begin{aligned}
    \xymatrix{
      K^0_{\widetilde{T}}(X_-) \ar@{=}[d] \ar[r]^{\FM} & K^0_{\widetilde{T}}(X_+) \ar@{=}[d] \\
      K^0_{\widetilde{T}}(E_-^\vee) \ar[d]_{z^{-\mu_-} z^{\rho_-} \hGamma_{E_-^\vee} \cup (2\pi\tti)^{\frac{\deg_0}{2}} \inv^* \tch({-})} \ar[r]^{\FM} & 
      K^0_{\widetilde{T}}(E_+^\vee) \ar[d]^{z^{-\mu_+} z^{\rho_+} \hGamma_{E_+^\vee} \cup (2\pi\tti)^{\frac{\deg_0}{2}} \inv^* \tch({-})} \\
      \tcH(E_-^\vee) \ar[r]^{\UU_{E^\vee}} & \tcH(E_+^\vee)
    }
  \end{aligned}
\end{equation}
where $\rho_\pm \in H^2_{\widetilde{T}}(E_\pm^\vee)$ is the $\widetilde{T}$-equivariant first Chern class of $E_\pm^\vee$ and $\mu_\pm$ are the $\widetilde{T}$-equivariant grading operators.  Recall that
\begin{align*}
  \Gamma_{E_\pm^\vee} = \Gamma_{X_\pm} \Gamma(E_\pm^\vee)
  && \rho_\pm = \rho_{X_\pm} + c_1^{\widetilde{T}} (E_\pm^\vee)
\end{align*}
and that the Chern roots of $E_\pm^\vee$ are pulled back from the common blow-down $\overline{X}_0$ of $X_\pm$.  Part~(2) of Theorem~\ref{thm:U} thus implies that we can factor out the contributions of  $\Gamma(E_\pm^\vee)$ and $c_1^{\widetilde{T}} (E_\pm^\vee)$ from the vertical maps in \eqref{eq:U_is_independent_of_kappa_1}, replacing the vertical arrows by:
\[
z^{-\mu_{X_\pm}} z^{\rho_{X_\pm}} \hGamma_{X_\pm} \cup (2\pi\tti)^{\frac{\deg_0}{2}} \inv^* \tch({-})
\]
This proves:
\begin{lemma}
  The transformations $\UU_X \colon \cH(X_-) \to \cH(X_+)$ and $\UU_{E^\vee} \colon \cH(E_-^\vee) \to \cH(E_+^\vee)$ coincide under the natural identifications of $\cH(X_\pm)$ with $\cH(E_\pm^\vee)$.  In particular, $\UU_{E^\vee}$ is independent of $\kappa_1,\ldots,\kappa_k$.
\end{lemma}

The $I$-functions $I_{X_+,Y_+}$ and $I_{E_+^\vee}$ are related\footnote{An analogous relationship holds between $I_{X_-,Y_-}$ and $I_{E_-^\vee}$.} by:
\[
 I_{E_+^\vee}(\sfy)\Big|_{\lambda=0,\kappa={-z}} = 
e^{\pi \tti c_1(E_+^\vee)/z} I_{X_+,Y_+}(\pm\sfy) 
\]
where the subscript on the left-hand side denotes the specialization:
\begin{equation}
  \label{eq:specialization}
  \begin{cases}
    \lambda_i=0 & 1 \leq i \leq m \\
    \kappa_j = {-z} & 1 \leq j \leq k
  \end{cases}
\end{equation}
and the $\pm$ on the right-hand side denotes the change of variables:
\begin{align}
  \label{eq:pm}
  &\log \sfy_i \mapsto \log \sfy_i - \pi \tti \sum_{j=1}^k l_{ij} && 1 \leq i \leq r    
  && \text{with} && E_j =\sum_{i=1}^r l_{ij} \sfp_i
\end{align}
The specialization \eqref{eq:specialization} is given by a shift $\bbS \colon \kappa_j \mapsto \kappa_j - z$ in the equivariant parameters followed by passing to the non-equivariant limit.  Note that the change of variables \eqref{eq:pm} maps $\sfy^d$ to $(-1)^{{-c_1}(E_+^\vee) \cdot d} \sfy^d$.

Recall from Theorem~\ref{thm:U} that, after analytic continuation, we have $I_{E_+^\vee} = \UU_{E^\vee} I_{E_-^\vee}$.  Since $\UU_{E^\vee}$ is independent of $\kappa_j$,~$1 \leq j \leq k$, it follows that $\UU_{E^\vee}$ commutes with the shift $\bbS$.  Since the Chern roots of $E^\vee$ are pulled back from the common blow-down $\overline{X}_0$ of $X_\pm$, it follows that{
\[
\UU_{E^\vee} \, e^{\pi \tti c_1(E_-^\vee)/z} = e^{\pi \tti c_1(E_+^\vee)/z} \, \UU_{E^\vee}
\]
Setting $\lambda = 0$ and $\kappa_j = {-z}$ in the equality $I_{E_+^\vee} = \UU_{E^\vee} I_{E_-^\vee}$, and replacing $\cH(E_\pm^\vee)$ and $\UU_{E^\vee}$ with their non-equivariant limits
\begin{align*}
  \cH(E_\pm^\vee) := H^\bullet_\CR(E_\pm^\vee) \otimes \CC(\!(z^{-1})\!) && \text{and} &&
  \UU_{E^\vee} \colon \cH(E_-^\vee) \to \cH(E_+^\vee)
\end{align*}
we find that
\[
 I_{X_+,Y_+} = \UU_{E^\vee} I_{X_-,Y_-}
\]
after analytic continuation.  Thus:
\[
 I_{X_+,Y_+} = \UU_{X} I_{X_-,Y_-}
\]
after analytic continuation.

\subsection{Compatibility of Fourier--Mukai Transformations}
\label{sec:FM_compatibility}

For the analogue of part~(3) of Theorem~\ref{thm:U}, we need to compare the Fourier--Mukai transformation associated to $X_+ \dashrightarrow X_-$ with the Fourier--Mukai transformation associated to $Y_+ \dashrightarrow Y_-$.  This is a base change question (cf.~Lemma~\ref{lem:FM_E_X}), but this time we do not have flatness.  By assumption, we have:
\begin{equation}
  \label{eq:FM_Y_setup}
  \begin{aligned}
    \xymatrix{
      & \widetilde{Y} \ar[ld]_-{F_-} \ar[rd]^-{F_+} \ar[d]_-{\tilde{\iota}} \\
      Y_-\ar[d]_{\iota_-}  &\widetilde{X} \ar[ld]^-{f_-} \ar[rd]_-{f_+} & Y_+ \ar[d]^{\iota_+}\\
      X_-  && X_+
    }
  \end{aligned}
\end{equation}
where the vertical maps are inclusions, the bottom triangle is \eqref{eq:common_blowup}  and the top triangle is the analog of \eqref{eq:common_blowup}  for $Y_\pm$.  The substacks $\tilde{Y}$ is defined by the vanishing of a section $\tilde{s} \colon \widetilde{X} \to \widetilde{E}$, where $\widetilde{E} \to \widetilde{X}$ is the direct sum of line bundles
\[
\widetilde{E} := \bigoplus_{i=1}^k L_{\widetilde{X}}(E_i)
\]
The line bundles $E_-$,~$\widetilde{E}$, and~$E_+$ are all canonically identified via $f_-^\star$ and $f_+^\star$, since they are all pulled back from the common blow-down $\overline{X}_0$ of $X_\pm$.  The section $\widetilde{s}$ coincides both with the pullback of the section $s_+$ via $f_+$ and with the pullback of the section $s_-$ via $f_-$.  Since the zero loci of $s_\pm$ are assumed to intersect the flopping locus transversely, $\tilde{s}$ is a regular section of $\widetilde{E}$ and the substack $\widetilde{Y} \subset \widetilde{X}$ is smooth.

\begin{lemma}
  \label{lem:FM_Y_X}
  The following diagram commutes:
  \begin{equation}
    \label{eq:FM_Y_commutes}
    \begin{aligned}
      \xymatrix{
        K^0(X_-) \ar[r]^{\FM} \ar[d]_{\iota_-^\star} & K^0(X_+) \ar[d]^{\iota_+^\star} \\
        K^0(Y_-)_\amb \ar[r]^{\FM} & K^0(Y_+)_\amb
      }
    \end{aligned}
  \end{equation}
  where the top horizontal arrow is the Fourier--Mukai transformation $(f_+)_\star (f_-)^\star$ from~\eqref{eq:FM_Y_setup}, and the bottom horizontal arrow is the Fourier--Mukai transformation $(F_+)_\star (F_-)^\star$ from~\eqref{eq:FM_Y_setup}.
\end{lemma}

\begin{proof}
  The pullback along $f_+$ of the Koszul resolution of $\cO_{Y_+}$ in $X_+$ gives the Koszul resolution of $\cO_{\widetilde{Y}}$ in $\widetilde{X}$.  This implies that, in the right-hand square in \eqref{eq:FM_Y_setup}, $\widetilde{X}$ and $Y_+$ are Tor-independent over $X_+$~\cite[Tag 08IA]{stacks-project}.  Tor-independent base-change~\cite[Tag 08IB]{stacks-project} now implies that:
\[
(F_+)_\star \circ \tilde{\iota}^\star = \iota_+^\star \circ (f_+)_\star
\]
Since $F_-^\star \circ \iota_-^\star = \tilde{\iota}^\star \circ f_-^\star$, it follows that
\[
(\iota_+)^\star (f_+)_\star (f_-)^\star = (F_+)_\star (F_-)^\star (\iota_-)_\star
\]
which is the result.
\end{proof}

\begin{remark}
  This argument in fact proves that the analog of diagram~\eqref{eq:FM_Y_commutes} for derived categories is commutative, but we only need the statement at the level of $K$-theory.
\end{remark}

\subsection{Completing the Proof}  Denote by $\UU_X$ the transformation from the non-equivariant version of Theorem~\ref{thm:U} applied to $X_\pm$.  This is a map $\UU_X \colon \cH(X_-) \to \cH(X_+)$
between the non-equivariant Givental spaces for $X_\pm$:
\[
\cH(X_\pm) := H^\bullet_\CR(X_\pm) \otimes \CC(\!(z^{-1})\!)
\]
Let us remark again that the Chern roots of $E_\pm$ are pulled back from the common blow-down $\overline{X}_0$ of $X_\pm$; the second part of Theorem~\ref{thm:U} therefore gives:
\begin{equation}
  \label{eq:intertwiner_U_X}
  \UU_X \hGamma(E_-)  = \hGamma(E_+) \UU_X
\end{equation}
The results from \S\ref{sec:ci_analytic_continuation} and \S\ref{sec:FM_compatibility} combine to give a commutative diagram:
\[
\xymatrix{
  K^0(X_-) \ar[rr]^{\FM} \ar[rdd] \ar[ddd]_{\iota_-^\star} && K^0(X_+) \ar[rdd] \ar'[dd]_-{\iota_+^\star}[ddd]\\
  \\
  & \tcH(X_-) \ar[rr]^<<<<<<<<<{\UU_X} \ar[ddd]_{\iota_-^\star} && \tcH(X_+) \ar[ddd]_{\iota_+^\star}\\
  K^0(Y_-)_\amb \ar'[r]^-{\FM}[rr] \ar[rdd] && K^0(Y_+)_\amb \ar[rdd] \\
  \\
  & \tcH(Y_-)_\amb \ar@{..>}[rr] && \tcH(Y_+)_\amb \\
}
\]
where $\tcH(Y_\pm)_\amb$ is the ambient part of the multi-valued Givental space:
\begin{equation}
  \label{eq:multi-valued_ambient}
  \tcH(Y_\pm)_\amb := H_{\amb}^\bullet(Y_\pm) \otimes \CC[\log z](\!(z^{-1/k})\!)
\end{equation}
with $k$ as in the statement of Theorem~\ref{thm:U}, and:
\begin{itemize}
\item the top diagonal maps are the $K$-theory framing maps from Definition~\ref{def:K_framing} but with $\hGamma_{X_\pm}$ replaced by $\hGamma_{X_\pm,Y_\pm} := \hGamma_{X_\pm} \hGamma(E_\pm)^{-1}$;
\item the bottom diagonal maps are the ambient $K$-group framing maps from \S\ref{sec:ambient_part}.
\end{itemize}
Here:
\begin{itemize}
\item the top face is commutative, by Theorem~\ref{thm:U} and \eqref{eq:intertwiner_U_X};
\item the back face is commutative, by Lemma~\ref{lem:FM_Y_X};
\item the sides are commutative, by the definition of the framing maps;
\end{itemize}
and we want to define the dotted arrow so that all faces commute.  Define $\UU_Y \colon \tcH(Y_-)_\amb \to \tcH(Y_+)_\amb$ to be the unique map such that the bottom face commutes.  Chasing diagrams shows that the front face commutes also.  Since $I_{X_+,Y_+} = \UU_X I_{X_-,Y_-}$ after analytic continuation and since $I_{Y_\pm} := \iota_\pm^\star I_{X_\pm,Y_\pm}$, we conclude that $I_{Y_+} = \UU_Y I_{Y_-}$ after analytic continuation.

\begin{theorem} 
\label{thm:U_ci} 
Consider the ambient part 
of the (non-equivariant) Givental space for $Y_\pm$ with the Novikov variable $Q$ specialized to~$1$:
\[
\cH(Y_\pm)_\amb = H^\bullet_{\amb}(Y_\pm) \otimes \CC(\!(z^{-1})\!)
\]
Regard $\cH(Y_\pm)_\amb$ as a graded vector space, where we use the age-shifted grading on $H_{\amb}^\bullet(Y_\pm)$ and set  $\deg z=2$.
There exists a degree-preserving
$\CC(\!(z^{-1})\!)$-linear transformation 
\[
\UU_Y \colon \cH(Y_-)_\amb \to \cH(Y_+)_\amb
\]
such that:
\begin{enumerate}
\item $I_{Y_+}(\sfy,z) = \UU_Y I_{Y_-}(\sfy,z)$ after analytic continuation in $\sfy^e$
along the path $\gamma$ in Figure {\rm\ref{fig:path_ancont}}; 
\item $\UU_Y \circ (g_-^\star v\cup)= (g_+^\star v\cup) \circ \UU_Y$ for all 
$v\in H^2(\overline{X}_0)$, where  $\overline{X}_0$ is the common blow-down of $X_\pm$ and $g_\pm \colon Y_\pm \to 
\overline{X}_0$ is the composition of the inclusion $\iota_\pm \colon Y_\pm \to X_\pm$ with the blow-down $X_\pm \to \overline{X}_0$;
\item there is a commutative diagram
\[
\xymatrix{
  K^0(Y_-)_\amb \ar[r]^{\FM} \ar[d] & K^0(Y_+)_\amb \ar[d]\\ 
  \tcH(Y_-)_\amb \ar[r]^{\UU_Y} & \tcH(Y_+)_\amb
}
\]
where $\FM$ is the Fourier--Mukai transformation given by the top triangle in \eqref{eq:FM_Y_setup} and the vertical arrows are the ambient $K$-group framing defined in \S\ref{sec:ambient_part}.
\end{enumerate} 
If $Y_\pm$ is compact then $\UU_Y$ intertwines the (possibly-degenerate) symplectic pairings on $\cH(Y_\pm)_\amb$. 
\end{theorem} 

\begin{proof}
  Everything has been proved except the statement that, if $Y_\pm$ is compact, then $\UU_Y$ intertwines the pairings on $\cH(Y_\pm)_\amb$.   But:
\begin{align*}
  \Big( \UU_Y(-z) \iota_-^\star \alpha, \UU_Y(z) \iota_-^\star \beta \Big)_{Y_+}
  &= 
  \Big( \iota_+^\star \UU_X(-z)  \alpha, \iota_+^\star \UU_X(z) \beta \Big)_{Y_+} \\
  &= \Big( \UU_X(-z)  \alpha, e(E_+) \UU_X(z) \beta \Big)_{X_+} \\
  &= \Big( \UU_X(-z)  \alpha, \UU_X(z) e(E_-) \beta \Big)_{X_+} && \text{by Theorem~\ref{thm:U}(2)}\\ 
  &= \Big( \alpha, e(E_-) \beta \Big)_{X_-} && \text{since $\UU_X$ is pairing-preserving} \\ 
  &= \Big( \iota_-^\star \alpha, \iota_-^\star \beta \Big)_{Y_-}
\end{align*}
\end{proof}
\begin{remark}
  \label{rem:degeneracy}
  If $Y_\pm$ is compact then the Givental space for $Y_\pm$ has a well-defined symplectic pairing, but the restriction of this pairing to the ambient part is non-degenerate if and only if $(\iota_\pm)_\star \colon H^\bullet_{\amb}(Y_\pm) \to H^\bullet_{\CR}(X_\pm)$ is injective.  This would hold by the Lefschetz Theorem if $E_\pm$ were a direct sum of ample line bundles, but in our situation the line bundles are always  semi-ample and the question is more subtle.  Injectivity holds when $Y_\pm$ is a regular semiample hypersurface by a result of Mavlyutov~\cite[Theorem~5.1]{Mavlyutov}.
\end{remark}

Theorem~\ref{thm:U_ci} is the analog, for toric complete intersections, of Theorem~\ref{thm:U}.  The analog of Theorem~\ref{thm:CTC_qconn} also holds: 

\begin{theorem} 
\label{thm:CTC_qconn_ci}
Let $(\bF^\pm, \bnabla^\pm,\bE^\pm)$ be the global 
quantum connections for the ambient parts $H^\bullet_{\amb}(Y_\pm)$ over $\tcM_\pm^\circ(\CC[z])$ 
from Theorem \ref{thm:global_qconn_ci}. 
We have that $\bE^+ = \bE^-$ on $\tcM$. 
There exists a gauge transformation 
\[
\Theta_Y \in \Hom\big(H^\bullet_{\amb}(Y_-), H^\bullet_{\amb}(Y_+)\big) 
\otimes \cO_{\cU^\circ}[z][\![\sfy_1,\dots,\sfy_{r-1}]\!] 
\]
over $\tcM^\circ(\CC[z])$ such that:
\begin{itemize} 
\item $\bnabla^-$ and $\bnabla^+$ are gauge-equivalent 
via $\Theta_Y$, i.e.~$
\bnabla^+ \circ \Theta_Y = \Theta_Y \circ \bnabla^-$; 

\item $\Theta_Y$ is homogeneous of degree zero, i.e.~$
\bGr^+ \circ \Theta_Y = \Theta_Y \circ \bGr^-$ with 
$\bGr^\pm := z\parfrac{}{z} + \bE^\pm + \mu^\pm$; 

\item if $Y_\pm$ are compact then $\Theta_Y$ preserves the (possibly-degenerate) orbifold Poincar\'{e} pairing 
on $H^\bullet_{\amb}(Y_\pm)$, i.e.~$(\Theta_Y(\sfy,-z) \alpha, \Theta_Y(\sfy,z) \beta) = (\alpha,\beta)$.
\end{itemize} 
Moreover, the analytic continuation of flat sections coincides, via the ambient $K$-group framing defined in~\S\ref{sec:ambient_part}, with the Fourier--Mukai transformation: 
\begin{align*}
  \Theta_Y\Bigl( \frs(E)(\tau_-(\sfy),z) \Bigr) = \frs(\FM(E))(\tau_+(\sfy),z) &&
  \text{for all $E \in K^0(Y_-)_\amb$}
\end{align*}
where $\tau_\pm$ are the mirror maps from Theorem~\ref{thm:global_qconn_ci}.
\end{theorem} 

Theorem~\ref{thm:CTC_qconn_ci} follows from Theorem~\ref{thm:U_ci} exactly as Theorem~\ref{thm:CTC_qconn} follows from Theorem~\ref{thm:U}.
The  transformation $\UU_Y$ in Theorem~\ref{thm:U_ci} and the gauge transformation 
$\Theta_Y$ in Theorem~\ref{thm:CTC_qconn_ci} are related by 
\[
L_+^\amb(\tau_+(\sfy),z)^{-1} \circ \Theta_Y = 
\UU \circ L_-^\amb(\tau_-(\sfy),z)^{-1} 
\]
where $L_\pm$ is the ambient fundamental solution from~\S\ref{sec:ambient_part}.
The gauge transformation $\Theta_Y$ sends 
the section $\Upsilon_0^-\in \bF^-$ to the section 
$\Upsilon_0^+\in \bF^+$, 
where $\Upsilon_0^\pm$ are as 
in Theorem \ref{thm:global_qconn_ci}.

\section*{Acknowledgements}
We thank Alessio Corti, Alexander Kasprzyk, Yuan-Pin Lee, Nathan Priddis and Mark Shoemaker for a number 
of useful conversations. 
This research was supported by a Royal Society University Research Fellowship; the Leverhulme Trust; ERC Starting Investigator Grant number~240123; EPSRC Mathematics Platform grant EP/I019111/1; JSPS Kakenhi Grant Number 25400069; NFGRF, University of Kansas; and Simons Foundation Collaboration Grant 311837.

\bibliographystyle{plain}
\bibliography{bibliography}

\end{document}